\documentclass[11pt]{amsart}
\usepackage[utf8x]{inputenc}
\usepackage[english]{babel}
\usepackage{microtype}

\usepackage{amsmath,amssymb,mathrsfs,amsthm,amsfonts,mathtools}
\usepackage[inline]{enumitem} % change enum. items
\usepackage[dvipsnames]{xcolor}
\usepackage{hyperref}
\usepackage{tikz}
\usepackage{bbm}
\usepackage[english]{babel}
\usepackage{times}
\tikzstyle{block} = [rectangle, draw, 
   ]
\usetikzlibrary{shapes,positioning}
\usetikzlibrary{decorations.pathreplacing}

\usepackage{comment}
\hypersetup{%
  colorlinks=true, linkcolor=blue,
  citecolor=Green
}
\usepackage[paper=letterpaper,margin=1in]{geometry}
\usepackage{acronym}
\usepackage{dsfont}
\usepackage[dvipsnames]{xcolor}

\DeclareMathOperator\Var{Var}
\DeclareMathOperator{\sgn}{sgn}

\newcommand{\ve}{\varepsilon}

\newcommand{\1}{\mathds{1}}

\newcommand\dP{\mathds{P}}

\newcommand\dE{\mathds{E}}

\newcommand\dS{\mathds{S}}

\newcommand\dX{\mathds{X}}

\newcommand\bR{\mathbb{R}}

\newcommand\bZ{\mathbb{Z}}

\newcommand\cA{\mathcal{A}}
\newcommand\cB{\mathcal{B}}
\newcommand\cC{\mathcal{C}}
\newcommand\cD{\mathcal{D}}
\newcommand\cE{\mathcal{E}}
\newcommand\cF{\mathcal{F}}
\newcommand\cG{\mathcal{G}}
\newcommand\cH{\mathcal{H}}
\newcommand\cI{\mathcal{I}}
\newcommand\cJ{\mathcal{J}}
\newcommand\cK{\mathcal{K}}
\newcommand\cL{\mathcal{L}}
\newcommand\cN{\mathcal{N}}

\newcommand\cR{\mathcal{R}}
\newcommand\cS{\mathcal{S}}
\newcommand\cM{\mathcal{M}}
\newcommand\cT{\mathcal{T}}

\newcommand\cQ{\mathcal{Q}}

\newcommand\cU{\mathcal{U}}

\newcommand\cZ{\mathcal{Z}}
\newcommand\fm{\mathfrak{m}}

\newcommand\fB{\mathfrak{B}}
\newcommand\fS{\mathfrak{S}}

\newcommand\fR{\mathfrak{R}}
\newcommand\fT{\mathfrak{T}}
\newcommand\fL{\mathfrak{L}}

\newtheorem{stat}{Statement}[section]
\newtheorem{proposition}[stat]{Proposition}
\newtheorem{corollary}[stat]{Corollary}

\newtheorem{theorem}[stat]{Theorem}
\newtheorem{lemma}[stat]{Lemma}
\theoremstyle{definition}

\newtheorem{remark}[stat]{Remark}

\numberwithin{equation}{section}

\makeatletter
\renewcommand{\tocsection}[3]{%
  \indentlabel{\@ifnotempty{#2}{\bfseries\ignorespaces#1 #2\quad}}\bfseries#3}
\renewcommand{\tocsubsection}[3]{%
  \indentlabel{\@ifnotempty{#2}{\ignorespaces#1 #2\quad}}#3}
\newcommand\@dotsep{4.5}
\def\@tocline#1#2#3#4#5#6#7{\relax
  \ifnum #1>\c@tocdepth % then omit
  \else
    \par \addpenalty\@secpenalty\addvspace{#2}%
    \begingroup \hyphenpenalty\@M
    \@ifempty{#4}{%
      \@tempdima\csname r@tocindent\number#1\endcsname\relax
    }{%
      \@tempdima#4\relax
    }%
    \parindent\z@ \leftskip#3\relax \advance\leftskip\@tempdima\relax
    \rightskip\@pnumwidth plus1em \parfillskip-\@pnumwidth
    #5\leavevmode\hskip-\@tempdima{#6}\nobreak
    \leaders\hbox{$\m@th\mkern \@dotsep mu\hbox{.}\mkern \@dotsep mu$}\hfill
    \nobreak
    \hbox to\@pnumwidth{\@tocpagenum{\ifnum#1=1\bfseries\fi#7}}\par% <-- \bfseries for \section page
    \nobreak
    \endgroup
  \fi}
\AtBeginDocument{%
\expandafter\renewcommand\csname r@tocindent0\endcsname{0pt}
}
\def\l@subsection{\@tocline{2}{0pt}{2.5pc}{5pc}{}}
\makeatother

\begin{document}%\onehalfspacing

\title[Limit theorems for the one-dimensional parabolic Anderson model with white noise potential]{Limit theorems for the one-dimensional parabolic Anderson model with white noise potential}

\author[K.\ Kim]{Kunwoo Kim }
\address{K.\ Kim,
  Pohang University of Science and Technology (POSTECH), South Korea
  }
\email{kunwoo@postech.ac.kr}

\author[U.\ Kim]{Uijun Kim }
\address{U.\ Kim,
  Pohang University of Science and Technology (POSTECH), South Korea
  }
\email{ujkim@postech.ac.kr}

\author[J.\ Yi]{Jaeyun Yi }
\address{J.\ Yi,
  Korea Institute for Advanced Study (KIAS), South Korea
  }
\email{jaeyun@kias.re.kr}

\date{\today}

\begin{abstract}
We consider the parabolic Anderson model
$\partial_{t}u=\partial_{x}^{2}u+\xi u$ on
$\bR_{+}\times\bR$ with $u(0,\cdot)\equiv1$, where $\xi$ is a
spatial white noise. We study the long-time behavior
of the spatial integral
$U(t):=\int_{-L(t)/2}^{L(t)/2}u(t,x)\,dx$, where
$L(t)=\exp(\alpha^{3}t^{3}/24)$ with $\alpha>0$.
We establish a weak law of large
numbers for $\alpha>1$ and a central limit theorem for
$\alpha>2$. Moreover, for every $\alpha\in(0,2)$,
we show that
$U(t)$ converges in distribution, after explicit
centering and scaling, to a totally asymmetric
$\alpha$-stable law. To
the best of our knowledge, these are the first stable
limit laws for the
continuous parabolic Anderson model. In addition, we
establish two
spectral results of independent interest for the
one-dimensional Anderson Hamiltonian
on a growing interval. They are the main ingredients
of the proofs.
The first gives the lower-tail asymptotics of
the lowest eigenvalue, including the exact prefactor.
The second shows that, on this lower-tail event, the $L^{1}$
norm of the corresponding eigenfunction concentrates around
an explicit deterministic value.

\noindent{\it Keywords:} parabolic Anderson model, Anderson Hamiltonian, weak law
of large numbers, central limit theorem, stable limit laws

\noindent{\it MSC 2020:}
Primary 60H15, 60F05; secondary 35R60, 60K37.

\end{abstract}

\maketitle

\tableofcontents

\section{Introduction and main results}\label{sec:intro}
 
\subsection{The problem}\label{subsec:the_problem}
 
We consider the parabolic Anderson model (PAM) in one dimension
\begin{equation}\label{eq:PAM}
    \begin{cases}
        \partial_{t}u(t,x) = \partial_{x}^{2} u(t,x)+\xi(x)\,u(t,x)\,,
        & t>0,\ x\in\bR\,,\\
        u(0,x) = 1\,, & x\in\bR\,,
    \end{cases}
\end{equation}
where $\xi$ is a spatial white noise on $\bR$. More precisely,
$\xi$ is a centered Gaussian process indexed by
$C_{0}^{\infty}(\bR)$ with covariance
\begin{equation*}
    \dE\bigl[\xi(f)\,\xi(g)\bigr]
    = \int_{\bR} f(x)\,g(x)\,dx\,,
    \qquad f,g\in C_{0}^{\infty}(\bR)\,,
\end{equation*}
which is formally written as
$\dE[\xi(x)\xi(y)]=\delta_{0}(x-y)$.
Since $\xi$ is not defined pointwise, the classical
Feynman--Kac formula does not apply directly. Let $j$ be a
smooth symmetric probability density supported on
$[-1,1]$ whose Fourier transform is nonnegative, and
set $j_{\ve}(x):=\ve^{-1}j(\ve^{-1}x)$ for $\ve>0$.
Then the mollified noise
$\xi_{\ve}(x):=\xi\bigl(j_{\ve}(x-\cdot)\bigr)$
is a smooth Gaussian field on $\bR$.
Hence, \eqref{eq:PAM}
with $\xi$ replaced by $\xi_{\ve}$ admits the classical
Feynman--Kac representation
\begin{equation*}
    u_{\ve}(t,x)
    = \dE_{x}\biggl[\exp\biggl(\int_{0}^{t}\xi_{\ve}(W_{s})\,ds
    \biggr)\biggr]\,.
\end{equation*}
Here, $(W_{t})_{t\ge0}$ is a Brownian motion on $\bR$ with
generator $\partial_{x}^{2}$, independent of $\xi$. We
write $\dP_{x}$ and $\dE_{x}$ for the law of $W$ started from
$x$ and the corresponding expectation, and $\dP$ and $\dE$ for
the law of $\xi$ and its expectation. 
By \cite[Lemma~A.1]{Chen14}, the limit
\begin{equation*}
    \int_{0}^{t}\xi(W_s)\,ds
    := \lim_{\ve\downarrow0}
    \int_{0}^{t}\xi_\ve(W_s)\,ds
\end{equation*}
exists in $L^{2}(\dP\otimes\dP_x)$ for every $t\ge0$.
The limit is independent of the mollifier, including
the Gaussian mollifiers used in \cite{HHNT15,GLGL23}
(see the proof of \cite[Lemma~A.1]{Chen14}
and \cite[Remark~1.6]{GLGL23}).
In addition, it has finite exponential moments
\cite[(3.3)]{Chen14}.
We thus define
\begin{equation}\label{eq:FK_solution}
    u(t,x)
    = \dE_{x}\biggl[\exp\biggl(\int_{0}^{t}\xi(W_{s})\,ds\biggr)
    \biggr]\,.
\end{equation}
By \cite[Theorem~5.7]{HHNT15}, the random field $u$ defined by \eqref{eq:FK_solution}
 is a mild solution to \eqref{eq:PAM}, where the product $\xi u$
is understood as a Stratonovich integral. Moreover, for every fixed $t\ge0$, $x\in\bR$, and
$p\in[1,\infty)$,
\begin{equation*}
    u_\ve(t,x)\longrightarrow u(t,x)
    \quad\text{in }L^p(\dP)
    \quad\text{as }\ve\downarrow0,
\end{equation*}
where $u_\ve$ solves \eqref{eq:PAM} with $\xi$
replaced by $\xi_\ve$
(see \cite[Section~4.1.2]{HHNT15}
and \cite[(1.3)]{GLGL23}).
This identifies $u$ with the Stratonovich solution
considered in \cite{GLGL23}, allowing us to apply
their moment estimates.
\begin{remark}\label{rem:convention}
We use the diffusion generator $\partial_x^2$
to match the normalization of the Anderson Hamiltonian
in the spectral results used below.
To apply the results of \cite{Chen14,HHNT15},
which use the generator $\tfrac12\partial_x^2$,
we make the following time change.
Writing $B$ for a standard Brownian motion,
we may take $W_s=B_{2s}$, and hence
\begin{equation*}
    \int_0^t \xi(W_s)\,ds
    = \frac12\int_0^{2t}\xi(B_r)\,dr\,.
\end{equation*}
Accordingly, we apply the results of
\cite{Chen14,HHNT15} at time $2t$ and with
the potential $\xi/2$.
The results of \cite{GLGL23} are stated for a general
diffusion constant $\kappa$ and apply directly
with $\kappa=1$.
\end{remark}

A fundamental problem for the PAM is to understand the
long-time behavior of its solution. In the quenched
setting, one fixes a realization of the potential and studies
the almost-sure behavior of the solution. For generalized
Gaussian potentials including white noise, Chen obtained the
precise asymptotics of $\log u(t,x)$ \cite{Chen14}.
In our case, the growth is of order
$t(\log t)^{2/3}$ \cite[Theorem~1.4]{Chen14}.
In the annealed setting, one averages over the potential
and studies the moments of the solution. For the model \eqref{eq:PAM}, a direct computation using
\cite[Theorems~2.6 and~2.7]{GLGL23} gives
\begin{equation}\label{eq:Lyapunov_exponent}
    \gamma(p) := \lim_{t\to\infty}\frac{1}{t^{3}}
    \log\dE\bigl[u(t,x)^{p}\bigr]
    = \frac{p^{3}}{48}\,,
    \qquad p>0\,,
\end{equation}
where the limit does not depend on $x$ by stationarity of
$u(t,\cdot)$. Since $\gamma(p)/p$ is strictly increasing in $p$, the
solution is fully intermittent in the sense of
\cite{CM94}. The moments are driven by
rare, exceptionally high peaks of $u(t,\cdot)$ rather than by
its typical values.

It is natural to ask how the high peaks affect spatial
averages over boxes growing with time. In particular,
when is the spatial average asymptotic to its expectation,
and what is the limiting law of its fluctuations?
The growth rate of the box plays a central role in
both questions.

Our work is inspired by \cite{BABM05}, which studies
sums of i.i.d.\ random exponentials
$\sum_{i=1}^{N}e^{tX_i}$ with $N=N(t)$ growing with $t$.
When the $X_i$ have Weibull upper tails,  two thresholds for the growth
rate of $N(t)$ arise: above the first, the law of large numbers holds,
and above the second, the central limit theorem holds. Below the central limit threshold,
suitable centering and scaling yield totally asymmetric
stable laws.

For the PAM, establishing analogous results requires
control of the spatial dependence of the solution.
In the discrete setting, laws of large numbers and
central limit theorems have been established for
Bernoulli obstacles \cite{BAMR05}, for a wide class of
i.i.d.\ potentials \cite{BAMR07}, and for a potential
white in time and space on $\mathbb{Z}^d$ \cite{CM07}. As far as we know, stable limits have been obtained only for Weibull
potentials \cite{BAMR19}.\footnote{An earlier stable limit theorem was
stated in \cite{GS15}. As noted in \cite{BAMR19}, its proof has a
problem in the treatment of the normalization.}

For continuous models, the picture is far less
complete. Most known results concern a fixed time, and the limiting fluctuations are then Gaussian (see, for instance,
\cite{HNV20,BY22, CKNP23}, and \cite{KT24} for further
references). To the best of our knowledge, averages over
boxes growing with time have been considered only in
\cite{KY22,CKNP23,KT24}. For the stochastic heat
equation driven by space--time white noise, \cite{KY22}
obtained laws of large numbers and a central limit theorem
for averages over intervals of length $e^{\Lambda t}$. The
thresholds for $\Lambda$ are expressed in terms of the
moment Lyapunov exponents. For a critical long-range stochastic heat equation in dimension
$d\ge3$, \cite{KT24} obtained a non-Gaussian limit for averages over a
ball of radius $R$ at times of order $R^{2}$. This limit has finite variance and is therefore not a stable law.
Stable limit laws are thus unknown for the continuous setting. As noted in \cite{KT24}, stable fluctuations for the stochastic heat equation driven by space--time white noise have long been expected in analogy with the discrete case, and the case of a spatial
Gaussian noise remains open as well.

The obstruction is methodological. The law of large numbers and the central limit theorem
can often be proved using moment estimates, after approximating spatial averages by sums of independent box contributions.
For a stable limit, however, one needs the tail of an individual contribution at the normalization scale of the full sum, up to a factor $1+o(1)$.
Moment estimates alone do not provide this precision.
In \cite{BAMR19}, the spectral representation of the
solution on a box relates these tail estimates to
the principal eigenvalue. For Weibull potentials on the lattice, their analysis
uses rank-one perturbation theory around a dominant
site to obtain the required precision. The corresponding $\ell^{2}$-normalized eigenfunction concentrates at that site, and its $\ell^{1}$ norm
tends to one.

The solution to \eqref{eq:PAM} also admits a spectral
representation in terms of an Anderson Hamiltonian,
since the noise is time-independent. However, the lattice argument does not apply directly  because  spatial white noise has no pointwise values, and
the relevant eigenfunction is localized over a small
interval. Its $L^{1}$ norm is therefore no longer close to one, but depends on the
localization scale and on the profile, and its square appears in the
leading spectral contribution to the spatial integral. In this paper, we control both the lower tail of the lowest eigenvalue and the $L^{1}$ norm of its eigenfunction with the precision needed to prove stable limit laws. We now state our main results.

\subsection{Main results}\label{subsec:main_results}

Our main results concern the spatial integral
\begin{equation}\label{eq:spatial_integral}
    U(t) := \int_{Q_{L(t)}} u(t,x)\,dx\,,
\end{equation}
where $Q_{L(t)}:=(-L(t)/2,\,L(t)/2)$ and $u$ is given by \eqref{eq:FK_solution}, which is the solution to \eqref{eq:PAM}. The length $L(t)$ is chosen
in terms of the moment Lyapunov exponents
\eqref{eq:Lyapunov_exponent}. We parametrize its
growth by an exponent $\alpha>0$ through
\begin{equation}\label{eq:box_size}
    L(t) := \exp\bigl(2\gamma(\alpha)\,t^{3}\bigr)
    = \exp\biggl(\frac{\alpha^{3}t^{3}}{24}\biggr)\,.
\end{equation}
We begin with the
weak law of large numbers and the central limit theorem.

\begin{theorem}[Weak law of large numbers]\label{thm:PAM_WLLN}
    Let $L(t)$ be as in \eqref{eq:box_size} with $\alpha>1$. Then,
    as $t\to\infty$,
    \begin{equation}\label{eq:PAM_WLLN}
        \frac{U(t)}{\dE[U(t)]} \xrightarrow{\ \dP\ } 1\,.
    \end{equation}
\end{theorem}
 
\begin{theorem}[Central limit theorem]\label{thm:PAM_CLT}
    Let $L(t)$ be as in \eqref{eq:box_size} with $\alpha>2$. Then,
    as $t\to\infty$,
    \begin{equation}\label{eq:PAM_CLT}
        \frac{U(t)-\dE[U(t)]}{\sqrt{\Var\bigl(U(t)\bigr)}}
        \xrightarrow{\ d\ } \cN(0,1)\,.
    \end{equation}
\end{theorem}
 
Before stating the stable limit theorem, we introduce
an intermediate scale $\ell(t)$ and the integral of
the solution over a single interval of that length, which
enter the centering in the case $\alpha=1$. 
Let
\begin{equation}\label{eq:block_number}
    n(t):=\bigl\lfloor L(t)e^{-t}\bigr\rfloor\,,\qquad
    \ell(t):=\frac{L(t)}{n(t)}\,,
\end{equation}
so that $\ell(t)e^{-t}\to1$ as $t\to\infty$, and write
\begin{equation}\label{eq:block_integral}
    U_{0}(t):=\int_{Q_{\ell(t)}}u(t,x)\,dx\,.
\end{equation}
By stationarity of $u(t,\cdot)$, $U_{0}(t)$ has the same law as
the integral of $u(t,\cdot)$ over any interval of length
$\ell(t)$.
 
\begin{theorem}[Stable limit theorem]\label{thm:PAM_Stable}
    Let $L(t)$ be as in \eqref{eq:box_size} with $\alpha\in(0,2)$.
    Define
    \begin{equation}\label{eq:stable_scale_B_alpha}
        B_{\alpha}(t) := \frac{\pi}{2}
        \biggl(\frac{4\pi}{\alpha t}\biggr)^{1-1/\alpha}
        \exp\biggl(\frac{\alpha^{2}\,t^{3}}{16}\biggr)
    \end{equation}
    and
    \begin{equation}\label{eq:stable_centering}
        A(t) :=
        \begin{cases}
            0\,, & 0<\alpha<1\,,\\[4pt]
            \ell(t)^{-1}\,\dE\bigl[U_{0}(t)\,;\,
            U_{0}(t)\le B_{\alpha}(t)\bigr]\,,
            & \alpha=1\,,\\[8pt]
            \dE\bigl[u(t,0)\bigr]\,, & 1<\alpha<2\,.
        \end{cases}
    \end{equation}
    Then, as $t\to\infty$,
    \begin{equation}\label{eq:stable_convergence}
        \frac{U(t)-L(t)\,A(t)}{B_{\alpha}(t)}
        \xrightarrow{\ d\ } \dS_{\alpha}\,,
    \end{equation}
    where $\dS_{\alpha}$ is the $\alpha$-stable law with
    characteristic function
    \begin{equation}\label{eq:stable_char_fn}
        \phi_{\alpha}(z)=
        \begin{cases}
            \exp\Bigl(-\Gamma(1-\alpha)\,|z|^{\alpha}
            \exp\bigl(-\tfrac{i\pi\alpha}{2}\sgn(z)\bigr)\Bigr)\,,
            & 0<\alpha<1\,,\\[6pt]
            \exp\Bigl(iz(1-\gamma_{\mathrm{E}})
            -\tfrac{\pi}{2}|z|\bigl(1+i\sgn(z)\,
            \tfrac{2}{\pi}\log|z|\bigr)\Bigr)\,,
            & \alpha=1\,,\\[6pt]
            \exp\Bigl(\tfrac{\Gamma(2-\alpha)}{\alpha-1}\,
            |z|^{\alpha}
            \exp\bigl(-\tfrac{i\pi\alpha}{2}\sgn(z)\bigr)\Bigr)\,,
            & 1<\alpha<2\,.
        \end{cases}
    \end{equation}
    Here $\Gamma$ denotes the gamma function, $\sgn$ the sign
    function with $\sgn(0):=0$, and
    $\gamma_{\mathrm{E}}\approx 0.5772$ the Euler--Mascheroni
    constant. In particular, $\dS_{\alpha}$ is totally
    asymmetric, i.e., totally skewed to the right.
\end{theorem}

\begin{remark}
    Theorem~\ref{thm:PAM_WLLN} covers the case $\alpha=2$,
    but the limiting fluctuations remain open. For sums of i.i.d.\ random exponentials, a central limit theorem still holds at the boundary between the Gaussian and stable regimes, provided the variance is replaced by a truncated second moment
\cite[Theorem~2.5]{BABM05}. We expect an analogous behavior
    for $U(t)$ when $\alpha=2$.
\end{remark}

We now outline the proofs. The overall strategy follows the discrete case \cite{BAMR19}. 
We partition $Q_{L(t)}$ into $n(t)$ blocks of length $\ell(t)$, with $n(t)$ and $\ell(t)$ as in
\eqref{eq:block_number}. On each block, we consider the solution of
\eqref{eq:PAM} with Dirichlet boundary conditions.
Let $\tilde{U}_{0}(t)$ denote the integral of the Dirichlet solution over $Q_{\ell(t)}$.
Since $\xi$ is a white noise, its restrictions to
disjoint blocks are independent and identically distributed.
Hence, the integrals of the Dirichlet solutions over
their blocks are i.i.d., each with the same law as $\tilde{U}_{0}(t)$.
We establish the limit theorems for sums of these block integrals.
We then show that the difference between $U(t)$ and
these sums is negligible after centering and scaling, which transfers the limit theorems to $U(t)$. The advantage of this reduction is that each of these
integrals admits a spectral representation in terms
of the eigenpairs $(\lambda_{k},\varphi_{k})$ of the Anderson Hamiltonian on an interval of length $\ell(t)$. Since $-\xi$ has the same law as $\xi$, we state all spectral results for the operator $-\partial_{x}^{2}+\xi$ rather than $-\partial_{x}^{2}-\xi$.

We use the upper and lower moment bounds on
$\tilde{U}_{0}(t)$ from Proposition~\ref{prop:moment_bounds_Dirichlet_PAM}
to prove the weak law of large numbers and the central limit theorem for the block sums. For the stable limit theorem, we apply a classical
convergence criterion for triangular arrays \cite{Pet75} to
the block sums, as in \cite{BABM05,BAMR19}. The core of the
proof is the tail asymptotics
\begin{equation}\label{eq:intro_block_tail}
    \lim_{t\to\infty}n(t)\,
    \dP\bigl(\tilde{U}_{0}(t)>xB_{\alpha}(t)\bigr)=x^{-\alpha}\,,
    \qquad x>0\,.
\end{equation}
The parametrization \eqref{eq:box_size} of $L(t)$ is chosen
precisely so that the limit holds with index $\alpha$. 

To establish \eqref{eq:intro_block_tail}, we need
tail asymptotics with a relative error tending to zero.
The moment Lyapunov exponents give only the leading
exponential rate and do not provide this precision.
In particular, they do not determine the polynomial
prefactor or the constant in $B_{\alpha}(t)$. We obtain the required precision through the spectral representation. We show that, on the rare event in \eqref{eq:intro_block_tail}, the
approximation
\begin{equation*}
    \tilde{U}_{0}(t) \approx
    \|\varphi_{1}\|_{1}^{2}\,e^{-t\lambda_{1}}
\end{equation*}
holds with conditional probability tending to one.
This rests on two results of independent interest, the lower tail of
$\lambda_{1}$
(Theorem~\ref{thm:sharp_lower_tail_lowest_eigenvalue}) and the
concentration of $\|\varphi_{1}\|_{1}$ around an explicit deterministic
value on this lower-tail event
(Theorem~\ref{thm:eigenfunction_concentration}).
The constants in these two theorems determine $B_{\alpha}(t)$.

Theorem~\ref{thm:sharp_lower_tail_lowest_eigenvalue}
gives the exact asymptotics of $\dP(\lambda_{1}<-a)$, with the
precise Kramers prefactor, as $a$ and the interval length tend to
infinity together. McKean proved that $\lambda_{1}$ converges
in law to a Gumbel distribution after centering and scaling
\cite{McK94}. This weak convergence gives no information in
this regime. The best available estimates are the two-sided bounds of Hsu and
Labb\'e \cite{HL23}. These bounds determine the exponent only up to a
factor $1\pm\eta$ for a fixed $\eta>0$. The resulting error is
exponentially large, so the prefactor cannot be identified. 
Our proof starts from the Riccati transform
\cite{McK94,AD14,DL20}. It turns the lower-tail probability of $\lambda_{1}$ into the probability that a diffusion explodes before a given time. The diffusion is
attracted to a potential well, and an explosion requires
the crossing of a potential barrier, a Kramers escape. This picture suggests a renewal structure with many failed attempts near the well followed by one crossing.
The difficulty is that the given time and the level $a$ grow together, so the
well, the barrier, and the law of every attempt move with
them. Classical renewal theorems require a fixed increment
distribution and do not apply. We instead build a renewal structure
adapted to this movement and make every estimate explicit
in the level (see Section~\ref{subsec:sharp_lower_tail} for
details).

The localization theory of Dumaz and Labb\'e \cite{DL20} describes
the eigenfunctions at the bottom of the spectrum on events of
probability tending to one, but gives no information
conditionally on the lower-tail event, whose probability
vanishes. Theorem~\ref{thm:eigenfunction_concentration}
provides exactly this information. Our proof starts from
the variational characterization of $\lambda_{1}$, which
assigns a Gaussian cost to the lower-tail event. The
Gagliardo--Nirenberg inequality identifies the minimizers
of this cost. They are the translates of an explicit sech profile, all with the same rescaled $L^{1}$ norm $\pi/\sqrt{2}$. The difficulty is that the cost bound concerns one fixed profile, while the center of the eigenfunction is random and the $L^{1}$ norm does not control the shape on a long interval. We handle the randomness with a
mesh of candidate centers. For each center, the spectral gap confines
the mass of the eigenfunction to a bounded neighborhood, which reduces
the problem to profiles on a fixed window, and the Borell--TIS
inequality shows that any profile far from a minimizer of the cost is
too expensive to occur (see
Section~\ref{subsec:eigenfunction_concentration} for details).

The remainder of the paper is organized as follows.
Section~\ref{sec:Anderson_Hamiltonian_finite_interval} proves the two spectral results, Theorems~\ref{thm:sharp_lower_tail_lowest_eigenvalue} and~\ref{thm:eigenfunction_concentration}.
Section~\ref{sec:Dirichlet_PAM_finite_interval} introduces the
Dirichlet solutions, recalls the spectral representation,
proves the moment bounds for their spatial integrals, and
prepares the comparison of $U(t)$ with the block sums.
Sections~\ref{sec:proof_PAM_WLLN}, \ref{sec:proof_PAM_CLT},
and~\ref{sec:proof_PAM_Stable} prove
Theorems~\ref{thm:PAM_WLLN}, \ref{thm:PAM_CLT},
and~\ref{thm:PAM_Stable}, respectively. Finally, Appendix~\ref{apx:Auxiliary_estimates} collects auxiliary
estimates on the diffusion associated with the Riccati
transform, which are used in the proofs of
Section~\ref{sec:Anderson_Hamiltonian_finite_interval}.

\medskip
\noindent\textbf{Notation.}
Throughout the paper, for positive functions $f$ and
$g$, we write $f\sim g$ if $f/g\to1$, and $f\ll g$ if
$f/g\to0$, as the argument tends to infinity. 
  
\section{The Anderson Hamiltonian on a finite interval}
\label{sec:Anderson_Hamiltonian_finite_interval}
In this section, we write $t$ for the spatial coordinate of the Anderson Hamiltonian. This coordinate also serves as the time parameter of the associated Riccati diffusion. For $\cL>0$, we consider the Anderson Hamiltonian on $[0,\cL]$ with
Dirichlet boundary conditions,
\begin{equation}\label{eq:Anderson_Hamiltonian}
    \cH_{\cL} := -\partial_{t}^{2}+\xi\,,
\end{equation}
where $\xi$ is a white noise on $[0,\cL]$. We adopt the sign convention
\eqref{eq:Anderson_Hamiltonian}, which is standard in the
literature.

This operator has long been studied, beginning with Frisch
and Lloyd \cite{FL60} and Halperin \cite{Hal65}, who
computed the density of states. A rigorous
construction was given by Fukushima and Nakao, who showed that
the operator is self-adjoint with discrete spectrum \cite{FN77}.
McKean proved that the lowest eigenvalue has Gumbel fluctuations
as $\cL\to\infty$ \cite{McK94}.
More recently, Dumaz and Labb\'e obtained a detailed description of the
spectrum in the large-volume limit, from the localized bottom through the
crossover to the delocalized regime at high energies \cite{DL20,DL23,DL24}.

\subsection{The eigenvalue problem and the Riccati transform}
\label{subsec:eigenvalue_problem_Riccati}
In this subsection, we review the construction of $\cH_{\cL}$,
the Riccati transform, and the results of \cite{McK94,AD14} on
the first explosion time. The Riccati transform is the main
tool of this section. Let
\begin{equation*}
    H_{0}^{1}(0,\cL) := \bigl\{\varphi\in H^{1}(0,\cL)
    \,:\,\varphi(0)=\varphi(\cL)=0\bigr\}\,.
\end{equation*}
We realize the white noise as $\xi=B'$, the distributional
derivative of a Brownian motion $B$ on $[0,\cL]$.
Thus, for every $w\in H_{0}^{1}(0,\cL)$,
\begin{equation}\label{eq:def_noise_pairing}
    \langle\xi,w\rangle = -\int_{0}^{\cL}w'(t)\,B(t)\,dt\,.
\end{equation}
Following Fukushima and Nakao \cite{FN77}, we define $\cH_{\cL}$ through the quadratic form
\begin{equation}\label{eq:FN_form}
    \cE_{\cL}(u,v) := \int_{0}^{\cL}u'v'\,dt
    - \int_{0}^{\cL}(uv)'(t)\,B(t)\,dt
    = \int_{0}^{\cL}u'v'\,dt + \langle\xi,uv\rangle\,,
    \qquad u,v\in H_{0}^{1}(0,\cL)\,.
\end{equation}
Almost surely, $\cE_{\cL}$ is closed and bounded below, and
$\cH_{\cL}$ is the associated self-adjoint operator on $L^{2}(0,\cL)$. 
Its resolvent is compact, so the spectrum is discrete. 
A pair $(\lambda,\varphi)\in\bR\times H_{0}^{1}(0,\cL)$ with
$\varphi\neq0$ is an eigenpair if, for all
$v\in H_{0}^{1}(0,\cL)$,
\begin{equation}\label{eq:weak_eigen_relation}
    \cE_{\cL}(\varphi,v) = \lambda\,\langle\varphi,v\rangle\,.
\end{equation}
The eigenpairs can also be described through an integrated form of the
eigenvalue equation. For $\lambda\in\bR$, we say that $\phi$ solves
$\cH_{\cL}\phi=\lambda\phi$ in the integrated sense if
$\phi\in C^{1}([0,\cL])$ and, for all $t\in[0,\cL]$,
\begin{equation}\label{eq:integrated_eigenvalue_equation}
    \phi'(t)-\phi'(0)
    = \phi(t)B(t)-\int_{0}^{t}\phi'(s)B(s)\,ds
    - \lambda\int_{0}^{t}\phi(s)\,ds\,.
\end{equation}
Then $(\lambda,\varphi)$ is an eigenpair if and only if
$\varphi\neq0$ solves $\cH_{\cL}\varphi=\lambda\varphi$ in the
integrated sense and $\varphi(0)=\varphi(\cL)=0$.
The eigenvalues are simple, since a solution of
\eqref{eq:integrated_eigenvalue_equation} with $\phi(0)=0$ is determined by
$\phi'(0)$ (see also \cite[Theorem~2.1]{AC15}). We denote them by
\begin{equation*}
    \lambda_{1}(\cL) < \lambda_{2}(\cL) < \cdots\,,
\end{equation*}
with corresponding eigenfunctions $\varphi_{k}$, normalized by
$\|\varphi_{k}\|_{2}=1$.
The first two eigenvalues admit the variational
characterization
\begin{equation}\label{eq:minimax_characterization}
    \lambda_{1}(\cL)=\inf_{\substack{u\in H_{0}^{1}(0,\cL)\\ \|u\|_{2}=1}}
    \cE_{\cL}(u,u)\,,
    \qquad
    \lambda_{2}(\cL)=\inf_{\substack{u\in H_{0}^{1}(0,\cL)\\
    \|u\|_{2}=1,\ u\perp\varphi_{1}}}
    \cE_{\cL}(u,u)\,.
\end{equation}
For each $a\in\bR$, let $\phi_{a}$ be the solution of
$\cH_{\cL}\phi_{a}=-a\phi_{a}$ in the integrated sense with
$\phi_{a}(0)=0$ and $\phi_{a}'(0)=1$. Then $-a$ is an
eigenvalue of $\cH_{\cL}$ if and only if $\phi_{a}(\cL)=0$.
We extend $B$ to a Brownian motion on $[0,\infty)$.
On each interval on which $\phi_{a}$ does not vanish, the Riccati
transform $X_{a}:=\phi_{a}'/\phi_{a}$ satisfies
\begin{equation}\label{eq:Riccati_SDE}
    dX_{a}(t) = \bigl(a-X_{a}(t)^{2}\bigr)\,dt + dB(t)\,,
    \qquad t\ge0\,,
\end{equation}
with $X_{a}(0)=+\infty$ (see \cite[Section~2.1]{DL20} for the
construction of the diffusion started from $+\infty$).
Each zero of $\phi_{a}$ corresponds to an explosion of $X_{a}$ to $-\infty$,
after which the diffusion restarts from $+\infty$.
Moreover, the number of explosions of $X_{a}$ before time
$\cL$ equals the number of eigenvalues of $\cH_{\cL}$ below
$-a$ \cite[Sections~2 and~3]{AD14} (see also \cite{McK94,DL20}).
For $a=-\lambda_{1}(\cL)$ there are no eigenvalues below $-a$,
so $X_{a}$ does not explode before time $\cL$ and $\phi_{a}$
has no zeros in $(0,\cL)$. Since $\varphi_{1}$ is a multiple
of $\phi_{a}$, it has a constant sign, and we take
$\varphi_{1}\ge0$.
Set $\zeta_{a}^{(0)}:=0$, and for $k\ge1$ let $\zeta_{a}^{(k)}$
denote the $k$-th explosion time of the restarted diffusion, which is
finite almost surely \cite[Section~3.2]{AD14}.
Consequently, for $k\ge1$,
\begin{equation}\label{eq:eigenvalue_explosion}
    \dP\bigl(\lambda_{k}(\cL)<-a\bigr)
    = \dP\bigl(\zeta_{a}^{(k)}<\cL\bigr)\,.
\end{equation}
By the strong Markov property, the increments
$\zeta_{a}^{(k)}-\zeta_{a}^{(k-1)}$, $k\ge1$, are independent
and each distributed as $\zeta_{a}^{(1)}$. We write
\begin{equation}\label{eq:first_explosion_mean}
    m(a) := \dE\bigl[\zeta_{a}^{(1)}\bigr]
\end{equation}
for the mean of the first explosion time.
McKean studied the first explosion time through its Laplace
transform, which solves a boundary-value problem for the
generator of $X_{a}$ \cite{McK94}. In our notation, his
result reads
\begin{equation}\label{eq:first_explosion_exp_limit}
    \frac{\zeta_{a}^{(1)}}{m(a)} \xrightarrow{\ d\ }
    \mathrm{Exp}(1) \qquad\text{as } a\to\infty\,,
\end{equation}
and we refer to \cite[Theorem~3.3]{AD14} for a proof.
We also recall from \cite{McK94} the exact formula
\begin{equation}\label{eq:first_explosion_mean_formula}
    m(a) = \sqrt{2\pi}\int_{0}^{\infty}\frac{1}{\sqrt{v}}
    \exp\Bigl(2av-\frac{v^{3}}{6}\Bigr)\,dv\,.
\end{equation}
Its asymptotic behavior as $a\to\infty$, noted in \cite{McK94}
and computed in detail in \cite[Appendix~3]{AD14}, is
\begin{equation}\label{eq:first_explosion_mean_asymptotics}
    m(a) = \frac{\pi}{\sqrt{a}}\exp\Bigl(\frac{8}{3}a^{3/2}\Bigr)
    \bigl(1+o(1)\bigr)\,.
\end{equation}

\subsection{Sharp lower-tail asymptotics for the lowest eigenvalue}
\label{subsec:sharp_lower_tail}
We begin with a heuristic for the lower tail of
$\lambda_{1}(\cL)$. By \eqref{eq:eigenvalue_explosion} with $k=1$,
$\dP(\lambda_{1}(\cL)<-a)$ is the probability that $X_{a}$
explodes before time $\cL$. The convergence
\eqref{eq:first_explosion_exp_limit} suggests that this
probability is close to $1-e^{-\cL/m(a)}$. Therefore, in the
joint regime $a,\cL\to\infty$ with $\cL\ll m(a)$, we expect
\begin{equation}\label{eq:relation_lambda_zeta}
    \dP\bigl(\lambda_{1}(\cL)<-a\bigr)
    \approx \frac{\cL}{m(a)}\,.
\end{equation}
However, \eqref{eq:first_explosion_exp_limit} gives no
information in this regime. Weak convergence
controls $\dP(\zeta_{a}^{(1)}<x\,m(a))$ at each fixed $x>0$,
whereas in \eqref{eq:relation_lambda_zeta} the relevant point
$x=\cL/m(a)$ tends to zero as $a$ and $\cL$ vary simultaneously.
On the other hand, the joint regime was studied by Hsu and Labb\'e \cite[Theorem~2]{HL23}. In dimension one, and for the lowest eigenvalue, their result states that for every $\eta\in(0,1)$ there exist $\gamma_{2}>\gamma_{1}>0$ and $a_{0}>0$ such that, for all $\cL\ge1$ and $a\ge a_{0}$,
\begin{equation}\label{eq:HL_tail_bounds}
    \exp\Bigl(-\gamma_{2}\,\cL\sqrt{a}\,
    e^{-\frac{8}{3}(1-\eta)a^{3/2}}\Bigr)
    \le \dP\bigl(\lambda_{1}(\cL)\ge-a\bigr)
    \le \exp\Bigl(-\gamma_{1}\,\cL\sqrt{a}\,
    e^{-\frac{8}{3}(1+\eta)a^{3/2}}\Bigr)\,.
\end{equation}
These bounds hold with no restriction on the relative growth
of $a$ and $\cL$. When $\cL\sqrt{a}\,e^{-\frac{8}{3}a^{3/2}}\to0$,
the bounds \eqref{eq:HL_tail_bounds} yield
\begin{equation*}
    \gamma_{1}\,\cL\sqrt{a}\,e^{-\frac{8}{3}(1+\eta)a^{3/2}}
    \bigl(1+o(1)\bigr)
    \le \dP\bigl(\lambda_{1}(\cL)<-a\bigr)
    \le \gamma_{2}\,\cL\sqrt{a}\,e^{-\frac{8}{3}(1-\eta)a^{3/2}}\,.
\end{equation*}
This determines neither the exponent, due to the loss $\eta>0$,
nor the prefactor, since $\gamma_{1}<\gamma_{2}$.
The following theorem provides the sharp asymptotics in the window \eqref{eq:admissible_window} and makes \eqref{eq:relation_lambda_zeta} precise.

\begin{theorem}\label{thm:sharp_lower_tail_lowest_eigenvalue}
    Let $(\cL_{a})_{a>0}$ be a family of positive numbers satisfying, as $a\to\infty$,
    \begin{equation}\label{eq:admissible_window}
        a\exp\bigl(4(\log a)^{2}\bigr) 
        \ll \cL_{a} \ll m(a)\,,
    \end{equation}
    where $m(a)$ is defined in \eqref{eq:first_explosion_mean}. Then, as $a\to\infty$,
    \begin{equation}\label{eq:lower_tail_asymptotics}
        \dP\bigl(\lambda_{1}(\cL_{a})<-a\bigr)
        \sim \frac{\cL_{a}}{m(a)}
        \sim \cL_{a}\,\frac{\sqrt{a}}{\pi}
        \exp\Bigl(-\frac{8}{3}a^{3/2}\Bigr)\,.
    \end{equation}
\end{theorem}

\begin{remark}\label{rem:admissible_window}
    The two bounds in \eqref{eq:admissible_window} play different
    roles. The upper bound cannot be removed. If
    $\cL_{a}/m(a)$ converges
    to some $c>0$, then by \eqref{eq:first_explosion_exp_limit} the
    probability $\dP(\lambda_{1}(\cL_{a})<-a)$ converges to
    $1-e^{-c}$. The event is then no longer rare, and
    \eqref{eq:lower_tail_asymptotics} fails. The lower bound is technical. It comes from the error terms arising in the proof (see, e.g., the overshoot
    estimate \eqref{eq:sharp_estimate_overshoot}), and we have
    not attempted to optimize it. 
\end{remark}

We now outline the proof of
Theorem~\ref{thm:sharp_lower_tail_lowest_eigenvalue}. Thanks to \eqref{eq:eigenvalue_explosion}, we will estimate $\dP\bigl(\zeta_{a}^{(1)}<\cL\bigr)$. The drift in
\eqref{eq:Riccati_SDE} can be written as $a-x^{2}=-V'(x)$, where
\begin{equation}\label{eq:Riccati_potential}
    V(x) = \frac{x^{3}}{3}-ax\,.
\end{equation}

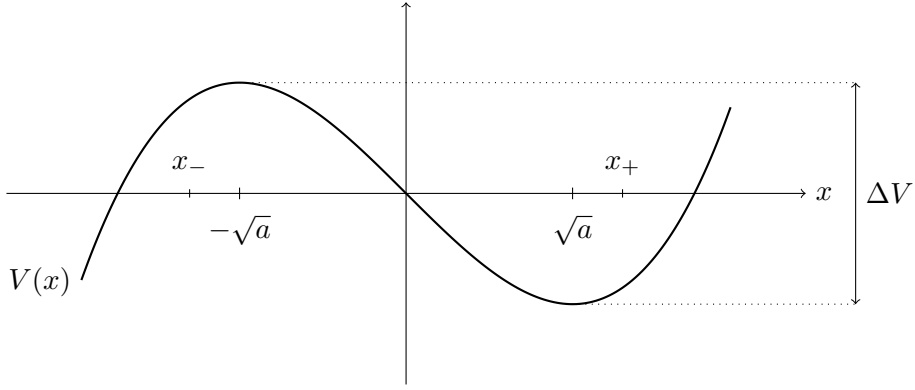
\begin{figure}[t]
\centering
\begin{tikzpicture}[scale=2.2]
% Axes
\draw[->] (-2.4,0) -- (2.4,0) node[right] {$x$};
\draw[->] (0,-1.15) -- (0,1.15);
% Potential
\draw[thick, smooth, domain=-1.95:1.95, samples=100]
    plot (\x, {\x*\x*\x/3 - \x});
\node[left] at (-1.95,-0.53) {$V(x)$};
% Barrier and well locations
\draw (-1,0.035) -- (-1,-0.035)
    node[below=2pt] {$-\sqrt{a}$};
\draw (1,0.035) -- (1,-0.035)
    node[below=2pt] {$\sqrt{a}$};
% Auxiliary levels
\draw (-1.30,0.025) -- (-1.30,-0.025);
\node[above=3pt] at (-1.30,0) {$x_{-}$};
\draw (1.30,0.025) -- (1.30,-0.025);
\node[above=3pt] at (1.30,0) {$x_{+}$};
% Barrier height
\draw[dotted] (-1,0.6667) -- (2.7,0.6667);
\draw[dotted] (1,-0.6667) -- (2.7,-0.6667);
\draw[<->] (2.7,0.6667) -- (2.7,-0.6667)
    node[midway, right] {$\Delta V$};
\end{tikzpicture}
\caption{The potential $V(x)=\frac{x^{3}}{3}-ax$, with the well
at $\sqrt{a}$, the barrier at $-\sqrt{a}$, and the auxiliary
levels $x_{\pm}$ defined in \eqref{eq:def_delta_x_pm}.}
\label{fig:Riccati_potential}
\end{figure}

The potential $V$ has a local minimum at $\sqrt{a}$ and a local
maximum at $-\sqrt{a}$, with barrier height
\begin{equation*}
    \Delta V = V\bigl(-\sqrt{a}\bigr)-V\bigl(\sqrt{a}\bigr)
    = \frac{4}{3}\,a^{3/2}\,.
\end{equation*}
The diffusion is therefore attracted to the well at $\sqrt{a}$, and
an explosion to $-\infty$ requires crossing the barrier at
$-\sqrt{a}$. Heuristically, $X_{a}$ falls quickly from $+\infty$
into the well, stays there for an exponentially long time,
and eventually crosses the barrier, after which it is driven
to $-\infty$ almost immediately. This is the classical picture of the Kramers escape \cite{HTB90}. The exponent $2\Delta V=\frac{8}{3}a^{3/2}$ and
the prefactor
$\sqrt{V''(\sqrt{a})\,\lvert V''(-\sqrt{a})\rvert}/(2\pi)
=\sqrt{a}/\pi$ are precisely those appearing in
\eqref{eq:lower_tail_asymptotics}.
We make the heuristic precise through a renewal structure.
Starting near the well, the diffusion makes successive
attempts to cross the barrier. An attempt succeeds if the
diffusion reaches the far side of the barrier, and fails if it
returns to its starting level. Each new attempt then starts
from the same level, so successive attempts are i.i.d.\ by the
strong Markov property. The proof thus reduces to estimating
the probability of success in a single attempt and the time
accumulated over the failed ones. The first step reduces
$\dP\bigl(\zeta_{a}^{(1)}<\cL_{a}\bigr)$ to a
hitting probability. To state this reduction, we fix some notation. For $a>1$, set
\begin{equation}\label{eq:def_delta_x_pm}
    \delta(a) := \frac{\log a}{a^{1/4}}\,,
    \qquad
    x_{\pm}(a) := \pm\bigl(\sqrt{a}+\delta(a)\bigr)\,,
\end{equation}
and write $x_{\pm}$ instead of $x_{\pm}(a)$ for notational
simplicity. We also define
\begin{equation}\label{eq:deterministic_time_scale}
    \tau(a)
    := \int_{x_{+}}^{\infty}\frac{dz}{z^{2}-a}
    = \int_{-\infty}^{x_{-}}\frac{dz}{z^{2}-a}
    = \frac{1}{2\sqrt{a}}
    \log\biggl(\frac{2\sqrt{a}+\delta(a)}{\delta(a)}\biggr)\,,
\end{equation}
the time required for the solution of the ODE $x'=a-x^{2}$ to
descend from $+\infty$ to $x_{+}$ or, equivalently, from
$x_{-}$ to $-\infty$. For $y\in\bR$, write
\begin{equation}\label{eq:def_hitting_time}
    T_{y} := \inf\bigl\{t\ge0\,:\,X_{a}(t)=y\bigr\}\,,
\end{equation}
and set $H:=T_{x_{-}}\wedge T_{x_{+}}$.
In the rest of this subsection and in
Appendix~\ref{apx:Auxiliary_estimates}, $\dP$ and $\dE$ denote the
probability and expectation for the solution of \eqref{eq:Riccati_SDE}
started from the entrance boundary $+\infty$, and, for $z\in\bR$, we write
$\dP_{z}$ and $\dE_{z}$ when the diffusion starts from $X_{a}(0)=z$.
Probabilities and expectations under $\dP$ are the limits of the
corresponding quantities under $\dP_{z}$ as $z\to+\infty$.

We replace the boundaries $\pm\infty$ by the finite levels
$x_{\pm}$. The renewal argument requires the attempts to
restart from a level that the diffusion revisits before the
explosion. The entrance boundary $+\infty$ is left immediately
and never revisited, whereas the diffusion returns to $x_{+}$
after every failed attempt. With our choice of $\delta(a)$, the two outer passages, from $+\infty$ to $x_+$ and from $x_-$ to the explosion, each take at most $\tau(a)+\sqrt a$ with probability tending to one, as shown in the proof below. Since $2\tau(a)+2\sqrt a=o(\cL_a)$, the explosion time is governed by the passage from $x_+$ to $x_-$. The following proposition makes this reduction precise.

\begin{proposition}\label{prop:hitting_reduction}
    Let $(\cL_{a})_{a>0}$ satisfy \eqref{eq:admissible_window} and
    set $\bar{\cL}_{a}:=\cL_{a}-2\tau(a)-2\sqrt{a}$. Then, as
    $a\to\infty$,
    \begin{equation}\label{eq:hitting_reduction}
        (1-o(1))\,\dP_{x_{+}}\bigl(T_{x_{-}}<\bar{\cL}_{a}\bigr)
        \le \dP\bigl(\zeta_{a}^{(1)}<\cL_{a}\bigr)
        \le \dP_{x_{+}}\bigl(T_{x_{-}}<\cL_{a}\bigr)\,.
    \end{equation}
\end{proposition}

\begin{proof}
    We first prove the upper bound in \eqref{eq:hitting_reduction}.
    Since $X_{a}$ starts from $+\infty$, on the event
    $\bigl\{\zeta_{a}^{(1)}<\cL_{a}\bigr\}$ it hits $x_{+}$
    and then $x_{-}$ before exploding. Hence, by the strong
    Markov property at $T_{x_{+}}$,
    \begin{equation*}
        \dP\bigl(\zeta_{a}^{(1)}<\cL_{a}\bigr)
        \le \dP_{x_{+}}\bigl(\zeta_{a}^{(1)}<\cL_{a}\bigr)
        \le \dP_{x_{+}}\bigl(T_{x_{-}}<\cL_{a}\bigr)\,.
    \end{equation*}
    We now prove the lower bound. Write
    \begin{equation}\label{eq:tau_*}
    \tau_{*}(a):=\tau(a)+\sqrt{a}\,,
    \end{equation} and define the events
    \begin{equation*}
        E_{1} := \bigl\{T_{x_{+}}<\tau_{*}(a)\bigr\}\,,\quad
        E_{2} := \bigl\{T_{x_{-}}-T_{x_{+}}<\bar{\cL}_{a}\bigr\}\,,\quad
        E_{3} := \bigl\{\zeta_{a}^{(1)}-T_{x_{-}}<\tau_{*}(a)\bigr\}\,.
    \end{equation*}
    On $E_{1}\cap E_{2}\cap E_{3}$, we have
    $\zeta_{a}^{(1)}<\bar{\cL}_{a}+2\tau_{*}(a)=\cL_{a}$.
    Hence, by the strong Markov property at $T_{x_{+}}$ and then
    at $T_{x_{-}}$,
    \begin{equation*}
        \dP\bigl(\zeta_{a}^{(1)}<\cL_{a}\bigr)
        \ge \dP\bigl(E_{1}\cap E_{2}\cap E_{3}\bigr)
        = \dP(E_{1})\,\dP_{x_{+}}\bigl(T_{x_{-}}<\bar{\cL}_{a}\bigr)\,
          \dP_{x_{-}}\bigl(\zeta_{a}^{(1)}<\tau_{*}(a)\bigr)\,.
    \end{equation*}
    Thus it suffices to show that, as $a\to\infty$,
    \begin{equation}\label{eq:two_remaining_limits}
        \dP(E_{1})\to1\,,\qquad
        \dP_{x_{-}}\bigl(\zeta_{a}^{(1)}<\tau_{*}(a)\bigr)\to1\,.
    \end{equation}
    We begin with the first limit in
    \eqref{eq:two_remaining_limits}. For $x\ge x_{+}$, define
    \begin{equation*}
        F_{1}(x) := \int_{x_{+}}^{x}\frac{dz}{a-z^{2}}\,.
    \end{equation*}
    Then $F_{1}\le0$ on $[x_{+},\infty)$, and $F_{1}$
    extends continuously to $+\infty$ with
    $F_{1}(+\infty)=-\tau(a)$. Moreover,
    \begin{equation*}
        F_{1}'(x)=\frac{1}{a-x^{2}}\,,\qquad
        F_{1}''(x)=\frac{2x}{(a-x^{2})^{2}}\,.
    \end{equation*}
    Applying It\^o's formula to $F_{1}(X_{a})$ stopped at
    $\tau_{*}(a)\wedge T_{x_{+}}$, and using the continuous
    extension of $F_{1}$, we obtain
    \begin{equation}\label{eq:F1_Ito_formula}
        F_{1}\bigl(X_{a}(\tau_{*}(a)\wedge T_{x_{+}})\bigr)
        = -\tau(a)
        + \int_{0}^{\tau_{*}(a)\wedge T_{x_{+}}}
          \biggl(1+\frac{X_{a}(s)}{(a-X_{a}(s)^{2})^{2}}\biggr)\,ds
        + M^{(1)}_{\tau_{*}(a)}\,,
    \end{equation}
    where
    \begin{equation*}
        M^{(1)}_{t} := \int_{0}^{t\wedge T_{x_{+}}}
        \frac{dB(s)}{a-X_{a}(s)^{2}}\,.
    \end{equation*}
    We claim that
    \begin{equation*}
        \Bigl\{M^{(1)}_{\tau_{*}(a)}>-\frac{\sqrt{a}}{2}\Bigr\}
        \subset E_{1}\,.
    \end{equation*}
    Indeed, for $s\le T_{x_{+}}$ we have $X_{a}(s)\ge x_{+}>0$, so
    the drift integrand in \eqref{eq:F1_Ito_formula} is at least
    $1$. Since $F_{1}\le0$ on $[x_{+},\infty)$, it follows that
    \begin{equation*}
        -\tau(a)+\bigl(\tau_{*}(a)\wedge T_{x_{+}}\bigr)
        +M^{(1)}_{\tau_{*}(a)}
        \le F_{1}\bigl(X_{a}(\tau_{*}(a)\wedge T_{x_{+}})\bigr)
        \le 0\,.
    \end{equation*}
    On the event $\bigl\{M^{(1)}_{\tau_{*}(a)}>-\sqrt{a}/2\bigr\}$, this
    yields $\tau_{*}(a)\wedge T_{x_{+}}<\tau_{*}(a)$, and hence $E_{1}$ occurs.
    Consequently,
    \begin{equation*}
        \dP\bigl(E_{1}^{c}\bigr)
        \le \dP\Bigl(M^{(1)}_{\tau_{*}(a)}\le-\frac{\sqrt{a}}{2}\Bigr)\,.
    \end{equation*}
    To bound the right-hand side, note that $X_{a}(s)\ge x_{+}$ for
    $s\le T_{x_{+}}$ implies
    \begin{equation*}
        \bigl|a-X_{a}(s)^{2}\bigr|
        \ge x_{+}^{2}-a
        \ge 2\sqrt{a}\,\delta(a)\,,
    \end{equation*}
    so that
    \begin{equation*}
        \bigl\langle M^{(1)}\bigr\rangle_{\tau_{*}(a)}
        \le \frac{\tau_{*}(a)}{4a\,\delta(a)^{2}}\,.
    \end{equation*}
    Applying Markov's inequality to the supermartingale
    $\exp\bigl(-\lambda M^{(1)}_{t}
    -\tfrac{\lambda^{2}}{2}\bigl\langle M^{(1)}\bigr\rangle_{t}\bigr)$
    and optimizing over $\lambda>0$, we obtain
    \begin{equation*}
        \dP\Bigl(M^{(1)}_{\tau_{*}(a)}\le-\frac{\sqrt{a}}{2}\Bigr)
        \le \exp\biggl(-\frac{a^{2}\,\delta(a)^{2}}
        {2\,\tau_{*}(a)}\biggr)\,.
    \end{equation*}
    By the definition of $\tau_*(a)$ and $\delta(a)$ (see \eqref{eq:def_delta_x_pm} and \eqref{eq:tau_*}), we have $\dP(E_{1})\to1$.

    We turn to the second limit in
    \eqref{eq:two_remaining_limits}, that is,
    $\dP_{x_{-}}\bigl(\zeta_{a}^{(1)}<\tau_{*}(a)\bigr)\to1$. For
    $x<-\sqrt{a}$, define
    \begin{equation*}
        F_{2}(x) := \int_{-\infty}^{x}\frac{dz}{z^{2}-a}\,.
    \end{equation*}
    Then $F_{2}\ge0$ on $(-\infty,-\sqrt{a})$, $F_{2}$
    extends continuously to $-\infty$ with $F_{2}(-\infty)=0$,
    and $F_{2}(x_{-})=\tau(a)$ by \eqref{eq:deterministic_time_scale}. 
    Moreover,
    \begin{equation*}
        F_{2}'(x)=\frac{1}{x^{2}-a}\,,\qquad
        F_{2}''(x)=-\frac{2x}{(x^{2}-a)^{2}}\,.
    \end{equation*}
    Since $F_{2}$ diverges as $x\uparrow-\sqrt{a}$, we stop the
    diffusion $X_{a}$ at the level $x_{-}+\delta(a)/2$, i.e., we set
    \begin{equation*}
        T_{\mathrm{up}}
        := \inf\bigl\{t<\zeta_{a}^{(1)}\,:\,
        X_{a}(t)=x_{-}+\delta(a)/2\bigr\}\,,
    \end{equation*}
    with the convention $\inf\emptyset=\infty$. Applying It\^o's formula to $F_{2}(X_{a})$ under $\dP_{x_{-}}$, stopped at $\tau_{*}(a)\wedge\zeta_{a}^{(1)}\wedge T_{\mathrm{up}}$, and using the continuous extension of $F_{2}$, we
    obtain
    \begin{equation}\label{eq:F2_Ito_formula}
        F_{2}\bigl(X_{a}(\tau_{*}(a)\wedge\zeta_{a}^{(1)}
        \wedge T_{\mathrm{up}})\bigr)
        = \tau(a)
        - \bigl(\tau_{*}(a)\wedge\zeta_{a}^{(1)}
        \wedge T_{\mathrm{up}}\bigr)
        + M^{(2)}_{\tau_{*}(a)} + R_{\tau_{*}(a)}\,,
    \end{equation}
    where
    \begin{equation*}
        M^{(2)}_{t} :=
        \int_{0}^{t\wedge\zeta_{a}^{(1)}\wedge T_{\mathrm{up}}}
        \frac{dB(s)}{X_{a}(s)^{2}-a}\,,
        \qquad
        R_{t} :=
        -\int_{0}^{t\wedge\zeta_{a}^{(1)}\wedge T_{\mathrm{up}}}
        \frac{X_{a}(s)}{\bigl(X_{a}(s)^{2}-a\bigr)^{2}}\,ds\,.
    \end{equation*}
    Consider the event
    \begin{equation*}
        E_{4} := \bigl\{T_{\mathrm{up}}>\tau_{*}(a)\bigr\}
        \cap \Bigl\{\bigl|M^{(2)}_{\tau_{*}(a)}\bigr|
        <\frac{\sqrt{a}}{4}\Bigr\}
        \cap \Bigl\{\bigl|R_{\tau_{*}(a)}\bigr|
        <\frac{\sqrt{a}}{4}\Bigr\}\,.
    \end{equation*}
    We claim that
    $E_{4}\subset\bigl\{\zeta_{a}^{(1)}<\tau_{*}(a)\bigr\}$.
    Indeed, since $F_{2}\ge0$, \eqref{eq:F2_Ito_formula} gives
    \begin{equation*}
        \tau_{*}(a)\wedge\zeta_{a}^{(1)}\wedge T_{\mathrm{up}}
        \le \tau(a)+M^{(2)}_{\tau_{*}(a)}+R_{\tau_{*}(a)}\,.
    \end{equation*}
    On $E_{4}$, the left-hand side equals
    $\tau_{*}(a)\wedge\zeta_{a}^{(1)}$ because
    $T_{\mathrm{up}}>\tau_{*}(a)$, while the right-hand side is
    less than $\tau(a)+\sqrt{a}/2<\tau_{*}(a)$. Hence
    $\zeta_{a}^{(1)}<\tau_{*}(a)$, and the claim follows. It therefore
    suffices to show that $\dP_{x_{-}}\bigl(E_{4}^{c}\bigr)\to0$.
    By the union bound,
    \begin{equation}\label{eq:E4_complement_bound}
        \dP_{x_{-}}\bigl(E_{4}^{c}\bigr)
        \le \dP_{x_{-}}\bigl(T_{\mathrm{up}}\le\tau_{*}(a)\bigr)
        + \dP_{x_{-}}\Bigl(\bigl|M^{(2)}_{\tau_{*}(a)}\bigr|
        \ge\frac{\sqrt{a}}{4}\Bigr)
        + \dP_{x_{-}}\Bigl(\bigl|R_{\tau_{*}(a)}\bigr|
        \ge\frac{\sqrt{a}}{4}\Bigr)\,.
    \end{equation}
    We estimate the three terms on the right-hand side of
    \eqref{eq:E4_complement_bound} in turn. For the first
    term, note that $X_{a}(s)\le x_{-}+\delta(a)/2$ for all
    $s<T_{\mathrm{up}}$, and hence
    $a-X_{a}(s)^{2}\le-\sqrt{a}\,\delta(a)$. Integrating
    \eqref{eq:Riccati_SDE} up to $T_{\mathrm{up}}$ gives
    \begin{equation*}
        x_{-}+\frac{\delta(a)}{2}
        = X_{a}(T_{\mathrm{up}})
        \le x_{-}-\sqrt{a}\,\delta(a)\,T_{\mathrm{up}}
        +B(T_{\mathrm{up}})\,,
    \end{equation*}
    so that
    \begin{equation*}
        \bigl\{T_{\mathrm{up}}\le\tau_{*}(a)\bigr\}
        \subset\Bigl\{\sup_{s\ge0}
        \bigl(B(s)-\sqrt{a}\,\delta(a)\,s\bigr)
        \ge\frac{\delta(a)}{2}\Bigr\}\,.
    \end{equation*}
    Applying Doob's maximal inequality to the exponential
    martingale
    $\exp\bigl(2\sqrt{a}\,\delta(a)B(s)-2a\,\delta(a)^{2}s\bigr)$,
    we obtain
    \begin{equation*}
        \dP_{x_{-}}\bigl(T_{\mathrm{up}}\le\tau_{*}(a)\bigr)
        \le \exp\bigl(-\sqrt{a}\,\delta(a)^{2}\bigr)
        = \exp\bigl(-(\log a)^{2}\bigr)\,.
    \end{equation*}
    For the second term in \eqref{eq:E4_complement_bound}, since $X_{a}(s)\le x_{-}+\delta(a)/2$ for
    $s<\zeta_{a}^{(1)}\wedge T_{\mathrm{up}}$, we have
    $\bigl|F_{2}'\bigl(X_{a}(s)\bigr)\bigr|
    \le1/\bigl(\sqrt{a}\,\delta(a)\bigr)$, and hence
    \begin{equation*}
        \bigl\langle M^{(2)}\bigr\rangle_{\tau_{*}(a)}
        = \int_{0}^{\tau_{*}(a)\wedge\zeta_{a}^{(1)}
        \wedge T_{\mathrm{up}}}
          F_{2}'\bigl(X_{a}(s)\bigr)^{2}\,ds
        \le \frac{\tau_{*}(a)}{a\,\delta(a)^{2}}\,.
    \end{equation*}
    The same argument as for $M^{(1)}$, applied to
    $M^{(2)}$ and $-M^{(2)}$, gives
    \begin{equation*}
        \dP_{x_{-}}\Bigl(\bigl|M^{(2)}_{\tau_{*}(a)}\bigr|
        \ge\frac{\sqrt{a}}{4}\Bigr)
        \le 2\exp\biggl(-\frac{a^{2}\,\delta(a)^{2}}
        {32\,\tau_{*}(a)}\biggr)\,,
    \end{equation*}
    whose exponent is of order $-a(\log a)^{2}$. For the third
    term in \eqref{eq:E4_complement_bound}, we first claim that for every $x\le x_{-}+\delta(a)/2$,
    \begin{equation*}
        0\le F_{2}''(x)\le\frac{2}{\sqrt{a}\,\delta(a)^{2}}
        =\frac{2}{(\log a)^{2}}\,.
    \end{equation*}
    Indeed, for such $x$ we have
    $|x|-\sqrt{a}\ge\delta(a)/2$ and
    $(|x|+\sqrt{a})^{2}\ge4\sqrt{a}\,|x|$, so that
    \begin{equation*}
        F_{2}''(x)
        = \frac{2|x|}
        {(|x|-\sqrt{a})^{2}\,(|x|+\sqrt{a})^{2}}
        \le \frac{2|x|}{(\delta(a)/2)^{2}\cdot4\sqrt{a}\,|x|}
        = \frac{2}{\sqrt{a}\,\delta(a)^{2}}\,.
    \end{equation*}
    Since $\tau(a)\to0$, we have $\tau_{*}(a)\le2\sqrt{a}$
    for all sufficiently large $a$. As the integrand of $R$
    equals $\tfrac{1}{2}F_{2}''\bigl(X_{a}(s)\bigr)$, it follows that
    \begin{equation*}
        \bigl|R_{\tau_{*}(a)}\bigr|
        \le \frac{\tau_{*}(a)}{(\log a)^{2}}
        \le \frac{2\sqrt{a}}{(\log a)^{2}}
        < \frac{\sqrt{a}}{4}\,,
    \end{equation*}
    so that
    $\dP_{x_{-}}\bigl(\bigl|R_{\tau_{*}(a)}\bigr|
    \ge\sqrt{a}/4\bigr)=0$ for all sufficiently large $a$.
    Combining the three estimates
    with~\eqref{eq:E4_complement_bound} yields
    $\dP_{x_{-}}\bigl(E_{4}^{c}\bigr)\to0$. By the inclusion
    $E_{4}\subset\bigl\{\zeta_{a}^{(1)}<\tau_{*}(a)\bigr\}$, the
    second limit in \eqref{eq:two_remaining_limits} follows, and
    the proof is complete.
\end{proof}

By \eqref{eq:eigenvalue_explosion} and
Proposition~\ref{prop:hitting_reduction},
Theorem~\ref{thm:sharp_lower_tail_lowest_eigenvalue} reduces
to showing that for every family $(\cL_{a})_{a>0}$ satisfying
\eqref{eq:admissible_window},
\begin{equation}\label{eq:hitting_probability_asymptotics}
    \dP_{x_{+}}\bigl(T_{x_{-}}<\cL_{a}\bigr)
    \sim \frac{\cL_{a}}{m(a)}\,.
\end{equation}
Indeed, $\bar{\cL}_{a}\sim\cL_{a}$ and $\bar{\cL}_{a}$ also
satisfies \eqref{eq:admissible_window}, so
\eqref{eq:hitting_probability_asymptotics} applies to both
sides of \eqref{eq:hitting_reduction}.
We now construct a renewal structure for $X_{a}$ started
from $x_{+}$. Set $\cS_{0}:=0$, fix $k\ge1$, and suppose that
$\cS_{k-1}$ has been defined and is an almost surely finite
stopping time. By the strong Markov property of Brownian motion,
\begin{equation*}
    B^{(k)}(t) := B(\cS_{k-1}+t)-B(\cS_{k-1})\,,\qquad t\ge0\,,
\end{equation*}
is a standard Brownian motion independent of
$\cF_{\cS_{k-1}}$. Let $X_{a}^{(k)}$ denote the solution of
\begin{equation*}
    dX_{a}^{(k)}(t)
    = \Bigl(a-\bigl(X_{a}^{(k)}(t)\bigr)^{2}\Bigr)\,dt
    +dB^{(k)}(t)\,,
    \qquad X_{a}^{(k)}(0)=x_{+}\,,
\end{equation*}
defined up to its explosion time. Let $I(a)$ be the
neighborhood
\begin{equation}\label{eq:def_interval_I}
    I(a) := \Bigl(x_{+}-\frac{\delta(a)}{2},\,
    x_{+}+\frac{\delta(a)}{2}\Bigr)
\end{equation}
of $x_{+}$, and set
\begin{equation}\label{eq:def_rho_sigma}
    \rho_{k} := \inf\bigl\{t\ge0\,:\,
    X_{a}^{(k)}(t)\notin I(a)\bigr\}\,,
    \qquad
    \sigma_{k} := \inf\bigl\{t\ge\rho_{k}\,:\,
    X_{a}^{(k)}(t)\in\{x_{-},x_{+}\}\bigr\}\,.
\end{equation}
We call the path of $X_{a}^{(k)}$ on $[0,\sigma_{k}]$ the
$k$-th attempt.
Since explosion to $-\infty$ requires first exiting
$I(a)$ and then hitting $x_{-}$, the time $\sigma_{k}$ is
bounded by the first explosion time of $X_{a}^{(k)}$. As the
latter is finite almost surely, we conclude that
$\sigma_{k}<\infty$ almost surely. We define
\begin{equation*}
    \cS_{k} := \cS_{k-1}+\sigma_{k}\,,
\end{equation*}
which is again an almost surely finite stopping time.
We set
\begin{equation*}
    \chi_{k} := \1_{\{X_{a}^{(k)}(\sigma_{k})=x_{-}\}}\,,
\end{equation*}
and call the $k$-th attempt \emph{successful} if
$\chi_{k}=1$ and \emph{failed} otherwise. This completes the
recursive construction. 
For every $k\ge1$, the pair $(\sigma_{k},\chi_{k})$ is the
same measurable functional of $B^{(k)}$. Since $B^{(k)}$ is
independent of $\cF_{\cS_{k-1}}$ and
$(\sigma_{j},\chi_{j})_{j<k}$ are
$\cF_{\cS_{k-1}}$-measurable, the pairs
$(\sigma_{k},\chi_{k})_{k\ge1}$ are i.i.d.
We now express the hitting event
$\bigl\{T_{x_{-}}<\cL_{a}\bigr\}$ in terms of the first
successful attempt. Let
\begin{equation*}
    N := \inf\bigl\{k\ge1\,:\,\chi_{k}=1\bigr\}
\end{equation*}
be the index of the first successful attempt.
Since the $\chi_{k}$ are i.i.d.\ with $\dP(\chi_{1}=1)>0$, the index
$N$ is geometrically distributed and in particular finite
almost surely. Every failed attempt ends at $x_{+}$, so under
$\dP_{x_{+}}$ we have $X_{a}(\cS_{k-1}+t)=X_{a}^{(k)}(t)$ for
$t\in[0,\sigma_{k}]$ and $k\le N$. Moreover, $X_{a}$ first hits $x_{-}$ at the
end of the first successful attempt, that is,
\begin{equation}\label{eq:hitting_time_renewal_identity}
    T_{x_{-}} = \cS_{N} = \sum_{k=1}^{N}\sigma_{k}\,.
\end{equation}
The estimate \eqref{eq:hitting_probability_asymptotics} is
thus reduced to the following proposition.

\begin{proposition}\label{prop:renewal_hitting_asymptotics}
    Let $(\cL_{a})_{a>0}$ satisfy \eqref{eq:admissible_window}.
    Then, as $a\to\infty$,
    \begin{equation}\label{eq:renewal_hitting_asymptotics}
        \dP_{x_{+}}\bigl(\cS_{N}<\cL_{a}\bigr)
        \sim \frac{\cL_{a}}{\dE_{x_{+}}\bigl[\cS_{N}\bigr]}
        \sim \frac{\cL_{a}}{m(a)}\,.
    \end{equation}
\end{proposition}

We briefly outline the proof.
The law of $(\sigma_{k},\chi_{k})$ varies with $a$, so
classical renewal theorems, which require a fixed increment
distribution, do not apply. We instead estimate all quantities
with explicit dependence on $a$.
We introduce
\begin{equation}\label{eq:def_counting_process}
    \cM(t) := \sum_{k\ge1}\1_{\{\cS_{k-1}<t\}}\,\chi_{k}\,,
    \qquad t\ge0\,,
\end{equation}
the number of successful attempts that start before time $t$.
Since $N$ is the index of the first successful attempt,
\begin{equation}\label{eq:success_counting_identity}
    \bigl\{\cS_{N-1}<\cL_{a}\bigr\}
    = \bigl\{\cM(\cL_{a})\ge1\bigr\}\,.
\end{equation}

Before using \eqref{eq:success_counting_identity}, we record an identity for the mean time to the first successful attempt. Although $\sigma_k$ and $\chi_k$ may be dependent, the event $\{N\ge k\}$ depends only on $(\chi_j)_{j<k}$ and is therefore independent of $\sigma_k$, since the pairs $(\sigma_k,\chi_k)_{k\ge1}$ are i.i.d. Wald's identity thus gives \begin{equation}\label{eq:wald_identity} \dE_{x_{+}}\bigl[\cS_N\bigr] = \dE_{x_{+}}[N]\,\dE_{x_{+}}[\sigma_1] = \frac{\mu_a}{p_a}\,, \end{equation} where $p_a:=\dP_{x_{+}}(\chi_1=1)$ and $\mu_a:=\dE_{x_{+}}[\sigma_1]$.

We now use \eqref{eq:success_counting_identity} to estimate $\dP_{x_{+}}(\cS_{N-1}<\cL_a)$. The advantage of this formulation is that $\{\cS_{k-1}<\cL_a\}$ depends only on the preceding attempts and is therefore independent of $\chi_k$. The argument proceeds in four steps. First, using the renewal structure of the i.i.d.\
pairs, Lemma~\ref{lem:mean_counting_asymptotics} establishes $\dE_{x_{+}}\bigl[\cM(\cL_{a})\bigr] \sim \frac{p_{a}}{\mu_{a}}\,\cL_{a}.$
We then show
that $\mu_{a}/p_{a}\sim m(a)$
(Lemma~\ref{lem:renewal_mean_asymptotics}). Next, we prove
that $\dP_{x_{+}}\bigl(\cM(\cL_{a})\ge1\bigr)$ is asymptotic
to $\dE_{x_{+}}\bigl[\cM(\cL_{a})\bigr]$
(Lemma~\ref{lem:successful_attempt_count}). Finally,
replacing $\cS_{N-1}$ by $\cS_{N}$ produces only a negligible
error (Lemma~\ref{lem:terminal_attempt_negligible}).
We begin with the first step, the mean of $\cM(\cL_{a})$.

\begin{lemma}\label{lem:mean_counting_asymptotics}
    Let $(\cL_{a})_{a>0}$ satisfy \eqref{eq:admissible_window}.
    Then, as $a\to\infty$,
    \begin{equation*}
        \dE_{x_{+}}\bigl[\cM(\cL_{a})\bigr]
        \sim \frac{p_{a}}{\mu_{a}}\,\cL_{a}
        = \frac{\cL_{a}}{\dE_{x_{+}}\bigl[\cS_{N}\bigr]}\,,
    \end{equation*}
    where $p_{a}$ and $\mu_{a}$ are as in \eqref{eq:wald_identity}.
\end{lemma}

\begin{proof}
    We define
    \begin{equation*}
        \cK(t) := \sum_{k\ge1}\1_{\{\cS_{k-1}<t\}}\,,
        \qquad t\ge0\,,
    \end{equation*}
    the number of attempts that start before time $t$. Since $\1_{\{\cS_{k-1}<t\}}$ is $\cF_{\cS_{k-1}}$-measurable and $\chi_{k}$ is independent of
    $\cF_{\cS_{k-1}}$ with the same law as $\chi_{1}$,
    \begin{equation}\label{eq:M_K_mean_identity}
        \dE_{x_{+}}\bigl[\cM(t)\bigr]
        = p_{a}\sum_{k\ge1}
        \dP_{x_{+}}\bigl(\cS_{k-1}<t\bigr)
        = p_{a}\,\dE_{x_{+}}\bigl[\cK(t)\bigr]\,,
        \qquad t\ge0\,.
    \end{equation}
    By \eqref{eq:M_K_mean_identity} with $t=\cL_{a}$ and Wald's identity
    \eqref{eq:wald_identity}, it is enough to prove that, as
    $a\to\infty$,
    \begin{equation}\label{eq:attempt_count_asymptotics}
        \dE_{x_{+}}\bigl[\cK(\cL_{a})\bigr]
        \sim \frac{\cL_{a}}{\mu_{a}}\,.
    \end{equation}
    Note that $\{\cK(\cL_{a})\le n\}=\{\cS_{n}\ge\cL_{a}\}$ for
    every $n\ge1$, so $\cK(\cL_{a})$ is a stopping time with respect to 
    the filtration $(\cG_{n})_{n\ge1}$ with
    $\cG_{n}:=\sigma(\sigma_{1},\ldots,\sigma_{n})$. Applying
    Wald's identity to the bounded stopping times
    $\cK(\cL_{a})\wedge m$ and letting $m\to\infty$, we obtain
    \begin{equation}\label{eq:wald_identity_K}
        \dE_{x_{+}}\bigl[\cS_{\cK(\cL_{a})}\bigr]
        = \mu_{a}\,\dE_{x_{+}}\bigl[\cK(\cL_{a})\bigr]\,.
    \end{equation}
    Let $\cR(\cL_{a})$ denote the overshoot,
    \begin{equation*}
        \cR(\cL_{a}) := \cS_{\cK(\cL_{a})}-\cL_{a}\ge0\,,
    \end{equation*}
    so that
    \begin{equation*}
        \dE_{x_{+}}\bigl[\cK(\cL_{a})\bigr]
        = \frac{\cL_{a}}{\mu_{a}}
        + \frac{\dE_{x_{+}}\bigl[\cR(\cL_{a})\bigr]}{\mu_{a}}\,.
    \end{equation*}
    We show that the second term is negligible, that is,
    \begin{equation}\label{eq:estimate_overshoot}
        \dE_{x_{+}}\bigl[\cR(\cL_{a})\bigr] = o(\cL_{a})\,.
    \end{equation}
    By Lorden's inequality \cite[Theorem~1]{Lor70},
    \begin{equation}\label{eq:Lorden_inequality}
        \dE_{x_{+}}\bigl[\cR(\cL_{a})\bigr]
        \le \frac{\dE_{x_{+}}\bigl[\sigma_{1}^{2}\bigr]}
        {\dE_{x_{+}}\bigl[\sigma_{1}\bigr]}\,.
    \end{equation}
    We estimate the two moments in turn. Under $\dP_{x_{+}}$, the first attempt follows $X_{a}$ itself, so the estimates of Appendix~\ref{apx:Auxiliary_estimates} apply to $(\rho_{1},\sigma_{1})$. By continuity, the diffusion exits
    $I(a)$ through one of its endpoints,
    \begin{equation*}
        X_{a}(\rho_{1})
        \in \Bigl\{x_{+}-\frac{\delta(a)}{2},\;
        x_{+}+\frac{\delta(a)}{2}\Bigr\}\,.
    \end{equation*}
    By the elementary bound $\sigma_{1}^{2}\le2\rho_{1}^{2}+2(\sigma_{1}-\rho_{1})^{2}$,
    the strong Markov property at $\rho_{1}$, and
    Lemmas~\ref{lem:moment_bounds_exit_time_rho}\,\textup{(ii)}
    and~\ref{lem:H_moment_bounds}\,\textup{(i),(ii)}, there exists $C>0$ such that
    \begin{equation*}
        \begin{aligned}
            \dE_{x_{+}}\bigl[\sigma_{1}^{2}\bigr]
            &\le 2\,\dE_{x_{+}}\bigl[\rho_{1}^{2}\bigr]
            + 2\max\Bigl\{
            \dE_{x_{+}-\delta(a)/2}\bigl[H^{2}\bigr],\,
            \dE_{x_{+}+\delta(a)/2}\bigl[H^{2}\bigr]\Bigr\}\\
            &\le C\exp\bigl(4(\log a)^{2}\bigr)\,
            a^{1/2}(\log a)^{2}\,.
        \end{aligned}
    \end{equation*}
    On the other hand, since $I(a)\subset(\sqrt{a},\infty)$,
    the drift $a-x^{2}$ is negative and the scale density
    $e^{2V}$ is increasing on $I(a)$. Since $x_{+}$ is the
    midpoint of $I(a)$, the exit distribution satisfies
    \begin{equation}\label{eq:downward_exit_probability}
        \dP_{x_{+}}\Bigl(X_{a}(\rho_{1})
        =x_{+}-\frac{\delta(a)}{2}\Bigr)\ge\frac{1}{2}\,.
    \end{equation}
    Hence, by
    Lemma~\ref{lem:H_moment_bounds}\,\textup{(iii)} and the
    strong Markov property at $\rho_{1}$,
    \begin{equation*}
        \dE_{x_{+}}[\sigma_{1}]
        \ge \dP_{x_{+}}\Bigl(X_{a}(\rho_{1})
        =x_{+}-\frac{\delta(a)}{2}\Bigr)\,
        \dE_{x_{+}-\delta(a)/2}[H]
        \ge \frac{\delta(a)^{2}}{16}\,.
    \end{equation*}
    Combining the last two estimates with Lorden's
    inequality~\eqref{eq:Lorden_inequality}, and recalling
    $\delta(a)^{2}=(\log a)^{2}/\sqrt{a}$, we obtain
    \begin{equation}\label{eq:sharp_estimate_overshoot}
        \dE_{x_{+}}\bigl[\cR(\cL_{a})\bigr]
        \le C'\exp\bigl(4(\log a)^{2}\bigr)\,a
        = o(\cL_{a})\,,
    \end{equation}
    where the last step follows from the lower bound in
    \eqref{eq:admissible_window}. This proves
    \eqref{eq:estimate_overshoot}, and hence
    \eqref{eq:attempt_count_asymptotics}, which completes the
    proof.
\end{proof}

We turn to the second step of the proof of
Proposition~\ref{prop:renewal_hitting_asymptotics}. The next
lemma shows that $\dE_{x_{+}}[\cS_{N}]$ is asymptotic to the
mean of the first explosion time.

\begin{lemma}\label{lem:renewal_mean_asymptotics}
    As $a\to\infty$,
    \begin{equation*}
        \dE_{x_{+}}\bigl[\cS_{N}\bigr] \sim m(a)\,.
    \end{equation*}
\end{lemma}

\begin{proof}
    By \eqref{eq:hitting_time_renewal_identity},
    $\cS_{N}=T_{x_{-}}$ under $\dP_{x_{+}}$, and by the strong
    Markov property at $T_{x_{-}}$,
    \begin{equation}\label{eq:explosion_time_decomposition}
        \dE_{x_{+}}\bigl[T_{x_{-}}\bigr]
        = \dE_{x_{+}}\bigl[\zeta_{a}^{(1)}\bigr]
        - \dE_{x_{-}}\bigl[\zeta_{a}^{(1)}\bigr]\,.
    \end{equation}
    For the first term, the explicit formula \cite[Eq.~(3.7)]{AD14} gives
    \begin{equation*}
        0 \le m(a)-\dE_{x_{+}}\bigl[\zeta_{a}^{(1)}\bigr]
        = 2\int_{x_{+}}^{\infty}\!dx\int_{x}^{\infty}\!du\,
          \exp\Bigl(2a(u-x)+\frac{2}{3}\bigl(x^{3}-u^{3}\bigr)\Bigr)\,.
    \end{equation*}
    Since $u^{3}-x^{3}\ge3x^{2}(u-x)$ for $u\ge x\ge0$, the exponent is at
    most $-2(x^{2}-a)(u-x)$. The inner integral is then at most
    $1/(2(x^{2}-a))$, so the difference is at most $\tau(a)$ by
    \eqref{eq:deterministic_time_scale}. As $\tau(a)\to0$,
    \begin{equation*}
        \dE_{x_{+}}\bigl[\zeta_{a}^{(1)}\bigr] \sim m(a)\,.
    \end{equation*}
    For the second term, Lemma~\ref{lem:mean_explosion_from_x_minus} gives
    $\dE_{x_{-}}\bigl[\zeta_{a}^{(1)}\bigr]\ll m(a)$.
    Combining the two estimates with
    \eqref{eq:explosion_time_decomposition} completes the proof.
\end{proof}

The third step shows that the probability of at least one
success is asymptotic to the mean number of successes.

\begin{lemma}\label{lem:successful_attempt_count}
    Let $(\cL_{a})_{a>0}$ satisfy \eqref{eq:admissible_window}.
    Then, as $a\to\infty$,
    \begin{equation}\label{eq:successful_attempt_count}
        \dP_{x_{+}}\bigl(\cM(\cL_{a})\ge1\bigr)
        \sim \dE_{x_{+}}\bigl[\cM(\cL_{a})\bigr]\,,
    \end{equation}
    where $\cM$ is defined in \eqref{eq:def_counting_process}.
\end{lemma}

\begin{proof}
    Set $\cA_{k}:=\bigl\{\cS_{k-1}<\cL_{a},\,\chi_{k}=1\bigr\}$. Then
    \begin{equation*}
        \cM(\cL_{a}) = \sum_{k\ge1}\1_{\cA_{k}}\,,
        \qquad
        \bigl\{\cM(\cL_{a})\ge1\bigr\}
        = \bigcup_{k\ge1}\cA_{k}\,.
    \end{equation*}
    By the Bonferroni inequalities, applied to the
    finite unions $\bigcup_{k\le n}\cA_{k}$ and letting
    $n\to\infty$,
    \begin{equation*}
        \sum_{k\ge1}\dP_{x_{+}}(\cA_{k})
        -\sum_{i<j}\dP_{x_{+}}\bigl(\cA_{i}\cap\cA_{j}\bigr)
        \le \dP_{x_{+}}\Bigl(\bigcup_{k\ge1}\cA_{k}\Bigr)
        \le \sum_{k\ge1}\dP_{x_{+}}(\cA_{k})\,.
    \end{equation*}
    Since $\sum_{k\ge1}\dP_{x_{+}}(\cA_{k})
    =\dE_{x_{+}}\bigl[\cM(\cL_{a})\bigr]$, it remains to show
    that
    \begin{equation}\label{eq:second_moment_negligible}
        \sum_{i<j}\dP_{x_{+}}\bigl(\cA_{i}\cap\cA_{j}\bigr)
        \ll \dE_{x_{+}}\bigl[\cM(\cL_{a})\bigr]\,.
    \end{equation}
    Let $g(t):=\dE_{x_{+}}[\cM(t)]$. Since
    $(\sigma_{j},\chi_{j})_{j>i}$ is independent of
    $\cF_{\cS_{i}}$ and has the same law as
    $(\sigma_{j},\chi_{j})_{j\ge1}$, the strong Markov
    property at $\cS_{i}$ gives
    \begin{equation*}
        \dE_{x_{+}}\Bigl[\sum_{j>i}\1_{\cA_{j}}\,\Big|\,
        \cF_{\cS_{i}}\Bigr]
        = \1_{\{\cS_{i}<\cL_{a}\}}\,
        g\bigl(\cL_{a}-\cS_{i}\bigr)\,,
    \end{equation*}
    and hence, since $\cA_{i}\in\cF_{\cS_{i}}$,
    \begin{equation*}
        \sum_{i<j}\dP_{x_{+}}\bigl(\cA_{i}\cap\cA_{j}\bigr)
        = \sum_{i\ge1}\dE_{x_{+}}\bigl[\1_{\cA_{i}}\,
          \1_{\{\cS_{i}<\cL_{a}\}}\,
          g\bigl(\cL_{a}-\cS_{i}\bigr)\bigr]\,.
    \end{equation*}
    Since $g$ is nondecreasing, on the event
    $\{\cS_{i}<\cL_{a}\}$ we have
    $g(\cL_{a}-\cS_{i})\le g(\cL_{a})$, so that
    \begin{equation*}
        \sum_{i<j}\dP_{x_{+}}\bigl(\cA_{i}\cap\cA_{j}\bigr)
        \le g(\cL_{a})
        \sum_{i\ge1}\dP_{x_{+}}(\cA_{i})
        = \dE_{x_{+}}\bigl[\cM(\cL_{a})\bigr]^{2}\,.
    \end{equation*}
    By Lemmas~\ref{lem:mean_counting_asymptotics}
    and~\ref{lem:renewal_mean_asymptotics} and the upper bound
    in \eqref{eq:admissible_window}, we have
    $\dE_{x_{+}}\bigl[\cM(\cL_{a})\bigr]
    \sim\cL_{a}/m(a)\to0$. Hence the right-hand side is
    negligible compared with
    $\dE_{x_{+}}\bigl[\cM(\cL_{a})\bigr]$, which proves
    \eqref{eq:second_moment_negligible}, and
    \eqref{eq:successful_attempt_count} follows.
\end{proof}

We conclude with the fourth step, the comparison of
$\cS_{N-1}$ and $\cS_{N}$. The events
$\{\cS_{N-1}<\cL_{a}\}$ and $\{\cS_{N}<\cL_{a}\}$ differ
precisely when the first successful attempt starts before
time $\cL_{a}$ but ends at or after it. The next
lemma shows that this contribution is negligible.

\begin{lemma}\label{lem:terminal_attempt_negligible}
    Let $(\cL_{a})_{a>0}$ satisfy \eqref{eq:admissible_window}.
    Then, as $a\to\infty$,
    \begin{equation}\label{eq:terminal_attempt_negligible}
        \dP_{x_{+}}\bigl(\cS_{N-1}<\cL_{a}\le\cS_{N}\bigr)
        \ll \dP_{x_{+}}\bigl(\cS_{N-1}<\cL_{a}\bigr)\,.
    \end{equation}
    Consequently,
    \begin{equation}\label{eq:S_N_replacement}
        \dP_{x_{+}}\bigl(\cS_{N-1}<\cL_{a}\bigr)
        \sim \dP_{x_{+}}\bigl(\cS_{N}<\cL_{a}\bigr)\,.
    \end{equation}
\end{lemma}

\begin{proof}
    Choose $\Delta(a)>0$ such that, as $a\to\infty$,
    \begin{equation}\label{eq:def_Delta}
        \exp\bigl(2(\log a)^{2}\bigr)\,a^{1/4}\log a
        \ll \Delta(a) \ll \cL_{a}\,.
    \end{equation}
    This is possible by the lower bound in
    \eqref{eq:admissible_window}. Splitting the event according to the position of $\cS_{N-1}$, we obtain
    \begin{equation}\label{eq:split_Delta}
        \begin{aligned}
            \dP_{x_{+}}\bigl(\cS_{N-1}<\cL_{a}\le\cS_{N}\bigr)
            &\le \dP_{x_{+}}\bigl(\cL_{a}-\Delta(a)
            \le\cS_{N-1}<\cL_{a}\bigr)\\
            &\quad+ \dP_{x_{+}}\bigl(\cS_{N-1}<\cL_{a}-\Delta(a),\,
            \sigma_{N}>\Delta(a)\bigr)\,.
        \end{aligned}
    \end{equation}
    For the first term on the right-hand side of \eqref{eq:split_Delta}, 
    if $\cL_{a}-\Delta(a)\le\cS_{N-1}<\cL_{a}$, then
    $\cM(\cL_{a})-\cM(\cL_{a}-\Delta(a))\ge1$. Hence, by Markov's inequality and
    \eqref{eq:M_K_mean_identity} together with Wald's
    identity~\eqref{eq:wald_identity_K} applied at $\cL_{a}$ and
    $\cL_{a}-\Delta(a)$,
    \begin{equation*}
        \begin{aligned}
            \dP_{x_{+}}\bigl(\cL_{a}-\Delta(a)
            \le\cS_{N-1}<\cL_{a}\bigr)
            &\le p_{a}\,\dE_{x_{+}}\bigl[\cK(\cL_{a})
            -\cK(\cL_{a}-\Delta(a))\bigr]\\
            &\le \frac{p_{a}}{\mu_{a}}
            \bigl(\Delta(a)
            +\dE_{x_{+}}\bigl[\cR(\cL_{a})\bigr]\bigr)\,.
        \end{aligned}
    \end{equation*}
    On the other hand, since $\{\cS_{N-1}<\cL_{a}\}=\{\cM(\cL_{a})\ge1\}$ by \eqref{eq:success_counting_identity}, Lemmas~\ref{lem:mean_counting_asymptotics}
    and~\ref{lem:successful_attempt_count} give
    \begin{equation}\label{eq:denominator_asymptotics}
        \dP_{x_{+}}\bigl(\cS_{N-1}<\cL_{a}\bigr)
        \sim \frac{p_{a}}{\mu_{a}}\,\cL_{a}\,.
    \end{equation}
    Since $\Delta(a)=o(\cL_{a})$ and
    $\dE_{x_{+}}\bigl[\cR(\cL_{a})\bigr]=o(\cL_{a})$ by
    \eqref{eq:sharp_estimate_overshoot}, the first term in \eqref{eq:split_Delta} is negligible compared with $\dP_{x_{+}}\bigl(\cS_{N-1}<\cL_{a}\bigr)$.
    For the second term in \eqref{eq:split_Delta}, 
    the union bound and the strong Markov property at the times $\cS_{k-1}$ give
    \begin{equation}\label{eq:terminal_tail_factorization}
        \begin{aligned}
            \dP_{x_{+}}\bigl(\cS_{N-1}<\cL_{a}-\Delta(a),\,
            \sigma_{N}>\Delta(a)\bigr)
            &\le \sum_{k\ge1}\dP_{x_{+}}
            \bigl(\cS_{k-1}<\cL_{a}-\Delta(a),\,
            \sigma_{k}>\Delta(a),\,\chi_{k}=1\bigr)\\
            &= \dP_{x_{+}}\bigl(\sigma_{1}>\Delta(a),\,
            \chi_{1}=1\bigr)\,
            \dE_{x_{+}}\bigl[\cK(\cL_{a}-\Delta(a))\bigr]\,.
        \end{aligned}
    \end{equation}
    The first factor in
    \eqref{eq:terminal_tail_factorization} equals
    $p_{a}\,\dP_{x_{+}}\bigl(\sigma_{1}>\Delta(a)\mid
    \chi_{1}=1\bigr)$. For the second factor, since $\cK$ is
    nondecreasing, \eqref{eq:attempt_count_asymptotics} gives
    \begin{equation*}
        \dE_{x_{+}}\bigl[\cK(\cL_{a}-\Delta(a))\bigr]
        \le \dE_{x_{+}}\bigl[\cK(\cL_{a})\bigr]
        \sim \frac{\cL_{a}}{\mu_{a}}\,.
    \end{equation*}
    Hence
    \begin{equation*}
        \dP_{x_{+}}\bigl(\cS_{N-1}<\cL_{a}-\Delta(a),\,
        \sigma_{N}>\Delta(a)\bigr)
        \le (1+o(1))\,
        \dP_{x_{+}}\bigl(\sigma_{1}>\Delta(a)
        \mid\chi_{1}=1\bigr)\,
        \frac{p_{a}}{\mu_{a}}\,\cL_{a}\,.
    \end{equation*}
    Comparing this bound with \eqref{eq:denominator_asymptotics}, we
    need only prove that, as $a\to\infty$,
    \begin{equation}\label{eq:successful_attempt_long_tail}
        \dP_{x_{+}}\bigl(\sigma_{1}>\Delta(a)
        \mid\chi_{1}=1\bigr)\to0\,.
    \end{equation}
    Since $\sigma_{1}=\rho_{1}+(\sigma_{1}-\rho_{1})$, the union bound gives
    \begin{equation*}
        \dP_{x_{+}}\bigl(\sigma_{1}>\Delta(a)\mid\chi_{1}=1\bigr)
        \le \dP_{x_{+}}\Bigl(\rho_{1}>\frac{\Delta(a)}{2}
        \,\Big|\,\chi_{1}=1\Bigr)
        + \dP_{x_{+}}\Bigl(\sigma_{1}-\rho_{1}
        >\frac{\Delta(a)}{2}\,\Big|\,\chi_{1}=1\Bigr)\,.
    \end{equation*}
    A successful attempt exits $I(a)$ through its lower endpoint
    $x_{+}-\delta(a)/2$, since a path from the upper endpoint
    must return to $x_{+}$ before reaching $x_{-}$. Hence, by the
    strong Markov property at $\rho_{1}$,
    \begin{equation*}
        \dP_{x_{+}}\Bigl(\rho_{1}>\frac{\Delta(a)}{2},\,
        \chi_{1}=1\Bigr)
        = \dP_{x_{+}}\Bigl(\rho_{1}>\frac{\Delta(a)}{2},\,
        X_{a}(\rho_{1})=x_{+}-\frac{\delta(a)}{2}\Bigr)\,
        \dP_{x_{+}-\delta(a)/2}\bigl(T_{x_{-}}<T_{x_{+}}\bigr)
    \end{equation*}
    and
    \begin{equation*}
        \dP_{x_{+}}\bigl(\chi_{1}=1\bigr)
        = \dP_{x_{+}}\Bigl(X_{a}(\rho_{1})
        =x_{+}-\frac{\delta(a)}{2}\Bigr)\,
        \dP_{x_{+}-\delta(a)/2}\bigl(T_{x_{-}}<T_{x_{+}}\bigr)\,,
    \end{equation*}
    so the hitting probability cancels in the ratio. By Markov's
    inequality,
    Lemma~\ref{lem:moment_bounds_exit_time_rho}\,\textup{(i)},
    and \eqref{eq:downward_exit_probability},
    \begin{equation}\label{eq:rho_cond_tail}
        \begin{aligned}
            \dP_{x_{+}}\Bigl(\rho_{1}>\frac{\Delta(a)}{2}
            \,\Big|\,\chi_{1}=1\Bigr)
            &= \dP_{x_{+}}\Bigl(\rho_{1}>\frac{\Delta(a)}{2}
            \,\Big|\,
            X_{a}(\rho_{1})=x_{+}-\frac{\delta(a)}{2}\Bigr)\\
            &\le \frac{2}{\Delta(a)}\,
            \frac{\dE_{x_{+}}[\rho_{1}]}
            {\dP_{x_{+}}\bigl(X_{a}(\rho_{1})
            =x_{+}-\frac{\delta(a)}{2}\bigr)}\\
            &\le \frac{2}{\Delta(a)\sqrt{a}}
            \to 0\,.
        \end{aligned}
    \end{equation}
    Similarly, by the strong Markov property at $\rho_{1}$ and
    the same cancellation,
    \begin{equation}\label{eq:post_exit_cond_tail}
        \begin{aligned}
            \dP_{x_{+}}\Bigl(\sigma_{1}-\rho_{1}
            >\frac{\Delta(a)}{2}\,\Big|\,\chi_{1}=1\Bigr)
            &= \dP_{x_{+}-\delta(a)/2}\Bigl(H>\frac{\Delta(a)}{2}
            \,\Big|\,T_{x_{-}}<T_{x_{+}}\Bigr)\\
            &\le \frac{2}{\Delta(a)}\,
            \dE_{x_{+}-\delta(a)/2}\bigl[H\mid
            T_{x_{-}}<T_{x_{+}}\bigr]\\
            &\le \frac{C}{\Delta(a)}\,
            \exp\bigl(2(\log a)^{2}\bigr)\,a^{1/4}\log a
            \to 0\,,
        \end{aligned}
    \end{equation}
    where the last inequality follows from
    Lemma~\ref{lem:H_moment_bounds}\,\textup{(iv)}, and the convergence 
    from the choice of $\Delta(a)$ in
    \eqref{eq:def_Delta}. Combining \eqref{eq:rho_cond_tail} and
    \eqref{eq:post_exit_cond_tail} yields
    \eqref{eq:successful_attempt_long_tail}, so the second term
    in \eqref{eq:split_Delta} is negligible as well. This proves
    \eqref{eq:terminal_attempt_negligible}. Since
    $\{\cS_{N}<\cL_{a}\}\subset\{\cS_{N-1}<\cL_{a}\}$,
    \begin{equation*}
        \dP_{x_{+}}\bigl(\cS_{N-1}<\cL_{a}\bigr)
        -\dP_{x_{+}}\bigl(\cS_{N}<\cL_{a}\bigr)
        = \dP_{x_{+}}\bigl(\cS_{N-1}<\cL_{a}\le\cS_{N}\bigr)\,,
    \end{equation*}
    so \eqref{eq:S_N_replacement} follows.
\end{proof}

We now combine the four lemmas to prove the proposition.

\begin{proof}[Proof of
Proposition~\ref{prop:renewal_hitting_asymptotics}]
    Combining \eqref{eq:success_counting_identity} with
    Lemmas~\ref{lem:mean_counting_asymptotics},
    \ref{lem:successful_attempt_count},
    and~\ref{lem:terminal_attempt_negligible} gives
    \begin{equation*}
        \dP_{x_{+}}\bigl(\cS_{N}<\cL_{a}\bigr)
        \sim \dP_{x_{+}}\bigl(\cM(\cL_{a})\ge1\bigr)
        \sim \frac{\cL_{a}}{\dE_{x_{+}}\bigl[\cS_{N}\bigr]}
        \sim \frac{\cL_{a}}{m(a)}\,,
    \end{equation*}
    where the last equivalence follows from
    Lemma~\ref{lem:renewal_mean_asymptotics}. This proves
    \eqref{eq:renewal_hitting_asymptotics}.
\end{proof}

We are now ready to prove
Theorem~\ref{thm:sharp_lower_tail_lowest_eigenvalue} by combining
Propositions~\ref{prop:hitting_reduction}
and~\ref{prop:renewal_hitting_asymptotics}.

\begin{proof}[Proof of
Theorem~\ref{thm:sharp_lower_tail_lowest_eigenvalue}]
    By \eqref{eq:eigenvalue_explosion} with $k=1$, it suffices to
    show that
    $\dP\bigl(\zeta_{a}^{(1)}<\cL_{a}\bigr)\sim\cL_{a}/m(a)$.
    Proposition~\ref{prop:hitting_reduction} gives
    \begin{equation*}
        (1-o(1))\,\dP_{x_{+}}\bigl(T_{x_{-}}<\bar{\cL}_{a}\bigr)
        \le \dP\bigl(\zeta_{a}^{(1)}<\cL_{a}\bigr)
        \le \dP_{x_{+}}\bigl(T_{x_{-}}<\cL_{a}\bigr)\,.
    \end{equation*}
    Since $T_{x_{-}}=\cS_{N}$ under $\dP_{x_{+}}$ by
    \eqref{eq:hitting_time_renewal_identity},
    Proposition~\ref{prop:renewal_hitting_asymptotics} establishes
    \eqref{eq:hitting_probability_asymptotics}, and we apply it to
    each side separately. The right-hand side is asymptotic to
    $\cL_{a}/m(a)$, since $(\cL_{a})_{a>0}$ satisfies
    \eqref{eq:admissible_window} by assumption. 
    For the left-hand side, recall that
    $\bar{\cL}_{a}=\cL_{a}-2\tau(a)-2\sqrt{a}$. Since
    $2\tau(a)+2\sqrt{a}=o(\cL_{a})$ under
    \eqref{eq:admissible_window}, $\bar{\cL}_{a}\sim\cL_{a}$
    and $(\bar{\cL}_{a})_{a>0}$ satisfies
    \eqref{eq:admissible_window} as well. Hence
    \begin{equation*}
        \dP_{x_{+}}\bigl(T_{x_{-}}<\bar{\cL}_{a}\bigr)
        \sim \frac{\bar{\cL}_{a}}{m(a)}
        \sim \frac{\cL_{a}}{m(a)}\,.
    \end{equation*}
    The two bounds thus agree to leading order, which proves the
    first equivalence in
    \eqref{eq:lower_tail_asymptotics}. The second follows
    from \eqref{eq:first_explosion_mean_asymptotics}.
\end{proof}

Theorem~\ref{thm:sharp_lower_tail_lowest_eigenvalue} requires
the window \eqref{eq:admissible_window}. The next lemma and
its corollary give upper bounds that hold for every $\cL>0$
and every $a>0$, with no restriction between the two. The
proof of the lemma is a Laplace transform argument of McKean
\cite{McK94}, in the form of the fixed-point equation
\cite[Eq.~(3.6)]{AD14}.

\begin{lemma}\label{lem:left_tail_upper_bound_fixed_L}
    For every $\cL>0$ and every $a>0$,
    \begin{equation*}
        \dP\bigl(\lambda_{1}(\cL)<-a\bigr)
        \le e\,\frac{\cL}{m(a)}\,.
    \end{equation*}
\end{lemma}

\begin{proof}
    By \eqref{eq:eigenvalue_explosion} with $k=1$,
    $\dP\bigl(\lambda_{1}(\cL)<-a\bigr)
    =\dP\bigl(\zeta_{a}^{(1)}<\cL\bigr)$. For $\beta>0$,
    define
    \begin{equation*}
        g_{\beta}(y)
        := \dE_{y}\bigl[\exp\bigl(-\beta\zeta_{a}^{(1)}\bigr)
        \bigr]\,.
    \end{equation*}
    Since $\zeta_{a}^{(1)}$ is nondecreasing in the starting
    point, $g_{\beta}$ is nonincreasing and the limit
    $g_{\beta}(\infty):=\lim_{y\to\infty}g_{\beta}(y)$ exists.
    By monotone convergence, $g_{\beta}(\infty)$ is the Laplace
    transform of $\zeta_{a}^{(1)}$ for the diffusion started from
    $+\infty$. Recall from \cite[Eq.~(3.6)]{AD14} the fixed-point equation
    \begin{equation*}
        g_{\beta}(y)
        = 1-2\beta\int_{-\infty}^{y}\!dx\int_{x}^{\infty}\!du\,
        \exp\Bigl(2a(u-x)
        +\frac{2}{3}\bigl(x^{3}-u^{3}\bigr)\Bigr)\,
        g_{\beta}(u)\,.
    \end{equation*}
    Letting $y\to\infty$ and using
    $g_{\beta}(u)\ge g_{\beta}(\infty)$ together with
    \cite[Eqs.~(3.7) and (3.8)]{AD14}, we obtain
    \begin{equation*}
        g_{\beta}(\infty)
        \le 1-\beta\,m(a)\,g_{\beta}(\infty)\,,
    \end{equation*}
    and hence
    \begin{equation*}
        g_{\beta}(\infty) \le \frac{1}{1+\beta\,m(a)}\,.
    \end{equation*}
    Taking $\beta=1/\cL$ and applying Markov's inequality, we have
    \begin{equation*}
        \dP\bigl(\zeta_{a}^{(1)}<\cL\bigr)
        = \dP\bigl(\exp\bigl(-\zeta_{a}^{(1)}/\cL\bigr)
        >e^{-1}\bigr)
        \le e\,g_{1/\cL}(\infty)
        \le \frac{e}{1+m(a)/\cL}
        \le e\,\frac{\cL}{m(a)}\,.
    \end{equation*}
\end{proof}

\begin{corollary}\label{cor:second_eigenvalue_left_tail}
    For every $\cL>0$ and every $a>0$,
    \begin{equation*}
        \dP\bigl(\lambda_{2}(\cL)<-a\bigr)
        \le \Bigl(e\,\frac{\cL}{m(a)}\Bigr)^{2}\,.
    \end{equation*}
\end{corollary}

\begin{proof}
    By \eqref{eq:eigenvalue_explosion} with $k=2$,
    $\dP\bigl(\lambda_{2}(\cL)<-a\bigr)
    =\dP\bigl(\zeta_{a}^{(2)}<\cL\bigr)$.
    Recall that
    $\zeta_{a}^{(1)}$ and $\zeta_{a}^{(2)}-\zeta_{a}^{(1)}$ are
    i.i.d.\ by the strong Markov property. Since
    $\{\zeta_{a}^{(2)}<\cL\}\subset
    \{\zeta_{a}^{(1)}<\cL\}\cap
    \{\zeta_{a}^{(2)}-\zeta_{a}^{(1)}<\cL\}$, we have
    \begin{equation*}
        \dP\bigl(\zeta_{a}^{(2)}<\cL\bigr)
        \le \dP\bigl(\zeta_{a}^{(1)}<\cL\bigr)^{2}
        \le \Bigl(e\,\frac{\cL}{m(a)}\Bigr)^{2}\,,
    \end{equation*}
    where the last inequality follows from
    Lemma~\ref{lem:left_tail_upper_bound_fixed_L}.
\end{proof}

\subsection{Concentration of the \texorpdfstring{$L^{1}$}{L1} norm of
the eigenfunction \texorpdfstring{$\varphi_{1}$}{phi1}}
\label{subsec:eigenfunction_concentration}
We close this section with the second main result, the
concentration of $\|\varphi_{1}\|_{1}$ on the rare event of
Theorem~\ref{thm:sharp_lower_tail_lowest_eigenvalue}.
Since the localization theory of \cite{DL20} describes the
bottom of the spectrum on events of probability tending to
one, it gives no information conditionally on the event
$\{\lambda_{1}<-a\}$, whose probability vanishes
under \eqref{eq:concentration_window} below.
Conditioning on $\{\lambda_{1}<-a\}$ also changes
the picture of the spectrum. Without the conditioning,
the gap
$\lambda_{2}-\lambda_{1}$ is of order $|\lambda_{1}|^{-1/2}$
\cite[Theorem~1]{DL20}, which at $\lambda_{1}=-a$
would be of order $a^{-1/2}$. In contrast, the
following theorem shows
that, conditionally on $\{\lambda_{1}<-a\}$, the bound
$\lambda_{2}\ge-(1-\theta)a$ holds with probability tending
to one, so the gap is at least $\theta a$. It also
gives the concentration of $a^{1/4}\|\varphi_{1}\|_{1}$ on
the same event.
\begin{theorem}\label{thm:eigenfunction_concentration}
    Let $(\cL_{a})_{a>0}$ satisfy, for some $\eta\in(0,1)$,
    \begin{equation}\label{eq:concentration_window}
        a\exp\bigl(4(\log a)^{2}\bigr) \ll \cL_{a}
        \qquad\text{and}\qquad
        \limsup_{a\to\infty}
        \frac{\log\cL_{a}}{\tfrac{8}{3}\,a^{3/2}} \le 1-\eta\,.
    \end{equation}
    Fix $\theta\in\bigl(0,\,1-(1-\eta/2)^{2/3}\bigr)$
    and $\ve>0$. Then, as $a\to\infty$,
    \begin{equation}\label{eq:eigenfunction_concentration}
        \dP\bigl(\lambda_{1}(\cL_{a})<-a\bigr)
        \sim\dP\bigl(\lambda_{1}(\cL_{a})<-a\,,\
        \lambda_{2}(\cL_{a})\ge-(1-\theta)a\,,\
        \bigl|a^{1/4}\|\varphi_{1}\|_{1}
        -\pi/\sqrt{2} \bigr|\le\ve\bigr)\,.
    \end{equation}
\end{theorem}

\begin{remark}\label{rem:conditions_and_constant}
We comment on the conditions of
Theorem~\ref{thm:eigenfunction_concentration}, the
normalization $a^{1/4}$, and the constant $\pi/\sqrt{2}$.
\begin{enumerate}
    \item[\textup{(i)}] The lower bound on $\cL_{a}$ is the same as
    the lower bound in \eqref{eq:admissible_window}. As noted in
    Remark~\ref{rem:admissible_window}, it is technical.
        \item[\textup{(ii)}] The upper bound on $\cL_{a}$ and the
    condition on $\theta$ are used for the second event of
    \eqref{eq:eigenfunction_concentration}. By
    Corollary~\ref{cor:second_eigenvalue_left_tail} and
    \eqref{eq:first_explosion_mean_asymptotics},
    $\dP\bigl(\lambda_{2}(\cL_{a})<-(1-\theta)a\bigr)$ is negligible
    compared with $\dP\bigl(\lambda_{1}(\cL_{a})<-a\bigr)$ whenever
    \begin{equation*}
        \limsup_{a\to\infty}\frac{\log\cL_{a}}{a^{3/2}}
        < \frac{8}{3}\bigl(2(1-\theta)^{3/2}-1\bigr)\,,
    \end{equation*}
    and the second condition in
    \eqref{eq:concentration_window}, together with the range of
    $\theta$, ensures this inequality (see the proof of
    Theorem~\ref{thm:eigenfunction_concentration}).
    Combined with the first event, this gives the gap
    $\lambda_{2}-\lambda_{1}>\theta a$.
    \item[\textup{(iii)}]  By the variational characterization
\eqref{eq:minimax_characterization}, the event
$\{\lambda_{1}(\cL)<-a\}$ occurs if and only if
there exists $f\in H_{0}^{1}(0,\cL)$ with $\|f\|_{2}=1$
such that $-\langle\xi,f^{2}\rangle>a+\|f'\|_{2}^{2}.$ For a fixed deterministic $f$, the pairing $-\langle\xi,f^{2}\rangle$ is a centered Gaussian with variance $\|f\|_{4}^{4}$. Hence the probability of this inequality is at most $\exp(-J_{a}(f))$, where \begin{equation}\label{eq:def_J_functional} J_{a}(f) := \frac{\bigl(a+\|f'\|_{2}^{2}\bigr)^{2}} {2\,\|f\|_{4}^{4}}\,. \end{equation} 
Thus $J_a(f)$ describes the Gaussian cost associated with the above inequality for the fixed profile $f$. This suggests identifying the most favorable profiles by minimizing $J_a$.  For the corresponding variational problem on $\bR$, the infimum over all $L^{2}$-normalized functions in $H^{1}(\bR)$ equals $\frac{8}{3}a^{3/2}$, and the minimizers are, up to sign, the translates of $t\mapsto a^{1/4}\cQ(\sqrt{a}\,t)$, where $\cQ:=\frac{1}{\sqrt{2}}\operatorname{sech}$ (see \cite[Theorem~A.1]{Fra14} and the proof of Lemma~\ref{lem:variational_gap}). Every minimizer $f$ satisfies 
\[ a^{1/4}\|f\|_{1} = \int_{\bR}\cQ(y)\,dy = \frac{\pi}{\sqrt{2}}\,. \]
    
The minimizing profiles therefore predict the value
$\pi/\sqrt{2}$ for the rescaled $L^{1}$ norm of the ground
state conditional on $\{\lambda_{1}(\cL)<-a\}$.
The third event of \eqref{eq:eigenfunction_concentration}
makes this prediction precise. The same variational structure governs the ground state of Hill's operator with white noise potential on a fixed circle, through the exact formula of \cite{CM99} and the Laplace analysis of \cite{CRR06}. The value is also consistent with the shape of $\varphi_{1}$ in the typical regime \cite[Theorem~2]{DL20}. The minimizing profile $\cQ$ also agrees with the rescaled eigenfunction profile obtained in the typical regime \cite[Theorem~2]{DL20}.

\end{enumerate}
\end{remark}

The proof of Theorem~\ref{thm:eigenfunction_concentration}
consists of two estimates. The first bounds the
probability that the spectral gap fails and is provided
by Corollary~\ref{cor:second_eigenvalue_left_tail}. 
The following proposition provides the second, a bound on the
probability that $a^{1/4}\|\varphi_{1}\|_{1}$ is far from
$\pi/\sqrt{2}$ while the gap holds. The bound holds for every
$\cL$ and all sufficiently large $a$, with no restriction
between the two.
\begin{proposition}\label{prop:eigenfunction_L1_norm}
    Fix $\theta\in(0,1)$ and $\ve>0$. There exist
    $c_{1},c_{2},a_{0}>0$, depending only on $\theta$ and $\ve$,
    such that for every $\cL>0$ and every $a\ge a_{0}$,
    \begin{equation}\label{eq:eigenfunction_L1_norm_bound}
        \begin{aligned}
            \dP\Bigl(\lambda_{1}(\cL)<-a\,,\
            \lambda_{2}(\cL)&\ge-(1-\theta)a\,,\
            \bigl|a^{1/4}\|\varphi_{1}\|_{1}
            -\pi/\sqrt{2}\bigr|>\ve\Bigr) \\
            &\le (\cL\sqrt{a}+2)
            \exp\Bigl(-\Bigl(\frac{8}{3}+c_{1}\Bigr)a^{3/2}
            +c_{2}\,a^{3/4}\Bigr)\,.
        \end{aligned}
    \end{equation}
\end{proposition}

We outline the proof. In view of the Gaussian estimate
for a fixed function in
Remark~\ref{rem:conditions_and_constant}\,\textup{(iii)},
we seek a uniform increase over the minimum cost
$\frac{8}{3}a^{3/2}$ for the profiles that can occur
on the event in
\eqref{eq:eigenfunction_L1_norm_bound}.
Such an increase will provide the additional
exponential decay in the desired probability bound.
There are two difficulties in implementing this idea.

First, the global $L^{1}$ deviation alone need not
yield a uniform increase in the cost.
A tail of small amplitude spread over a long interval
can change the $L^{1}$ norm significantly while
having little effect on $J_a$. We use the spectral gap to control this tail contribution. More precisely,  let $\fm_{\varphi_{1}}$ be the expectation of $t$
under the probability measure $\varphi_{1}^{2}\,dt$,
\begin{equation}\label{eq:def_center_of_mass}
    \fm_{\varphi_{1}}
    := \int_{0}^{\cL}t\,\varphi_{1}(t)^{2}\,dt\,.
\end{equation}
On the event in Proposition~\ref{prop:eigenfunction_L1_norm},
the spectral gap is at least $\theta a$.
Lemma~\ref{lem:spectral_gap_concentration} then
controls the $L^{1}$ mass away from
$\fm_{\varphi_{1}}$.
This allows us to localize the global $L^{1}$
deviation to an interval of length of order
$a^{-1/2}$.

Second, the Gaussian bound $\exp(-J_a(f))$
applies to a fixed deterministic function $f$,
whereas $\varphi_{1}$ depends on $\xi$.
We must therefore control all possible profiles,
including their locations.
We cover the possible localization centers by
a deterministic mesh of scale $a^{-1/2}$.
For each mesh point, we consider the deterministic
class of functions satisfying the localization
and local $L^{1}$ constraints above.
Lemma~\ref{lem:global_to_local} shows that
$\varphi_{1}$ belongs to one of these classes
on the event under consideration.

On each class, the characterization of the minimizers
and a compactness argument yield
\[
    J_a(f)\ge
    \left(\frac{8}{3}+c_{1}\right)a^{3/2}
\]
for some $c_{1}>0$
(Lemma~\ref{lem:variational_gap}).
We then apply the Borell--TIS inequality to the
Gaussian supremum over each class, using this
cost gap and a bound on the expected supremum.
Finally, a union bound over the mesh points gives
\eqref{eq:eigenfunction_L1_norm_bound},
with the prefactor $\cL\sqrt{a}+2$.

The following lemma is the Poincar\'e inequality for
$\varphi_{1}^{2}\,dt$. For Schr\"odinger operators with Kato
class potentials, it is the variational representation of
the spectral gap in \cite[Corollary~1.3]{KS87}. For general
symmetric diffusion operators, we refer to
\cite[Section~4.2]{BGL14}. Since $\xi$ is a distribution,
these results do not apply directly, and we verify the
inequality through the form \eqref{eq:FN_form}.
\begin{lemma}\label{lem:spectral_gap_concentration}
    Almost surely, for every Lipschitz function $g$
    on $[0,\cL]$ with $\int_{0}^{\cL}g\,\varphi_{1}^{2}\,dt=0$,
    \begin{equation}\label{eq:poincare_ground_state}
        \bigl(\lambda_{2}(\cL)-\lambda_{1}(\cL)\bigr)
        \int_{0}^{\cL}g^{2}\varphi_{1}^{2}\,dt
        \le \int_{0}^{\cL}(g')^{2}\varphi_{1}^{2}\,dt\,.
    \end{equation}
    In particular,
    \begin{equation}\label{eq:gap_localization}
        \int_{0}^{\cL}\bigl(t-\fm_{\varphi_{1}}\bigr)^{2}
        \varphi_{1}(t)^{2}\,dt
        \le \frac{1}{\lambda_{2}(\cL)-\lambda_{1}(\cL)}\,,
    \end{equation}
    and, for every $r>0$,
    \begin{equation}\label{eq:L1_tail}
        \int_{\{|t-\fm_{\varphi_{1}}|>r\}}
        \bigl|\varphi_{1}(t)\bigr|\,dt
        \le \biggl(\frac{2}
        {\bigl(\lambda_{2}(\cL)-\lambda_{1}(\cL)\bigr)\,r}
        \biggr)^{1/2}\,.
    \end{equation}
\end{lemma}
\begin{proof}
    The eigenvalues are almost surely simple, so
    $\lambda_{2}(\cL)-\lambda_{1}(\cL)>0$.
    Fix such a realization and a Lipschitz function
    $g$ on $[0,\cL]$ with $\int_{0}^{\cL}g\,\varphi_{1}^{2}\,dt=0$.
    Since $g$ is Lipschitz and
    $\varphi_{1}\in H_{0}^{1}(0,\cL)$, both $g\varphi_{1}$ and
    $g^{2}\varphi_{1}$ belong to $H_{0}^{1}(0,\cL)$. A
    computation gives
    \begin{equation}\label{eq:IMS_identity}
        \cE_{\cL}\bigl(g\varphi_{1},\,g\varphi_{1}\bigr)
        = \cE_{\cL}\bigl(\varphi_{1},\,g^{2}\varphi_{1}\bigr)
        + \int_{0}^{\cL}|g'|^{2}\varphi_{1}^{2}\,dt\,.
    \end{equation}
    The noise term of \eqref{eq:FN_form} contributes
    equally to both sides, since it depends only on the
    product of its arguments, which equals
    $g^{2}\varphi_{1}^{2}$ on both sides.
    By \eqref{eq:weak_eigen_relation} with the test function
    $g^{2}\varphi_{1}$, the first term on the right-hand side of
    \eqref{eq:IMS_identity} equals
    $\lambda_{1}(\cL)\int_{0}^{\cL}g^{2}\varphi_{1}^{2}\,dt$.
    Moreover, $g\varphi_{1}\perp\varphi_{1}$ by the
    assumption on $g$, so
    \eqref{eq:minimax_characterization} gives
    \begin{equation*}
        \cE_{\cL}\bigl(g\varphi_{1},\,g\varphi_{1}\bigr)
        \ge \lambda_{2}(\cL)\int_{0}^{\cL}g^{2}
        \varphi_{1}^{2}\,dt\,.
    \end{equation*}
    Combining these two facts with
    \eqref{eq:IMS_identity} yields
    \eqref{eq:poincare_ground_state}.

    For \eqref{eq:gap_localization}, take
    $g(t):=t-\fm_{\varphi_{1}}$. Then $g$ is Lipschitz with
    $g'\equiv1$, and $\int_{0}^{\cL}g\,\varphi_{1}^{2}\,dt=0$
    by \eqref{eq:def_center_of_mass} and $\|\varphi_{1}\|_{2}=1$.
    The right-hand side of \eqref{eq:poincare_ground_state}
    equals $\|\varphi_{1}\|_{2}^{2}=1$, which proves
    \eqref{eq:gap_localization}.
    For \eqref{eq:L1_tail}, the Cauchy--Schwarz inequality gives
    \begin{equation*}
        \int_{\{|t-\fm_{\varphi_{1}}|>r\}}
        \bigl|\varphi_{1}(t)\bigr|\,dt
        \le \biggl(\int_{0}^{\cL}
        g^{2}\varphi_{1}^{2}\,dt\biggr)^{1/2}
        \biggl(\int_{\{|t-\fm_{\varphi_{1}}|>r\}}
        \frac{dt}{(t-\fm_{\varphi_{1}})^{2}}\biggr)^{1/2}\,,
    \end{equation*}
    where the first factor is bounded by
    \eqref{eq:gap_localization} and the second is at most
    $(2/r)^{1/2}$ by extending the integral to $\bR$.
\end{proof}

We now reduce the proof of
Proposition~\ref{prop:eigenfunction_L1_norm} to
deterministic classes of functions, using
Lemma~\ref{lem:spectral_gap_concentration}. Each class
consists of the functions that share, around a fixed
point, the properties of $\varphi_{1}$ on the event of the
proposition. Lemma~\ref{lem:global_to_local} below shows
that $\varphi_{1}$ itself belongs to one of the classes.
Fix
$\theta\in(0,1)$ and $\ve>0$. Set $C_{\theta}:=2/\theta+2$
and choose $R>1$ so large that
\begin{equation}\label{eq:choice_R}
    \sqrt{\frac{2}{\theta}}\,(R-1)^{-1/2} < \frac{\ve}{4}
    \qquad\text{and}\qquad
    \sup_{|z|\le\sqrt{C_{\theta}}}
    \biggl|\int_{-R}^{R}\cQ(y-z)\,dy
    - \frac{\pi}{\sqrt{2}}\biggr|
    < \frac{\ve}{4}\,.
\end{equation}
The first condition will control the $L^{1}$ norm of $\varphi_{1}$
outside an interval around $\fm_{\varphi_{1}}$,
and the second the $L^{1}$ norm of a minimizer inside it.
Both conditions hold once $R$ is large enough,
since $\cQ$ is positive and integrable.
We define the mesh
\begin{equation}\label{eq:def_mesh}
    \cU_{a} := \bigl(a^{-1/2}\bZ\cap[0,\cL]\bigr)\cup\{\cL\}\,,
\end{equation}
so that $|\cU_{a}| \le \cL\sqrt{a}+2$ and every point of $[0,\cL]$ lies within distance
$a^{-1/2}$ of $\cU_{a}$. For $U\in\cU_{a}$, define the interval
\begin{equation*}
    I_{U,R} := \bigl[U-Ra^{-1/2},\,U+Ra^{-1/2}\bigr]
    \cap[0,\cL]
\end{equation*}
and the class
\begin{equation}\label{eq:def_bad_class}
\begin{aligned}
    \cB_{a,\ve,R}(U) := \Bigl\{f\in H_{0}^{1}(0,\cL)\,:\,
    &f\ge0\,,\ \|f\|_{2}=1\,,\
    \int_{0}^{\cL}(t-U)^{2}f(t)^{2}\,dt
    \le\frac{C_{\theta}}{a}\,,\\
    &\Bigl|a^{1/4}\int_{I_{U,R}}f(t)\,dt
    -\frac{\pi}{\sqrt{2}}\Bigr|>\frac{\ve}{2}\Bigr\}\,.
\end{aligned}
\end{equation}
The first three conditions are properties of $\varphi_{1}$,
with $\fm_{\varphi_{1}}$ replaced by $U$.
The last condition localizes the $L^{1}$ condition
of \eqref{eq:eigenfunction_L1_norm_bound} to $I_{U,R}$,
with $\ve/2$ instead of $\ve$.
For $f\in\cB_{a,\ve,R}(U)$, set
\begin{equation}\label{eq:def_Y_f}
    Y_{f} := \frac{-\langle\xi,f^{2}\rangle}
    {a+\|f'\|_{2}^{2}}\,,
    \qquad
    Z_{U} := \sup_{f\in\cB_{a,\ve,R}(U)}Y_{f}\,.
\end{equation}
For a fixed $f$, the event $\{Y_{f}\ge1\}$ is the
condition $-\langle\xi,f^{2}\rangle\ge a+\|f'\|_{2}^{2}$,
which comes from $\lambda_{1}(\cL)<-a$ as in
Remark~\ref{rem:conditions_and_constant}\,\textup{(iii)}.
Finally, write
\begin{equation*}
    E_{a} := \bigl\{\lambda_{1}(\cL)<-a,\
    \lambda_{2}(\cL)\ge-(1-\theta)a\bigr\}
    \qquad \text{and} \qquad 
    \fB_{a,\ve} := \Bigl\{\bigl|a^{1/4}\|\varphi_{1}\|_{1}
    -\pi/\sqrt{2}\bigr|>\ve\Bigr\}\,,
\end{equation*}
so that the event of
Proposition~\ref{prop:eigenfunction_L1_norm} is
$E_{a}\cap \fB_{a,\ve}$. The following lemma also
shows $Y_{\varphi_{1}}\ge1$ on this event, and reduces the
proof of Proposition~\ref{prop:eigenfunction_L1_norm} to
bounding $\dP\bigl(Z_{U}\ge1\bigr)$ uniformly in $U$.

\begin{lemma}\label{lem:global_to_local}
    For every $\cL>0$ and $a>0$,
    \begin{equation}\label{eq:global_to_local}
        E_{a}\cap \fB_{a,\ve}
        \subset\bigcup_{U\in\cU_{a}}\{Z_{U}\ge1\}\,,
    \end{equation}
    and consequently
    \begin{equation}\label{eq:union_bound_reduction}
        \dP\bigl(E_{a}\cap \fB_{a,\ve}\bigr)
        \le \bigl(\cL\sqrt{a}+2\bigr)
        \sup_{U\in\cU_{a}}\dP\bigl(Z_{U}\ge1\bigr)\,.
    \end{equation}
\end{lemma}

\begin{proof}
    We work on the event $E_{a}\cap \fB_{a,\ve}$. 
    Since $\fm_{\varphi_{1}}\in[0,\cL]$, 
    by \eqref{eq:def_mesh} we may choose $U\in\cU_{a}$ with
    $|U-\fm_{\varphi_{1}}|\le a^{-1/2}$.

    We first show $\varphi_{1}\in\cB_{a,\ve,R}(U)$.
    The first two conditions in \eqref{eq:def_bad_class} hold
    by construction. For the third, on $E_{a}$ the gap
    satisfies $\lambda_{2}(\cL)-\lambda_{1}(\cL)\ge\theta a$,
    so Lemma~\ref{lem:spectral_gap_concentration} gives
    \begin{equation*}
        \int_{0}^{\cL}\bigl(t-\fm_{\varphi_{1}}\bigr)^{2}
        \varphi_{1}(t)^{2}\,dt \le \frac{1}{\theta a}\,.
    \end{equation*}
    Combined with $(t-U)^{2}\le2(t-\fm_{\varphi_{1}})^{2}
    +2(\fm_{\varphi_{1}}-U)^{2}$ and
    $|U-\fm_{\varphi_{1}}|\le a^{-1/2}$, this yields
    \begin{equation*}
        \int_{0}^{\cL}(t-U)^{2}\varphi_{1}(t)^{2}\,dt
        \le \frac{2}{\theta a}+\frac{2}{a}
        = \frac{C_{\theta}}{a}\,,
    \end{equation*}
    which explains the choice of $C_{\theta}$.
    For the last condition, for $t\notin I_{U,R}$,
    \begin{equation*}
        |t-\fm_{\varphi_{1}}|
        \ge |t-U|-|U-\fm_{\varphi_{1}}|
        > (R-1)\,a^{-1/2}\,.
    \end{equation*}
    Lemma~\ref{lem:spectral_gap_concentration} with
    $r=(R-1)a^{-1/2}$ and the gap bound on $E_{a}$ give
    \begin{equation*}
        a^{1/4}\int_{[0,\cL]\setminus I_{U,R}}
        \varphi_{1}(t)\,dt
        \le \sqrt{\frac{2}{\theta}}\,(R-1)^{-1/2}\,,
    \end{equation*}
    and the right-hand side is smaller than $\ve/4$ by the
    first condition in \eqref{eq:choice_R}. On $\fB_{a,\ve}$,
    the triangle inequality gives
    \begin{equation*}
        \Bigl|a^{1/4}\int_{I_{U,R}}\varphi_{1}\,dt
        -\frac{\pi}{\sqrt{2}}\Bigr|
        \ge \Bigl|a^{1/4}\|\varphi_{1}\|_{1}
        -\frac{\pi}{\sqrt{2}}\Bigr|
        -a^{1/4}\int_{[0,\cL]\setminus I_{U,R}}\varphi_{1}\,dt
        > \frac{\ve}{2}\,.
    \end{equation*}
    Thus $\varphi_{1}\in\cB_{a,\ve,R}(U)$.

    It remains to show $Y_{\varphi_{1}}\ge1$. The weak
    eigenvalue relation \eqref{eq:weak_eigen_relation} with
    the test function $\varphi_{1}$ gives
    \begin{equation*}
        \lambda_{1}(\cL)
        = \cE_{\cL}(\varphi_{1},\varphi_{1})
        = \|\varphi_{1}'\|_{2}^{2}
        +\langle\xi,\varphi_{1}^{2}\rangle\,.
    \end{equation*}
    On $E_{a}$ we have $\lambda_{1}(\cL)<-a$, so
    $-\langle\xi,\varphi_{1}^{2}\rangle
    >a+\|\varphi_{1}'\|_{2}^{2}$, that is,
    $Y_{\varphi_{1}}>1$. Hence
    $Z_{U}\ge Y_{\varphi_{1}}\ge1$,
    which proves \eqref{eq:global_to_local}. The union bound
    and \eqref{eq:def_mesh} give
    \eqref{eq:union_bound_reduction}.
\end{proof}

We turn to the bound on $\dP(Z_{U}\ge1)$. It consists of two
steps, a lower bound on the cost over $\cB_{a,\ve,R}(U)$ and
the Borell--TIS inequality. The family
$(Y_{f})_{f}$ in \eqref{eq:def_Y_f} is a centered Gaussian
process, and $\Var\langle\xi,f^{2}\rangle=\|f\|_{4}^{4}$
gives
\begin{equation}\label{eq:Var_Y_f}
    \Var(Y_{f}) = \frac{1}{2\,J_{a}(f)}\,,
\end{equation}
where $J_{a}$ is defined in \eqref{eq:def_J_functional}. The
following lemma gives a lower bound on $J_{a}(f)$ over the
class $\cB_{a,\ve,R}(U)$, which implies an upper bound on
$\Var(Y_{f})$.
\begin{lemma}\label{lem:variational_gap}
    There exists $c_{1}>0$, depending only on $\ve$ and
    $\theta$, such that for every $a>0$, every $U\in\cU_{a}$,
    and every $f\in\cB_{a,\ve,R}(U)$,
    \begin{equation}\label{eq:cost_gap}
        J_{a}(f) \ge \Bigl(\frac{8}{3}+c_{1}\Bigr)a^{3/2}\,.
    \end{equation}
\end{lemma}
\begin{proof}
    Fix $U\in\cU_{a}$ and $f\in\cB_{a,\ve,R}(U)$, and define
    $\psi\in H^{1}(\bR)$ by
    \begin{equation}\label{eq:rescaling}
        \psi(y) := a^{-1/4}f\bigl(U+a^{-1/2}y\bigr)\,,
    \end{equation}
    extending $f$ by zero outside $[0,\cL]$. A change of
    variables gives
    \begin{equation*}
        \begin{aligned}
            &\|\psi\|_{2}^{2}=1\,,\qquad
            \|\psi'\|_{2}^{2}=a^{-1}\|f'\|_{2}^{2}\,,\qquad
            \|\psi\|_{4}^{4}=a^{-1/2}\|f\|_{4}^{4}\,,\\
            &\int_{\bR}y^{2}\psi(y)^{2}\,dy
            = a\int_{0}^{\cL}(t-U)^{2}f(t)^{2}\,dt
            \le C_{\theta}\,,
        \end{aligned}
    \end{equation*}
    so that $J_{a}(f)=a^{3/2}\cJ(\psi)$, where
    \begin{equation*}
        \cJ(\psi)
        := \frac{\bigl(1+\|\psi'\|_{2}^{2}\bigr)^{2}}
        {2\,\|\psi\|_{4}^{4}}\,.
    \end{equation*}
    Under \eqref{eq:rescaling}, the interval $I_{U,R}$
    corresponds to $[-R,R]$, so
    $a^{1/4}\int_{I_{U,R}}f\,dt=\int_{-R}^{R}\psi\,dy$, and
    $f\in\cB_{a,\ve,R}(U)$ translates into $\psi\in\cC_{R}$,
    where
    \begin{equation}\label{eq:def_K_R_class}
        \cC_{R} := \Bigl\{\psi\in H^{1}(\bR)\,:\,
        \psi\ge0,\ \|\psi\|_{2}=1,\
        \int_{\bR}y^{2}\psi^{2}\le C_{\theta},\
        \Bigl|\int_{-R}^{R}\psi-\frac{\pi}{\sqrt{2}}\Bigr|
        >\frac{\ve}{2}\Bigr\}\,.
    \end{equation}
    Since $\cC_{R}$ does not depend on $a$ or $U$, it suffices
    to show
    \begin{equation}\label{eq:variational_gap}
        \inf_{\psi\in\cC_{R}}\cJ(\psi) > \frac{8}{3}\,.
    \end{equation}
    If $\cC_{R}=\varnothing$, there is nothing to
    prove. The proof of \eqref{eq:variational_gap} has two
    steps. We first identify the equality cases of the bound
    $\cJ\ge\frac{8}{3}$, and then show by compactness that
    the bound is strict on $\cC_{R}$.
    We use the sharp one-dimensional Gagliardo--Nirenberg
    inequality
    \begin{equation}\label{eq:GN}
        \|v\|_{4}^{4}
        \le \frac{1}{\sqrt{3}}\,\|v'\|_{2}\,\|v\|_{2}^{3}\,,
        \qquad v\in H^{1}(\bR)\,,
    \end{equation}
    see \cite{Nag41}, and \cite[Theorem~A.1]{Fra14} for the
    form used here.
    For $\|\psi\|_{2}=1$ and $r:=\|\psi'\|_{2}$,
    \begin{equation}\label{eq:J_lower}
        \cJ(\psi)
        \ge \frac{\sqrt{3}\,(1+r^{2})^{2}}{2\,r}
        \ge \frac{8}{3}\,,
    \end{equation}
    the middle expression being minimized at $r^{2}=1/3$.
    Equality in \eqref{eq:J_lower} thus requires equality in
    \eqref{eq:GN} and $r^{2}=1/3$.
    By \cite[Theorem~A.1]{Fra14}, the nonnegative equality
    cases of \eqref{eq:GN} are exactly
    $\psi(y)=A\operatorname{sech}\bigl(b(y-z)\bigr)$ with
    $A,b>0$ and $z\in\bR$.
    Since $\|\psi\|_{2}^{2}=2A^{2}/b$
    and $\|\psi'\|_{2}^{2}=2A^{2}b/3$, the constraints
    $\|\psi\|_{2}=1$ and $r^{2}=1/3$ force $b=1$, $A=1/\sqrt{2}$. 
    Therefore
    \begin{equation}\label{eq:J_equality_case}
        \cJ(\psi)=\frac{8}{3}\,,\ \psi\ge0\,,\
        \|\psi\|_{2}=1
        \quad\Longleftrightarrow\quad
        \psi=\cQ(\cdot-z)\ \text{for some } z\in\bR\,.
    \end{equation}
    By \eqref{eq:J_lower}, the infimum in
    \eqref{eq:variational_gap} is at least $8/3$. If
    \eqref{eq:variational_gap} failed, there would exist
    $\psi_{n}\in\cC_{R}$ with $\cJ(\psi_{n})\to8/3$. Set
    $r_{n}:=\|\psi_{n}'\|_{2}$. Again by \eqref{eq:J_lower},
    $(r_{n})$ is bounded and bounded away from $0$, so
    $(\psi_{n})$ is bounded in $H^{1}(\bR)$. After passing to a
    subsequence, $\psi_{n}\rightharpoonup\psi$ weakly in
    $H^{1}(\bR)$, and, by the Rellich--Kondrachov theorem on
    each $[-M,M]$ and a diagonal argument, $\psi_{n}\to\psi$ in
    $L^{2}_{\mathrm{loc}}(\bR)$. The moment bound in
    \eqref{eq:def_K_R_class} improves this to convergence in
    $L^{2}(\bR)$. Indeed, since
    $\int_{|y|>M}\psi_{n}^{2}\le C_{\theta}/M^{2}$ for
    every $M$,
    \begin{equation*}
        \int_{\bR}|\psi_{n}-\psi_{m}|^{2}
        \le \int_{-M}^{M}|\psi_{n}-\psi_{m}|^{2}
        + \frac{4C_{\theta}}{M^{2}}\,,
    \end{equation*}
    and letting $n,m\to\infty$ and then $M\to\infty$ shows that
    $(\psi_{n})$ is Cauchy in $L^{2}(\bR)$. Hence
    $\psi_{n}\to\psi$ in $L^{2}(\bR)$, $\psi\ge0$, and
    $\|\psi\|_{2}=1$.
    Applying \eqref{eq:GN} to $\psi_{n}-\psi$, together with
    $\psi_{n}\to\psi$ in $L^{2}(\bR)$ and the boundedness of
    $\|\psi_{n}'-\psi'\|_{2}$, gives $\psi_{n}\to\psi$ in
    $L^{4}(\bR)$. Together with
    $\|\psi'\|_{2}^{2}\le\liminf_{n}\|\psi_{n}'\|_{2}^{2}$, this
    yields $\cJ(\psi)\le\lim_{n}\cJ(\psi_{n})=8/3$. Since
    $\cJ(\psi)\ge8/3$ by \eqref{eq:J_lower}, equality holds,
    and $\psi=\cQ(\cdot-z)$ for some $z\in\bR$ by
    \eqref{eq:J_equality_case}. Moreover, by Fatou's lemma and
    $\int_{\bR}y\,\cQ(y)^{2}\,dy=0$,
    \begin{equation*}
        z^{2}
        \le \int_{\bR}y^{2}\cQ(y-z)^{2}\,dy
        \le \liminf_{n}\int_{\bR}y^{2}\psi_{n}^{2}\,dy
        \le C_{\theta}\,,
    \end{equation*}
    so $|z|\le\sqrt{C_{\theta}}$. Since
    $\psi_{n}\to\cQ(\cdot-z)$ in $L^{2}(\bR)$, the
    Cauchy--Schwarz inequality on $[-R,R]$ and the second
    condition in \eqref{eq:choice_R} give, for all large $n$,
    \begin{equation*}
        \Bigl|\int_{-R}^{R}\psi_{n}-\frac{\pi}{\sqrt{2}}\Bigr|
        \le \Bigl|\int_{-R}^{R}
        \bigl(\psi_{n}-\cQ(\cdot-z)\bigr)\Bigr|
        + \sup_{|z'|\le\sqrt{C_{\theta}}}
        \Bigl|\int_{-R}^{R}\cQ(\cdot-z')
        -\frac{\pi}{\sqrt{2}}\Bigr|
        < \frac{\ve}{2}\,,
    \end{equation*}
    contradicting the constraint in \eqref{eq:def_K_R_class}.
    This proves \eqref{eq:variational_gap}, and
    \eqref{eq:cost_gap} follows with
    $c_{1}:=\inf_{\cC_{R}}\cJ-\frac{8}{3}>0$, which
    depends only on $\ve$, $R$, and $C_{\theta}$, and
    hence only on $\ve$ and $\theta$.
\end{proof}

We now prove
Proposition~\ref{prop:eigenfunction_L1_norm} by combining
Lemmas~\ref{lem:global_to_local}
and~\ref{lem:variational_gap} with the Borell--TIS inequality.
\begin{proof}[Proof of Proposition~\ref{prop:eigenfunction_L1_norm}]
    By Lemma~\ref{lem:global_to_local}, it suffices to
    bound $\dP(Z_{U}\ge1)$ uniformly over
    $U\in\cU_{a}$.
    We first show that, for every $a>0$, uniformly in
    $U\in\cU_{a}$,
    \begin{equation}\label{eq:E_Z_bound}
        \dE[Z_{U}] \le 8\,(1+C_{\theta})^{1/2}\,a^{-3/4}\,,
    \end{equation}
    and then conclude by the Borell--TIS inequality.
    The last condition of \eqref{eq:def_bad_class} plays
    no role here, so we work with the larger class
    \begin{equation*}
        \cD_{a}(U) := \Bigl\{f\in H_{0}^{1}(0,\cL)\,:\,
        f\ge0,\ \|f\|_{2}=1,\
        \int_{0}^{\cL}(t-U)^{2}f(t)^{2}\,dt
        \le\frac{C_{\theta}}{a}\Bigr\}\,.
    \end{equation*}
    Since $\cB_{a,\ve,R}(U)\subset\cD_{a}(U)$, it suffices to
    bound $\dE\bigl[\sup_{f\in\cD_{a}(U)}Y_{f}\bigr]$, where
    $Y_{f}$ is given by the formula \eqref{eq:def_Y_f}.
    For $f\in\cD_{a}(U)$ with rescaling $\psi$ as in
    \eqref{eq:rescaling}, a change of variables in
    \eqref{eq:def_noise_pairing}, together with
    $\int_{\bR}h_{\psi}'(y)\,dy=0$, gives
    \begin{equation*}
        Y_{f}
        = a^{-3/4}\int_{\bR}h_{\psi}'(y)\,W_{U}(y)\,dy\,,
        \qquad
        h_{\psi} := \frac{\psi^{2}}{1+\|\psi'\|_{2}^{2}}\,,
    \end{equation*}
    where $W_{U}(y):=a^{1/4}\bigl(B(U+a^{-1/2}y)-B(U)\bigr)$
    and $B$ is extended to a two-sided Brownian motion on
    $\bR$. Set
    \begin{equation*}
        W_{U}^{*} := \sup_{y\in\bR}
        \frac{|W_{U}(y)|}{\sqrt{1+y^{2}}}\,.
    \end{equation*}
    Since
    $h_{\psi}'=2\psi\psi'/\bigl(1+\|\psi'\|_{2}^{2}\bigr)$, the
    Cauchy--Schwarz inequality gives
    \begin{equation*}
        \int_{\bR}\sqrt{1+y^{2}}\,|h_{\psi}'(y)|\,dy
        \le \frac{2\|\psi'\|_{2}}{1+\|\psi'\|_{2}^{2}}
        \biggl(\int_{\bR}(1+y^{2})\,\psi^{2}\,dy\biggr)^{1/2}
        \le (1+C_{\theta})^{1/2}\,,
    \end{equation*}
    using $2\|\psi'\|_{2}\le1+\|\psi'\|_{2}^{2}$,
    $\|\psi\|_{2}=1$, and
    $\int_{\bR}y^{2}\psi^{2}\le C_{\theta}$.
    Hence, for every $f\in\cD_{a}(U)$,
    \begin{equation*}
        |Y_{f}|
        \le a^{-3/4}\,W_{U}^{*}
        \int_{\bR}\sqrt{1+y^{2}}\,|h_{\psi}'|\,dy
        \le a^{-3/4}\,(1+C_{\theta})^{1/2}\,W_{U}^{*}\,.
    \end{equation*}
    It remains to bound $\dE[W_{U}^{*}]$.
    By the stationarity of Brownian increments and
    Brownian scaling, $W_{U}$ is a standard two-sided
    Brownian motion, so its law does not depend on $U$. We
    use $\sqrt{1+y^{2}}\ge1$ for $|y|\le1$ and
    $\sqrt{1+y^{2}}\ge|y|$ for $|y|\ge1$. By the time
    inversion $y\mapsto1/y$,
    \begin{equation*}
        \sup_{|y|\ge1}\frac{|W_{U}(y)|}{|y|}
        = \sup_{0<|u|\le1}\bigl|u\,W_{U}(1/u)\bigr|\,,
    \end{equation*}
    and $u\,W_{U}(1/u)$ is again a standard two-sided
    Brownian motion. Hence
    \begin{equation*}
        \dE\bigl[W_{U}^{*}\bigr]
        \le 2\,\dE\Bigl[\sup_{|y|\le1}|W_{U}(y)|\Bigr]
        \le 8\,.
    \end{equation*}
    Combining the last two bounds proves \eqref{eq:E_Z_bound}.
    Choose $a_{0}=a_{0}(\theta,\ve)$ so large that
    the right-hand side of \eqref{eq:E_Z_bound} is at most
    $1/2$ for $a\ge a_{0}$, so that $\dE[Z_{U}]\le1/2$.
    The map $f\mapsto Y_{f}$ is continuous on the separable
    class $\cB_{a,\ve,R}(U)\subset H^{1}$, so the
    Borell--TIS inequality (see, e.g.,
    \cite[Theorem~2.1.1]{AT07}) gives
    \begin{equation*}
        \dP\bigl(Z_{U}\ge1\bigr)
        \le \exp\Bigl(-\frac{(1-\dE[Z_{U}])^{2}}
        {2\sigma_{a}^{2}}\Bigr)\,,
        \qquad
        \sigma_{a}^{2}
        := \sup_{f\in\cB_{a,\ve,R}(U)}\Var(Y_{f})\,.
    \end{equation*}
    By \eqref{eq:Var_Y_f} and
    Lemma~\ref{lem:variational_gap}, we have
    \begin{equation*}
        \sigma_{a}^{2}
        \le \frac{1}{2\bigl(\frac{8}{3}+c_{1}\bigr)
        a^{3/2}}\,.
    \end{equation*}
    Set $c_{2}:=16\,(1+C_{\theta})^{1/2}
    \bigl(\tfrac{8}{3}+c_{1}\bigr)$.
    Since $(1-\dE[Z_{U}])^{2}\ge1-2\,\dE[Z_{U}]$, the bound on
    $\sigma_{a}^{2}$ and \eqref{eq:E_Z_bound} give, uniformly in
    $U\in\cU_{a}$,
    \begin{equation*}
        \frac{(1-\dE[Z_{U}])^{2}}{2\sigma_{a}^{2}}
        \ge \Bigl(\frac{8}{3}+c_{1}\Bigr)a^{3/2}
        - c_{2}\,a^{3/4}\,.
    \end{equation*}
    Combining this with \eqref{eq:union_bound_reduction} yields
    \eqref{eq:eigenfunction_L1_norm_bound}. Since $\theta$,
    $\ve$, and $R=R(\theta,\ve)$ determine $c_{1}$, $c_{2}$, and
    $a_{0}$, these constants depend only on $\theta$ and $\ve$.
\end{proof}

We now prove
Theorem~\ref{thm:eigenfunction_concentration}.
\begin{proof}[Proof of Theorem~\ref{thm:eigenfunction_concentration}]
    Notice that the event on the
    right-hand side of \eqref{eq:eigenfunction_concentration}
    is $E_{a}\cap \fB_{a,\ve}^{c}$.
    Decomposing the event
    $\{\lambda_{1}(\cL_{a})<-a\}
    \setminus\bigl(E_{a}\cap \fB_{a,\ve}^{c}\bigr)$ according
    to whether $\lambda_{2}(\cL_{a})<-(1-\theta)a$, we obtain
    \begin{equation*}
        \dP\bigl(\lambda_{1}(\cL_{a})<-a\bigr)
        -\dP\bigl(E_{a}\cap \fB_{a,\ve}^{c}\bigr)
        \le \dP\bigl(\lambda_{2}(\cL_{a})<-(1-\theta)a\bigr)
        + \dP\bigl(E_{a}\cap \fB_{a,\ve}\bigr)\,,
    \end{equation*}
    so it suffices to show that both terms on the right-hand
    side are negligible compared with
    $\dP(\lambda_{1}(\cL_{a})<-a)$.
    Since \eqref{eq:concentration_window} implies
    \eqref{eq:admissible_window},
    Theorem~\ref{thm:sharp_lower_tail_lowest_eigenvalue} and
    \eqref{eq:first_explosion_mean_asymptotics} give
    \begin{equation*}
        \dP\bigl(\lambda_{1}(\cL_{a})<-a\bigr)
        \sim \frac{\cL_{a}}{m(a)}
        = \cL_{a}\,\frac{\sqrt{a}}{\pi}\,
        e^{-\frac{8}{3}a^{3/2}}\bigl(1+o(1)\bigr)\,.
    \end{equation*}
    For the first term,
    Corollary~\ref{cor:second_eigenvalue_left_tail} gives
    \begin{equation*}
        \dP\bigl(\lambda_{2}(\cL_{a})<-(1-\theta)a\bigr)
        \le \Bigl(e\,\frac{\cL_{a}}{m((1-\theta)a)}\Bigr)^{2}\,.
    \end{equation*}
    Taking logarithms and using
    \eqref{eq:first_explosion_mean_asymptotics} at
    $(1-\theta)a$, we obtain
    \begin{equation*}
        \log\frac{\dP\bigl(\lambda_{2}(\cL_{a})
        <-(1-\theta)a\bigr)}
        {\dP\bigl(\lambda_{1}(\cL_{a})<-a\bigr)}
        \le \log\cL_{a}
        + \frac{8}{3}a^{3/2}
        - \frac{16}{3}(1-\theta)^{3/2}a^{3/2}
        + O(\log a)\,.
    \end{equation*}
    By the second condition in \eqref{eq:concentration_window},
    $\log\cL_{a}\le\frac{8}{3}(1-\eta)a^{3/2}\bigl(1+o(1)\bigr)$,
    and the condition $\theta<1-(1-\eta/2)^{2/3}$ is equivalent
    to $2-\eta<2(1-\theta)^{3/2}$.
    Therefore, the right-hand side tends to $-\infty$.
    For the second term,
    Proposition~\ref{prop:eigenfunction_L1_norm}
    together with $\cL_{a}\sqrt{a}+2\le2\cL_{a}\sqrt{a}$,
    valid for all sufficiently large $a$, gives
    \begin{equation*}
        \frac{\dP\bigl(E_{a}\cap \fB_{a,\ve}\bigr)}
        {\dP\bigl(\lambda_{1}(\cL_{a})<-a\bigr)}
        \le 2\pi\,
        e^{-c_{1}a^{3/2}+c_{2}a^{3/4}}\bigl(1+o(1)\bigr)
        \longrightarrow 0\,.
    \end{equation*}
    The proof is complete.
\end{proof}

\section{The Dirichlet PAM on a finite interval}
\label{sec:Dirichlet_PAM_finite_interval}
In this section, we study the PAM on a bounded interval with
Dirichlet boundary conditions. We first define its
solution and recall a spectral representation in terms of the
Anderson Hamiltonian of
Section~\ref{sec:Anderson_Hamiltonian_finite_interval}. The
lower-tail estimates for $\lambda_{1}$ and the
concentration of $\|\varphi_{1}\|_{1}$ from that section then
yield bounds on the moments of the spatial integral. We next partition
$Q_{L(t)}$ into $n(t)$ blocks of length $\ell(t)$
and prepare the comparison of $U(t)$ with the sum of the Dirichlet block
integrals.

Let $Q\subset\bR$ be a bounded open interval. We consider the
PAM on $Q$ with Dirichlet boundary conditions and constant
initial condition,
\begin{equation}\label{eq:Dirichlet_PAM}
    \begin{cases}
        \partial_{t}\tilde{u}_{Q}(t,x)
        = \partial_{x}^{2}\tilde{u}_{Q}(t,x)
        +\xi(x)\,\tilde{u}_{Q}(t,x)\,,
        & t>0\,,\ x\in Q\,,\\
        \tilde{u}_{Q}(t,x) = 0\,,
        & t>0\,,\ x\in\partial Q\,,\\
        \tilde{u}_{Q}(0,x) = 1\,,
        & x\in Q\,.
    \end{cases}
\end{equation}
As in Section~\ref{sec:intro}, the problem
\eqref{eq:Dirichlet_PAM} is formal, since $\xi$ is not defined
pointwise. We again define the solution by the Feynman--Kac
formula. With $W$, $\dE_{x}$, and the integral
$\int_{0}^{t}\xi(W_{s})\,ds$ as in Section~\ref{sec:intro}, set
\begin{equation*}
    \tau_{Q} := \inf\bigl\{s\ge0\,:\,W_{s}\notin Q\bigr\}\,,
\end{equation*}
and define the Dirichlet solution by
\begin{equation}\label{eq:def_Dirichlet_solution}
    \tilde{u}_{Q}(t,x)
    := \dE_{x}\Bigl[\exp\Bigl(\int_{0}^{t}\xi(W_{s})\,ds\Bigr)
    \,;\,\tau_{Q}>t\Bigr]\,,
    \qquad t>0\,,\ x\in Q\,.
\end{equation}
Define $\cH_{Q}:=-\partial_{x}^{2}-\xi$ through the form
\eqref{eq:FN_form} on $H_{0}^{1}(Q)$ with $\xi$ replaced by
$-\xi$, and write $(\lambda_{k}(Q),\varphi_{k})_{k\ge1}$ for
its eigenpairs, with $(\varphi_{k})_{k\ge1}$ an orthonormal
basis of $L^{2}(Q)$. Since $-\xi$ is
equal in law to $\xi$ and the law of $\xi$ is invariant under
translations, the eigenpairs
$(\lambda_{k}(Q),\varphi_{k})_{k\ge1}$ have the same joint law
as those of $\cH_{\cL}$ with $\cL=|Q|$. All applications of the
results of Section~\ref{sec:Anderson_Hamiltonian_finite_interval}
below are through this equality in law. For every $t>0$,
almost surely,
\begin{equation}\label{eq:spectral_rep_Dirichlet_PAM}
    \tilde{u}_{Q}(t,\cdot)
    = e^{-t\cH_{Q}}\1_{Q}
    = \sum_{k\ge1}e^{-t\lambda_{k}(Q)}
    \langle\varphi_{k},\1_{Q}\rangle\,\varphi_{k}
    \qquad\text{in } L^{2}(Q)\,,
\end{equation}
and hence
\begin{equation}\label{eq:spatial_integral_spectral_rep}
    \int_{Q}\tilde{u}_{Q}(t,x)\,dx
    = \sum_{k\ge1}e^{-t\lambda_{k}(Q)}
    \Bigl(\int_{Q}\varphi_{k}(x)\,dx\Bigr)^{2}\,.
\end{equation}
The first equality in \eqref{eq:spectral_rep_Dirichlet_PAM} follows from the Feynman--Kac formula of \cite[Theorem~2.24]{GL21} for a bounded interval with Dirichlet boundary conditions at both endpoints (Case~3-D in that paper). We apply this formula to $\tfrac{1}{2}\cH_Q=-\tfrac{1}{2}\partial_x^2-\xi/2$ at time $2t$, using the time change in Remark~\ref{rem:convention} (see also \cite[Remark~2.22, Example~2.28, and Proposition~2.29]{GL21}). The stochastic integral in \cite[(2.14)]{GL21} coincides with our time integral,
$\dP\otimes\dP_{x}$-almost surely. Indeed,
writing $\Lambda_{t}$ for the occupation density of $W$ on $[0,t]$,
we have $\int_{0}^{t}\xi_{\ve}(W_{s})\,ds=\xi(j_{\ve}*\Lambda_{t})$,
and $j_{\ve}*\Lambda_{t}\to\Lambda_{t}$ in $L^{2}(\bR)$ as
$\ve\downarrow0$, so that
\begin{equation*}
    \int_{0}^{t}\xi(W_{s})\,ds = \xi(\Lambda_{t})\,,
\end{equation*}
where $\xi(\Lambda_{t})$ is understood through the
$L^{2}(\bR)$-extension of $\xi$.

The second equality in \eqref{eq:spectral_rep_Dirichlet_PAM} follows from the spectral theorem, since $(\varphi_k)_{k\ge1}$ is an orthonormal basis of $L^2(Q)$ consisting of eigenfunctions of $\cH_Q$.

We next apply \eqref{eq:spatial_integral_spectral_rep} to
$Q=Q_{\cL}:=(-\cL/2,\,\cL/2)$ and estimate the moments
of the spatial integral of the Dirichlet solution.
The upper bound holds for every $\cL\ge1$, 
with constants depending only on $p$. The lower bound relies on
Theorem~\ref{thm:eigenfunction_concentration} and requires
$\cL=\cL(t)$ to lie in the corresponding window.
\begin{proposition}\label{prop:moment_bounds_Dirichlet_PAM}
    Let $p>0$. There exist $C_{1},t_{0}>0$, depending only on
    $p$, such that for every $\cL\ge1$ and every $t\ge t_{0}$,
    \begin{equation}\label{eq:moment_upper_bound}
        \dE\biggl[\Bigl(\int_{Q_{\cL}}
        \tilde{u}_{Q_{\cL}}(t,x)\,dx\Bigr)^{p}\biggr]
        \le C_{1}\,\cL^{p+1}\,t^{5/2}
        \exp\Bigl(\frac{p^{3}}{48}\,t^{3}\Bigr)\,.
    \end{equation}
    In addition, assume that $\cL=\cL(t)$ satisfies
    \begin{equation}\label{eq:moment_window_condition}
        a(t)\exp\bigl(4(\log a(t))^{2}\bigr) \ll \cL(t)
        \qquad\text{and}\qquad
        \limsup_{t\to\infty}
        \frac{\log\cL(t)}{\tfrac{8}{3}\,a(t)^{3/2}} \le 1-\eta
    \end{equation}
    for some $\eta\in(0,1)$, where
    $a(t):=p^{2}t^{2}/16$. Then there exists $C_{2}>0$,
    depending only on $p$, such that for all
    sufficiently large $t$,
    \begin{equation}\label{eq:moment_lower_bound}
        \dE\biggl[\Bigl(\int_{Q_{\cL(t)}}
        \tilde{u}_{Q_{\cL(t)}}(t,x)\,dx\Bigr)^{p}\biggr]
        \ge C_{2}\,\cL(t)\,t^{1-p}
        \exp\Bigl(\frac{p^{3}}{48}\,t^{3}\Bigr)\,.
    \end{equation}
\end{proposition}
\begin{proof}
    Throughout the proof, write
    $\lambda_{k}:=\lambda_{k}(Q_{\cL})$, and let $C>0$ denote a
    constant depending only on $p$, whose value may change from
    line to line. We begin with the upper bound. Since
    $\sum_{k\ge1}\bigl(\int_{Q_{\cL}}\varphi_{k}\bigr)^{2}
    =\|\1_{Q_{\cL}}\|_{2}^{2}=\cL$ by Parseval's identity,
    \eqref{eq:spatial_integral_spectral_rep} gives
    \begin{equation*}
        \dE\biggl[\Bigl(\int_{Q_{\cL}}
        \tilde{u}_{Q_{\cL}}(t,x)\,dx\Bigr)^{p}\biggr]
        \le \cL^{p}\,\dE\bigl[e^{-pt\lambda_{1}}\bigr]\,.
    \end{equation*}
    We estimate the exponential moment. By
    Lemma~\ref{lem:left_tail_upper_bound_fixed_L} and
    \eqref{eq:first_explosion_mean_asymptotics}, there exist
    $C_{0},x_{0}>0$, independent of $\cL$, such that
    \begin{equation}\label{eq:lower_tail_nonasymptotic}
        \dP\bigl(-\lambda_{1}>x\bigr)
        \le C_{0}\,\cL\,\sqrt{x}\,
        \exp\Bigl(-\frac{8}{3}x^{3/2}\Bigr)\,,
        \qquad x\ge x_{0}\,.
    \end{equation}
    By integration by parts and \eqref{eq:lower_tail_nonasymptotic}, we obtain 
    \begin{equation}\label{eq:lambda_1_Laplace_transform_bound}
        \dE\bigl[e^{-pt\lambda_{1}}\bigr]
        \le 2\,e^{ptx_{0}}
        + pt\int_{x_{0}}^{\infty}e^{ptx}\,
        \dP\bigl(-\lambda_{1}>x\bigr)\,dx
        \le 2\,e^{ptx_{0}}
        + C_{0}\,\cL\,pt\,\cZ(x_{0},pt)\,,
    \end{equation}
    where
    \begin{equation*}
        \cZ(x_{0},s)
        := \int_{x_{0}}^{\infty}\sqrt{x}\,
        \exp\Bigl(sx-\frac{8}{3}x^{3/2}\Bigr)\,dx\,.
    \end{equation*}
    After the substitution $r=\sqrt{x}$, we write the
exponent around its maximizer $r=s/4$ as
\begin{equation*}
    sr^{2}-\frac{8}{3}r^{3}
    = \frac{s^{3}}{48}
    -\frac{8}{3}\Bigl(r-\frac{s}{4}\Bigr)^{2}
    \Bigl(r+\frac{s}{8}\Bigr)\,.
\end{equation*}
This gives
    \begin{equation*}
        \cZ(x_{0},s)
        = 2\,e^{s^{3}/48}\int_{\sqrt{x_{0}}}^{\infty}r^{2}
        \exp\Bigl(-\frac{8}{3}\Bigl(r-\frac{s}{4}\Bigr)^{2}
        \Bigl(r+\frac{s}{8}\Bigr)\Bigr)\,dr\,.
    \end{equation*}
    Since $r+s/8\ge s/8$ on the domain of integration and
    $r^{2}\le2(r-s/4)^{2}+s^{2}/8$, the last integral is at
    most
    \begin{equation*}
        \int_{\bR}\Bigl(2v^{2}+\frac{s^{2}}{8}\Bigr)
        e^{-sv^{2}/3}\,dv
        \le C\,s^{3/2}\,,
        \qquad s\ge1\,.
    \end{equation*}
    Hence $\cZ(x_{0},s)\le C\,s^{3/2}\exp\bigl(s^{3}/48\bigr)$,
    and taking $s=pt$ in
    \eqref{eq:lambda_1_Laplace_transform_bound}, we
    obtain
    \begin{equation*}
        \dE\bigl[e^{-pt\lambda_{1}}\bigr]
        \le C\,\cL\,t^{5/2}
        \exp\Bigl(\frac{p^{3}}{48}\,t^{3}\Bigr)\,,
    \end{equation*}
    since $\cL\ge1$ and $e^{ptx_{0}}\le\exp(p^{3}t^{3}/48)$ for
    all $t\ge t_{0}(p)$. Multiplying by $\cL^{p}$ proves
    \eqref{eq:moment_upper_bound}.
    We now turn to the lower bound. Fix
    $\ve:=\pi/(2\sqrt{2})$ and
    $\theta\in\bigl(0,\,1-(1-\eta/2)^{2/3}\bigr)$, and let
    $G_{t}$ denote the event
    \begin{equation*}
        G_{t} := \Bigl\{\lambda_{1}<-a(t)\,,\
        \lambda_{2}\ge-(1-\theta)a(t)\,,\
        \bigl|a(t)^{1/4}\|\varphi_{1}\|_{1}
        -\pi/\sqrt{2}\bigr|\le\ve\Bigr\}\,.
    \end{equation*}
    The condition \eqref{eq:moment_window_condition} is
    \eqref{eq:concentration_window} with $a=a(t)$, so
    Theorem~\ref{thm:eigenfunction_concentration},
    Theorem~\ref{thm:sharp_lower_tail_lowest_eigenvalue}, and
    \eqref{eq:first_explosion_mean_asymptotics} give
    \begin{equation}\label{eq:G_t_asymptotics}
        \dP(G_{t})
        \sim \dP\bigl(\lambda_{1}<-a(t)\bigr)
        \sim \frac{pt}{4\pi}\,\cL(t)
        \exp\Bigl(-\frac{p^{3}}{24}\,t^{3}\Bigr)\,.
    \end{equation}
    Recall from
    Section~\ref{subsec:eigenvalue_problem_Riccati} that
    $\varphi_{1}\ge0$, so
    $\langle\varphi_{1},\1_{Q_{\cL(t)}}\rangle
    =\|\varphi_{1}\|_{1}$, and on $G_{t}$,
    $\|\varphi_{1}\|_{1}\ge(\pi/\sqrt{2}-\ve)\,a(t)^{-1/4}$.
    Since every term in
    \eqref{eq:spatial_integral_spectral_rep} is nonnegative,
    the first term alone gives
    \begin{equation*}
        \dE\biggl[\Bigl(\int_{Q_{\cL(t)}}
        \tilde{u}_{Q_{\cL(t)}}(t,x)\,dx\Bigr)^{p}\biggr]
        \ge \dE\Bigl[\|\varphi_{1}\|_{1}^{2p}\,
        e^{-pt\lambda_{1}}\,;\,G_{t}\Bigr]
        \ge \Bigl(\frac{\pi^{2}}{8}\Bigr)^{p}
        a(t)^{-p/2}\,e^{pt\,a(t)}\,\dP(G_{t})\,,
    \end{equation*}
    and \eqref{eq:G_t_asymptotics} yields
    \eqref{eq:moment_lower_bound}, since
    $pt\,a(t)-\tfrac{p^{3}}{24}t^{3}=\tfrac{p^{3}}{48}t^{3}$.
\end{proof}

We now partition $Q_{L(t)}$ into blocks of the intermediate
scale $\ell(t)$ and associate to each block its Dirichlet block
integral. Recall $n(t)$ and $\ell(t)$ from \eqref{eq:block_number}, so
that $n(t)\,\ell(t)=L(t)$. For $i=1,\ldots,n(t)+1$,
define the grid points
\begin{equation*}
    q_{i} := -\frac{L(t)}{2}+(i-1)\,\ell(t)\,,
\end{equation*}
so that $q_{1}=-L(t)/2$ and $q_{n(t)+1}=L(t)/2$. For
$i\in\cI(t):=\{1,\ldots,n(t)\}$, define the blocks and block
integrals
\begin{equation*}
    Q_{i} := \bigl(q_{i},\,q_{i+1}\bigr)\,,
    \qquad
    U_{i}(t) := \int_{Q_{i}}u(t,x)\,dx\,.
\end{equation*}
By the stationarity of $u(t,\cdot)$, each $U_{i}(t)$ has the
same law as $U_{0}(t)$ of \eqref{eq:block_integral}. For
$i\in\cI(t)$, let $\tilde{u}_{Q_{i}}$ denote the Dirichlet
solution \eqref{eq:def_Dirichlet_solution} on $Q_{i}$, and
define the Dirichlet block integrals (compare
\eqref{eq:spatial_integral})
\begin{equation}\label{eq:def_Dirichlet_block_integral}
    \tilde{U}_{i}(t) := \int_{Q_{i}}\tilde{u}_{Q_{i}}(t,x)\,dx\,,
    \qquad
    \tilde{U}_{0}(t) := \int_{Q_{\ell(t)}}
    \tilde{u}_{Q_{\ell(t)}}(t,x)\,dx\,.
\end{equation}
For each fixed $t$, the family
$\bigl(\tilde{U}_{i}(t)\bigr)_{i\in\cI(t)}$ is i.i.d., and each
$\tilde{U}_{i}(t)$ has the same law as $\tilde{U}_{0}(t)$.
Indeed, $\tilde{U}_{i}(t)$ is a measurable function of
the restriction of $\xi$ to $Q_{i}$, and these restrictions
are independent and, by translation invariance, identically distributed.
Finally, we define
\begin{equation}\label{eq:tilde_U}
    \tilde{U}(t)
    := \sum_{i\in\cI(t)}\tilde{U}_{i}(t)\,.
\end{equation}

\begin{remark}\label{rem:window_condition_ell}
Since $a(t)\exp\bigl(4(\log a(t))^{2}\bigr)
=e^{O((\log t)^{2})}$ and
$\log\ell(t)=t(1+o(1))=o\bigl(a(t)^{3/2}\bigr)$, the
window condition \eqref{eq:moment_window_condition} holds
with $\cL(t)=\ell(t)$ for every fixed $p>0$ and every
$\eta\in(0,1)$, so
Proposition~\ref{prop:moment_bounds_Dirichlet_PAM} applies
to $\tilde{U}_{0}(t)$.
\end{remark}

In the proofs of the main theorems, we show that
$U(t)$ and $\tilde{U}(t)$ differ by an error that vanishes in
probability after the centering and scaling of each theorem.
To this end, we first compare $U_{i}(t)$ with
$\tilde{U}_{i}(t)$. For $i\in\cI(t)$, split $Q_{i}$ into
\begin{equation*}
    \fR_{i} := \bigl(q_{i}+t^{4},\,q_{i+1}-t^{4}\bigr)\,,
    \qquad
    \fT_{i}^{L} := \bigl(q_{i},\,q_{i}+t^{4}\bigr)\,,
    \qquad
    \fT_{i}^{R} := \bigl(q_{i+1}-t^{4},\,q_{i+1}\bigr)\,,
\end{equation*}
and, for $i=1,\ldots,n(t)+1$, set
\begin{equation*}
    \fS_{i} := \bigl(q_{i}-e^{t/100},\,q_{i}+e^{t/100}\bigr)\,.
\end{equation*}
For all sufficiently large $t$,
\begin{equation}\label{eq:layer_window_inclusion}
    \fT_{i}^{L}\subset\fS_{i}\,,
    \qquad
    \fT_{i}^{R}\subset\fS_{i+1}\,.
\end{equation}

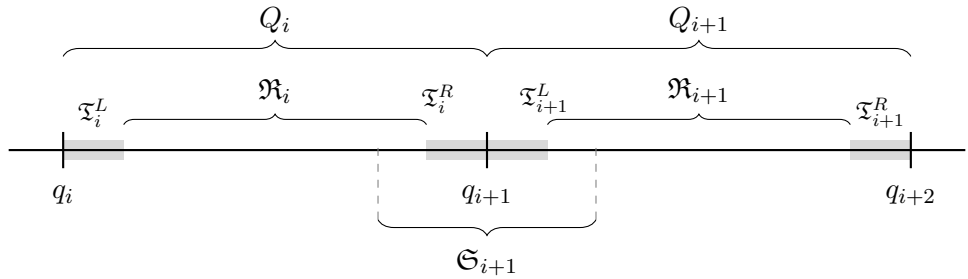
\begin{figure}[ht]
\centering
\begin{tikzpicture}[x=0.8cm,y=0.8cm]
  \fill[gray!30] (0,-0.16) rectangle (1,0.16);
  \fill[gray!30] (6,-0.16) rectangle (7,0.16);
  \fill[gray!30] (7,-0.16) rectangle (8,0.16);
  \fill[gray!30] (13,-0.16) rectangle (14,0.16);
  \draw[thick] (-0.9,0) -- (14.9,0);
  \foreach \x in {0,7,14}{\draw[thick] (\x,-0.3) -- (\x,0.3);}
  \node[below] at (0,-0.42) {$q_{i}$};
  \node[below] at (7,-0.42) {$q_{i+1}$};
  \node[below] at (14,-0.42) {$q_{i+2}$};
  \draw[decorate,decoration={brace,amplitude=4pt}] (1,0.42) -- (6,0.42);
  \draw[decorate,decoration={brace,amplitude=4pt}] (8,0.42) -- (13,0.42);
  \node at (3.5,0.98) {$\fR_{i}$};
  \node at (10.5,0.98) {$\fR_{i+1}$};
  \node[scale=0.85] at (0.5,0.62) {$\fT_{i}^{L}$};
  \node[scale=0.85,anchor=east] at (6.6,0.85) {$\fT_{i}^{R}$};
  \node[scale=0.85,anchor=west] at (7.4,0.85) {$\fT_{i+1}^{L}$};
  \node[scale=0.85] at (13.5,0.62) {$\fT_{i+1}^{R}$};
  \draw[decorate,decoration={brace,amplitude=6pt}] (0,1.55) -- (7,1.55);
  \draw[decorate,decoration={brace,amplitude=6pt}] (7,1.55) -- (14,1.55);
  \node at (3.5,2.15) {$Q_{i}$};
  \node at (10.5,2.15) {$Q_{i+1}$};
  \draw[decorate,decoration={brace,mirror,amplitude=6pt}] (5.2,-1.15) -- (8.8,-1.15);
  \node at (7,-1.85) {$\fS_{i+1}$};
  \draw[dashed,gray] (5.2,-1.15) -- (5.2,0.16);
  \draw[dashed,gray] (8.8,-1.15) -- (8.8,0.16);
\end{tikzpicture}
\caption{The blocks $Q_{i}=(q_{i},q_{i+1})$ and
$Q_{i+1}$. The shaded regions are the sets $\fT_{j}^{L}$
and $\fT_{j}^{R}$, $j\in\{i,i+1\}$, of length $t^{4}$. The set
$\fS_{i+1}$ has length $2e^{t/100}$ and is centered at
$q_{i+1}$, so that it contains the two adjacent shaded
regions.}
\label{fig:blocks}
\end{figure}

We first separate the interior and boundary contributions:
\begin{equation}\label{eq:interior_boundary_split}
    U(t)-\tilde{U}(t)
    = \sum_{i\in\cI(t)}\int_{\fR_{i}}
    \bigl[u(t,x)-\tilde{u}_{Q_{i}}(t,x)\bigr]\,dx
    + \sum_{i\in\cI(t)}\int_{Q_{i}\setminus\fR_{i}}
    \bigl[u(t,x)-\tilde{u}_{Q_{i}}(t,x)\bigr]\,dx\,.
\end{equation}
Further decomposing the boundary contribution, we obtain
\begin{equation}\label{eq:U_tildeU_error_decomposition}
    \begin{aligned}
        U(t)-\tilde{U}(t)
        &= \sum_{i\in\cI(t)}\int_{\fR_{i}}
        \bigl[u(t,x)-\tilde{u}_{Q_{i}}(t,x)\bigr]\,dx
        + \sum_{i\in\cI(t)}\int_{\fT_{i}^{L}}
        \bigl[u(t,x)-\tilde{u}_{\fS_{i}}(t,x)\bigr]\,dx\\
        &\qquad
        + \sum_{i\in\cI(t)}\int_{\fT_{i}^{R}}
        \bigl[u(t,x)-\tilde{u}_{\fS_{i+1}}(t,x)\bigr]\,dx
        + \sum_{i\in\cI(t)}\int_{\fT_{i}^{L}}
        \bigl[\tilde{u}_{\fS_{i}}(t,x)
        -\tilde{u}_{Q_{i}}(t,x)\bigr]\,dx\\
        &\qquad
        + \sum_{i\in\cI(t)}\int_{\fT_{i}^{R}}
        \bigl[\tilde{u}_{\fS_{i+1}}(t,x)
        -\tilde{u}_{Q_{i}}(t,x)\bigr]\,dx\,.
    \end{aligned}
\end{equation}
By \eqref{eq:def_Dirichlet_solution},
$0\le\tilde{u}_{Q}\le u$ on $Q$ for every interval $Q$,
so the integrands in the first three sums of
\eqref{eq:U_tildeU_error_decomposition} are nonnegative.
The following lemma shows that the corresponding
integrals have small moments: for every fixed $p>0$,
their $p$-th moments are $o(e^{-t^{6}})$,
uniformly in the block index. Since $n(t)=e^{O(t^{3})}$, each of these sums
therefore converges to zero in every fixed positive moment.
The remaining terms in either decomposition require
estimates adapted to the centering and scaling of
each theorem. We establish these estimates in
Lemmas~\ref{lem:WLLN_difference},
\ref{lem:CLT_difference}, and~\ref{lem:Stable_difference}.

\begin{lemma}\label{lem:approx_U_to_tilde_U}
    Let $p>0$ be fixed. As $t\to\infty$,
    \begin{enumerate}
        \item[\textup{(i)}]
        $\displaystyle
        \sup_{i\in\cI(t)}\dE\biggl[\biggl(\int_{\fR_{i}}
        \bigl[u(t,x)-\tilde{u}_{Q_{i}}(t,x)\bigr]\,dx
        \biggr)^{p}\biggr]
        = o\bigl(e^{-t^{6}}\bigr)\,,$
        \item[\textup{(ii)}]
        $\displaystyle
        \sup_{i\in\cI(t)}\dE\biggl[\biggl(\int_{\fT_{i}^{L}}
        \bigl[u(t,x)-\tilde{u}_{\fS_{i}}(t,x)\bigr]\,dx
        \biggr)^{p}\biggr]
        = o\bigl(e^{-t^{6}}\bigr)\,,$
        \item[\textup{(iii)}]
        $\displaystyle
        \sup_{i\in\cI(t)}\dE\biggl[\biggl(\int_{\fT_{i}^{R}}
        \bigl[u(t,x)-\tilde{u}_{\fS_{i+1}}(t,x)\bigr]\,dx
        \biggr)^{p}\biggr]
        = o\bigl(e^{-t^{6}}\bigr)\,.$
    \end{enumerate}
\end{lemma}
\begin{proof}
    We begin with \textup{(i)}. Fix $i\in\cI(t)$ and
    $x\in\fR_{i}$. By the Feynman--Kac representations
    \eqref{eq:FK_solution} and
    \eqref{eq:def_Dirichlet_solution}, and the
    Cauchy--Schwarz inequality,
    \begin{equation}\label{eq:FK_difference}
        \begin{aligned}
            u(t,x)-\tilde{u}_{Q_{i}}(t,x)
            &= \dE_{x}\Bigl[\exp\Bigl(\int_{0}^{t}
            \xi(W_{s})\,ds\Bigr)
            \,;\,\tau_{Q_{i}}\le t\Bigr]\\
            &\le \dP_{x}\bigl(\tau_{Q_{i}}\le t\bigr)^{1/2}\,
            \dE_{x}\Bigl[\exp\Bigl(2\int_{0}^{t}
            \xi(W_{s})\,ds\Bigr)\Bigr]^{1/2}\,.
        \end{aligned}
    \end{equation}
     Since $x$ is at distance at least $t^{4}$
    from $\partial Q_{i}$, the reflection principle gives
    \begin{equation}\label{eq:exit_estimate}
        \dP_{x}\bigl(\tau_{Q_{i}}\le t\bigr)
        \le 4\exp\Bigl(-\frac{t^{7}}{4}\Bigr)\,.
    \end{equation}
    For $b>0$, the inner expectation in
    \eqref{eq:noise_exponential_moment} below is the
    Feynman--Kac representation of the solution to
    \eqref{eq:PAM} with $\xi$ replaced by $b\,\xi$. Hence
    \cite[Theorem~2.6]{GLGL23} gives
    \begin{equation}\label{eq:noise_exponential_moment}
        \dE\Bigl[\dE_{x}\Bigl[\exp\Bigl(b\int_{0}^{t}
        \xi(W_{s})\,ds\Bigr)\Bigr]\Bigr]
        = e^{O(t^{3})}\,,
    \end{equation}
    where the implied constant depends only on $b$.
    Finally, recall from \eqref{eq:block_number} that
    $|\fR_{i}|\le\ell(t)=e^{t(1+o(1))}$.
    For $0<p\le1$, Jensen's inequality gives
    \begin{equation*}
        \begin{aligned}
            \dE\biggl[\biggl(\int_{\fR_{i}}
            \bigl[u(t,x)-\tilde{u}_{Q_{i}}(t,x)\bigr]\,dx
            \biggr)^{p}\biggr]
            &\le \biggl(\int_{\fR_{i}}
            \dE\bigl[u(t,x)-\tilde{u}_{Q_{i}}(t,x)\bigr]\,dx
            \biggr)^{p}\\
            &\le 2^{p}\,|\fR_{i}|^{p}\,e^{-p\,t^{7}/8}
            \sup_{x\in\fR_{i}}
            \dE\biggl[\dE_{x}\biggl[\exp\biggl(2\int_{0}^{t}
            \xi(W_{s})\,ds\biggr)\biggr]^{1/2}\biggr]^{p}\\
            &= o\bigl(e^{-t^{6}}\bigr)\,.
        \end{aligned}
    \end{equation*}
    Here the second inequality uses
    \eqref{eq:FK_difference} and \eqref{eq:exit_estimate}. The
    last equality follows from Jensen's inequality applied to
    the square root and \eqref{eq:noise_exponential_moment} with
    $b=2$. For $p>1$, H\"older's inequality gives
    \begin{equation*}
        \begin{aligned}
            \dE\biggl[\biggl(\int_{\fR_{i}}
            \bigl[u(t,x)-\tilde{u}_{Q_{i}}(t,x)\bigr]\,dx
            \biggr)^{p}\biggr]
            &\le |\fR_{i}|^{p-1}\int_{\fR_{i}}
            \dE\bigl[\bigl(u(t,x)
            -\tilde{u}_{Q_{i}}(t,x)\bigr)^{p}\bigr]\,dx\\
            &\le 2^{p}\,|\fR_{i}|^{p}\,e^{-p\,t^{7}/8}
            \sup_{x\in\fR_{i}}
            \dE\biggl[\dE_{x}\biggl[\exp\biggl(2\int_{0}^{t}
            \xi(W_{s})\,ds\biggr)\biggr]^{p/2}\biggr]\\
            &= o\bigl(e^{-t^{6}}\bigr)\,.
        \end{aligned}
    \end{equation*}
    Here the last equality follows from
    \eqref{eq:noise_exponential_moment}. For $1<p\le2$ we
    apply Jensen's inequality under $\dE$ and take $b=2$, and
    for $p>2$ we apply it under $\dE_{x}$ and take $b=p$.
    All bounds are
    independent of $i$, which proves \textup{(i)}.
    For \textup{(ii)}, the same argument applies with $Q_{i}$
    and $\fR_{i}$ replaced by $\fS_{i}$ and $\fT_{i}^{L}$.
    Each $x\in\fT_{i}^{L}$ is at distance at least
    $e^{t/100}-t^{4}$ from $\partial\fS_{i}$, so the reflection
    principle gives
    \begin{equation*}
        \dP_{x}\bigl(\tau_{\fS_{i}}\le t\bigr)
        \le 4\exp\Bigl(-\frac{(e^{t/100}-t^{4})^{2}}{4t}\Bigr)
        = o\bigl(e^{-t^{7}}\bigr)\,,
    \end{equation*}
    and the remaining estimates are unchanged. Finally,
    \textup{(iii)} follows in the same way, with
    $\fS_{i}$ and $\fT_{i}^{L}$ replaced by $\fS_{i+1}$ and
    $\fT_{i}^{R}$.
\end{proof}

\section{Proof of Theorem~\ref{thm:PAM_WLLN}}
\label{sec:proof_PAM_WLLN}
We first prove a weak law of large numbers for the sum
$\tilde{U}(t)$ (see \eqref{eq:tilde_U}) of the i.i.d.\ Dirichlet block integrals
\eqref{eq:def_Dirichlet_block_integral}. We then compare
$U(t)$ with $\tilde{U}(t)$, which gives
Theorem~\ref{thm:PAM_WLLN}.
\begin{theorem}\label{thm:PAM_WLLN_Dirichlet}
    Let $L(t)$ be as in \eqref{eq:box_size} with $\alpha>1$.
    Then, as $t\to\infty$,
    \begin{equation*}
        \frac{\tilde{U}(t)}{\dE[\tilde{U}(t)]}
        \xrightarrow{\ \dP\ } 1\,.
    \end{equation*}
\end{theorem}
\begin{proof}
    Throughout the proof, $C>0$ depends only on $\alpha$
    and may change from line to line.
    By Markov's inequality applied to
    $\bigl|\tilde{U}(t)-\dE[\tilde{U}(t)]\bigr|^{\alpha}$, it suffices to show that
    \begin{equation}\label{eq:WLLN_moment_ratio}
        \lim_{t\to\infty}
        \frac{\dE\Bigl[\bigl|\tilde{U}(t)
        -\dE[\tilde{U}(t)]\bigr|^{\alpha}\Bigr]}
        {\bigl(\dE[\tilde{U}(t)]\bigr)^{\alpha}}
        = 0\,.
    \end{equation}
    Since each $\tilde{U}_{i}(t)$ has the same law as
    $\tilde{U}_{0}(t)$, Remark~\ref{rem:window_condition_ell}
    and the lower bound \eqref{eq:moment_lower_bound} with
    $p=1$ give, for all sufficiently large $t$,
    \begin{equation}\label{eq:WLLN_denominator}
        \dE\bigl[\tilde{U}(t)\bigr]
        = n(t)\,\dE\bigl[\tilde{U}_{0}(t)\bigr]
        \ge C\,L(t)\,\exp\Bigl(\frac{t^{3}}{48}\Bigr)\,.
    \end{equation}
    By convexity of $x\mapsto x^{\alpha}$
    and \eqref{eq:moment_upper_bound} with $p=\alpha$,
    \begin{equation}\label{eq:WLLN_centered_moment}
        \dE\Bigl[\bigl|\tilde{U}_{0}(t)
        -\dE[\tilde{U}_{0}(t)]\bigr|^{\alpha}\Bigr]
        \le 2^{\alpha}\,\dE\bigl[\tilde{U}_{0}(t)^{\alpha}\bigr]
        \le C\,\ell(t)^{\alpha+1}\,t^{5/2}
        \exp\Bigl(\frac{\alpha^{3}}{48}\,t^{3}\Bigr)\,.
    \end{equation}
    For $\alpha\in(1,2]$, since the $\tilde{U}_{i}(t)$ are
    independent, the von Bahr--Esseen inequality \cite[Theorem~2]{vBE65}
    and \eqref{eq:WLLN_centered_moment} yield
    \begin{equation*}
        \dE\Bigl[\bigl|\tilde{U}(t)
        -\dE[\tilde{U}(t)]\bigr|^{\alpha}\Bigr]
        \le 2\sum_{i\in\cI(t)}
        \dE\Bigl[\bigl|\tilde{U}_{i}(t)
        -\dE[\tilde{U}_{i}(t)]\bigr|^{\alpha}\Bigr]
        \le C\,n(t)\,\ell(t)^{\alpha+1}\,t^{5/2}
        \exp\Bigl(\frac{\alpha^{3}}{48}\,t^{3}\Bigr)\,.
    \end{equation*}
    Since $n(t)\,\ell(t)^{\alpha+1}
    =L(t)\,\ell(t)^{\alpha}$ and
    $\ell(t)^{\alpha}\,t^{5/2}=e^{O(t)}$, combining this with
    \eqref{eq:WLLN_denominator} bounds the ratio in
    \eqref{eq:WLLN_moment_ratio} by
    \begin{equation*}
        C\,L(t)^{1-\alpha}
        \exp\Bigl(\frac{\alpha^{3}-\alpha}{48}\,t^{3}
        +O(t)\Bigr)\,.
    \end{equation*}
    By the definition of $L(t)$,
    $\log L(t)=\alpha^{3}t^{3}/24$, so the exponent equals
    \begin{equation*}
        \frac{t^{3}}{48}
        \Bigl(2(1-\alpha)\alpha^{3}+\alpha^{3}-\alpha\Bigr)+O(t)
        = -\frac{\alpha(\alpha-1)^{2}(2\alpha+1)}{48}\,t^{3}
        +O(t)
        \longrightarrow -\infty\,,
    \end{equation*}
    since $\alpha>1$.
    For $\alpha>2$, we apply Chebyshev's inequality
    instead, so it suffices to show
    \eqref{eq:WLLN_moment_ratio} with the exponent $\alpha$
    replaced by $2$. Since the $\tilde{U}_{i}(t)$ are
    independent and
    $\Var\bigl(\tilde{U}_{0}(t)\bigr)
    \le\dE\bigl[\tilde{U}_{0}(t)^{2}\bigr]$, the upper bound
    \eqref{eq:moment_upper_bound} with $p=2$ gives
    \begin{equation*}
        \dE\Bigl[\bigl(\tilde{U}(t)
        -\dE[\tilde{U}(t)]\bigr)^{2}\Bigr]
        = n(t)\,\Var\bigl(\tilde{U}_{0}(t)\bigr)
        \le C\,n(t)\,\ell(t)^{3}\,t^{5/2}
        \exp\Bigl(\frac{8}{48}\,t^{3}\Bigr)\,.
    \end{equation*}
    Since $n(t)\,\ell(t)^{3}=L(t)\,\ell(t)^{2}$ and
    $\ell(t)^{2}\,t^{5/2}=e^{O(t)}$, combining this with
    \eqref{eq:WLLN_denominator} bounds the resulting ratio by
    \begin{equation*}
        C\,L(t)^{-1}\exp\Bigl(\frac{6}{48}\,t^{3}+O(t)\Bigr)
        = C\exp\Bigl(-\frac{\alpha^{3}-3}{24}\,t^{3}+O(t)\Bigr)
        \longrightarrow 0\,.
    \end{equation*}
\end{proof}

We now compare the normalized sums appearing in
Theorems~\ref{thm:PAM_WLLN} and~\ref{thm:PAM_WLLN_Dirichlet}.
\begin{lemma}\label{lem:WLLN_difference}
    Let $L(t)$ be as in \eqref{eq:box_size}.
    Then, as $t\to\infty$,
    \begin{equation}\label{eq:WLLN_difference}
        \frac{U(t)}{\dE[U(t)]}
        - \frac{\tilde{U}(t)}{\dE[\tilde{U}(t)]}
        \xrightarrow{\ \dP\ } 0\,.
    \end{equation}
\end{lemma}
\begin{proof}
    Decompose the difference in \eqref{eq:WLLN_difference} as
    \begin{equation}\label{eq:WLLN_difference_decomposition}
        \frac{U(t)}{\dE[U(t)]}
        - \frac{\tilde{U}(t)}{\dE[\tilde{U}(t)]}
        = \frac{U(t)-\tilde{U}(t)}{\dE[U(t)]}
        + \frac{\tilde{U}(t)}{\dE[\tilde{U}(t)]}
        \Biggl(\frac{\dE[\tilde{U}(t)]}{\dE[U(t)]}-1\Biggr)\,.
    \end{equation}
    Since $\tilde{u}_{Q_{i}}\le u$ pointwise,
    $U(t)-\tilde{U}(t)\ge0$. It suffices to show that
    \begin{equation}\label{eq:WLLN_error_ratio}
        \lim_{t\to\infty}
        \frac{\dE\bigl[U(t)-\tilde{U}(t)\bigr]}{\dE[U(t)]}
        = 0\,.
    \end{equation}
    Indeed, the first term of
    \eqref{eq:WLLN_difference_decomposition} is nonnegative with
    expectation $\dE[U(t)-\tilde{U}(t)]/\dE[U(t)]$. The absolute
    value of the second term has the same expectation, since
    $\dE\bigl[\tilde{U}(t)/\dE[\tilde{U}(t)]\bigr]=1$. By
    Markov's inequality, both terms converge to $0$ in
    probability.
    We now prove \eqref{eq:WLLN_error_ratio}. By
    \eqref{eq:interior_boundary_split},
    Lemma~\ref{lem:approx_U_to_tilde_U}\,\textup{(i)} with $p=1$
    for the first sum, and $u-\tilde{u}_{Q_{i}}\le u$ for the
    second,
    \begin{equation}\label{eq:WLLN_error_split}
        \dE\bigl[U(t)-\tilde{U}(t)\bigr]
        \le n(t)\,o\bigl(e^{-t^{6}}\bigr)
        + \dE\Bigl[\sum_{i\in\cI(t)}
        \int_{Q_{i}\setminus\fR_{i}}u(t,x)\,dx\Bigr]\,.
    \end{equation}
    By Fubini's theorem and the stationarity of $u(t,\cdot)$,
    for every $i\in\cI(t)$ we have
    $\dE\bigl[\int_{Q_{i}\setminus\fR_{i}}u(t,x)\,dx\bigr]
    =2t^{4}\,\dE[u(t,0)]$ and
    $\dE[U(t)]=L(t)\,\dE[u(t,0)]$.
    Since $n(t)\le L(t)$ and $\dE[u(t,0)]\ge1$ by Jensen's
    inequality, $\dE[U(t)]\ge n(t)$,
    so the first term of \eqref{eq:WLLN_error_split} contributes
    $o(e^{-t^{6}})$ to \eqref{eq:WLLN_error_ratio}. For the
    second, we have
    \begin{equation*}
        \frac{\dE\bigl[\sum_{i\in\cI(t)}
        \int_{Q_{i}\setminus\fR_{i}}u(t,x)\,dx\bigr]}
        {\dE[U(t)]}
        = \frac{n(t)\cdot2t^{4}\,\dE[u(t,0)]}
        {L(t)\,\dE[u(t,0)]}
        = \frac{2t^{4}}{\ell(t)}
        \longrightarrow 0\,.
    \end{equation*}
    This proves \eqref{eq:WLLN_error_ratio}.
\end{proof}
\begin{proof}[Proof of Theorem~\ref{thm:PAM_WLLN}]
    By Lemma~\ref{lem:WLLN_difference} and
    Theorem~\ref{thm:PAM_WLLN_Dirichlet},
    \begin{equation*}
        \frac{U(t)}{\dE[U(t)]}
        = \Biggl(\frac{U(t)}{\dE[U(t)]}
        - \frac{\tilde{U}(t)}{\dE[\tilde{U}(t)]}\Biggr)
        + \frac{\tilde{U}(t)}{\dE[\tilde{U}(t)]}
        \xrightarrow{\ \dP\ } 1\,,
    \end{equation*}
    which is \eqref{eq:PAM_WLLN}.
\end{proof}

\section{Proof of Theorem~\ref{thm:PAM_CLT}}
\label{sec:proof_PAM_CLT}
As in the proof of Theorem~\ref{thm:PAM_WLLN}, we first prove a
central limit theorem for the sum $\tilde{U}(t)$. We
then compare $U(t)$ with $\tilde{U}(t)$ and conclude by
Slutsky's theorem.
\begin{theorem}\label{thm:PAM_CLT_Dirichlet}
    Let $L(t)$ be as in \eqref{eq:box_size} with $\alpha>2$.
    Then, as $t\to\infty$,
    \begin{equation*}
        \frac{\tilde{U}(t)-\dE[\tilde{U}(t)]}
        {\sqrt{\Var(\tilde{U}(t))}}
        \xrightarrow{\ d\ } \cN(0,1)\,.
    \end{equation*}
\end{theorem}
\begin{proof}
    By Lyapunov's central limit theorem
    \cite[Theorem~27.3]{Bil95}, it suffices to find $\delta>0$
    such that
    \begin{equation}\label{eq:Lyapunov_condition_tilde_U}
        \lim_{t\to\infty}
        \frac{\dE\Bigl[\bigl|\tilde{U}_{0}(t)
        -\dE[\tilde{U}_{0}(t)]\bigr|^{2+\delta}\Bigr]}
        {n(t)^{\delta/2}\,
        \Var\bigl(\tilde{U}_{0}(t)\bigr)^{1+\delta/2}}
        = 0\,.
    \end{equation}
    Since $\alpha^{3}>8$, we may fix $\delta>0$ with
    $(\delta+2)(\delta+4)<\alpha^{3}$. Throughout the proof,
    $C>0$ depends only on $\alpha$ and $\delta$ and may change
    from line to line.
    By convexity of $x\mapsto x^{2+\delta}$ and
    \eqref{eq:moment_upper_bound} with $p=2+\delta$,
    \begin{equation}\label{eq:CLT_moment_upper_tilde_U}
        \dE\Bigl[\bigl|\tilde{U}_{0}(t)
        -\dE[\tilde{U}_{0}(t)]\bigr|^{2+\delta}\Bigr]
        \le 2^{2+\delta}\,
        \dE\bigl[\tilde{U}_{0}(t)^{2+\delta}\bigr]
        \le C\,\ell(t)^{3+\delta}\,t^{5/2}
        \exp\Bigl(\frac{(2+\delta)^{3}}{48}\,t^{3}\Bigr)\,.
    \end{equation}
    For the variance, Remark~\ref{rem:window_condition_ell},
    \eqref{eq:moment_lower_bound} with $p=2$, and
    \eqref{eq:moment_upper_bound} with $p=1$ give
    \begin{equation}\label{eq:CLT_variance_lower_tilde_U}
        \Var\bigl(\tilde{U}_{0}(t)\bigr)
        = \dE\bigl[\tilde{U}_{0}(t)^{2}\bigr]
        - \bigl(\dE[\tilde{U}_{0}(t)]\bigr)^{2}
        \ge C\,\ell(t)\,t^{-1}
        \exp\Bigl(\frac{8}{48}\,t^{3}\Bigr)
    \end{equation}
    for all sufficiently large $t$. Indeed, the
    subtracted term is at most
    $C\,\ell(t)^{4}\,t^{5}\exp\bigl(\tfrac{2}{48}t^{3}\bigr)$.
    Combining \eqref{eq:CLT_moment_upper_tilde_U} and
    \eqref{eq:CLT_variance_lower_tilde_U} with
    $n(t)=L(t)/\ell(t)$ and $\ell(t)=e^{t(1+o(1))}$
    bounds the ratio in \eqref{eq:Lyapunov_condition_tilde_U}
    by
    \begin{equation*}
        C\,L(t)^{-\delta/2}
        \exp\Bigl(\frac{(2+\delta)^{3}-(8+4\delta)}{48}\,t^{3}
        +O(t)\Bigr)\,.
    \end{equation*}
    By the definition of $L(t)$,
    $\log L(t)=\alpha^{3}t^{3}/24$, so the exponent equals
    \begin{equation*}
        \frac{t^{3}}{48}
        \Bigl((2+\delta)^{3}-8-4\delta-\delta\alpha^{3}\Bigr)+O(t)
        = \frac{\delta\,t^{3}}{48}
        \Bigl((\delta+2)(\delta+4)-\alpha^{3}\Bigr)+O(t)
        \longrightarrow -\infty
    \end{equation*}
    by the choice of $\delta$. This proves
    \eqref{eq:Lyapunov_condition_tilde_U}.
\end{proof}

We now compare the centered and normalized sums appearing
in Theorems~\ref{thm:PAM_CLT} and~\ref{thm:PAM_CLT_Dirichlet}.
\begin{lemma}\label{lem:CLT_difference}
    Let $L(t)$ be as in \eqref{eq:box_size}.
    Then, as $t\to\infty$,
    \begin{equation}\label{eq:CLT_difference}
        \frac{U(t)-\dE[U(t)]}{\sqrt{\Var(U(t))}}
        - \frac{\tilde{U}(t)-\dE[\tilde{U}(t)]}
        {\sqrt{\Var(\tilde{U}(t))}}
        \xrightarrow{\ \dP\ } 0\,.
    \end{equation}
\end{lemma}
\begin{proof}
    Write $X(t):=U(t)-\dE[U(t)]$ and
    $\tilde{X}(t):=\tilde{U}(t)-\dE[\tilde{U}(t)]$.
    Throughout the proof, $C>0$ depends only on $\alpha$
    and may change from line to line. Since the
    $\tilde{U}_{i}(t)$ are i.i.d.,
    \eqref{eq:CLT_variance_lower_tilde_U} gives, for all
    sufficiently large $t$,
    \begin{equation}\label{eq:Var_tilde_U_lower}
        \Var\bigl(\tilde{U}(t)\bigr)
        = n(t)\,\Var\bigl(\tilde{U}_{0}(t)\bigr)
        \ge C\,n(t)\,\ell(t)\,t^{-1}
        \exp\Bigl(\frac{t^{3}}{6}\Bigr)\,.
    \end{equation}
    We first claim that
    \begin{equation}\label{eq:centered_error_var_negligible}
        \lim_{t\to\infty}
        \frac{\dE\bigl[\bigl(X(t)-\tilde{X}(t)\bigr)^{2}\bigr]}
        {\Var(\tilde{U}(t))}
        = 0\,.
    \end{equation}
    By \eqref{eq:U_tildeU_error_decomposition},
    $X(t)-\tilde{X}(t)$ is the sum of five centered terms, and
    it suffices to show that the variance of each is
    $o(\Var(\tilde{U}(t)))$. Three of them are covered by
    Lemma~\ref{lem:approx_U_to_tilde_U}, which with
    $p=2$ bounds the second moment of each of their $n(t)$
    summands by $o(e^{-t^{6}})$. Since $n(t)=e^{O(t^{3})}$, the
    variance of each of these three sums is at most
    $n(t)^{2}\,o(e^{-t^{6}})=o(1)$, hence
    $o(\Var(\tilde{U}(t)))$ by
    \eqref{eq:Var_tilde_U_lower}.
    The two remaining sums compare $\tilde{u}_{\fS_{i}}$
    with $\tilde{u}_{Q_{i}}$ on $\fT_{i}^{L}$, and
    $\tilde{u}_{\fS_{i+1}}$ with $\tilde{u}_{Q_{i}}$ on
    $\fT_{i}^{R}$. We treat the first, the second being
    analogous. Since $\Var(Z-Z')\le2\Var(Z)+2\Var(Z')$,
    \begin{equation}\label{eq:boundary_sum_split}
        \begin{aligned}
            &\Var\biggl(\sum_{i\in\cI(t)}\int_{\fT_{i}^{L}}
            \bigl[\tilde{u}_{\fS_{i}}(t,x)
            -\tilde{u}_{Q_{i}}(t,x)\bigr]\,dx\biggr)\\
            &\qquad\le 2\Var\biggl(\sum_{i\in\cI(t)}
            \int_{\fT_{i}^{L}}\tilde{u}_{\fS_{i}}(t,x)\,dx
            \biggr)
            + 2\Var\biggl(\sum_{i\in\cI(t)}\int_{\fT_{i}^{L}}
            \tilde{u}_{Q_{i}}(t,x)\,dx\biggr)\,.
        \end{aligned}
    \end{equation}
    For the first term on the right-hand side of
    \eqref{eq:boundary_sum_split},
    the $\fS_{i}$ are pairwise disjoint for all sufficiently large
    $t$, so the integrals over the $\fT_{i}^{L}$ are
    i.i.d., and, since $\fT_{1}^{L}\subset\fS_{1}$ and
    $\tilde{u}_{\fS_{1}}\ge0$,
    \begin{equation}\label{eq:window_term_bound}
        \begin{aligned}
            \Var\biggl(\sum_{i\in\cI(t)}\int_{\fT_{i}^{L}}
            \tilde{u}_{\fS_{i}}(t,x)\,dx\biggr)
            &= n(t)\,\Var\Bigl(\int_{\fT_{1}^{L}}
            \tilde{u}_{\fS_{1}}(t,x)\,dx\Bigr)\\
            &\le n(t)\,\dE\biggl[\Bigl(\int_{\fS_{1}}
            \tilde{u}_{\fS_{1}}(t,x)\,dx\Bigr)^{2}\biggr]\\
            &\le C\,n(t)\,e^{3t/100}\,t^{5/2}
            \exp\Bigl(\frac{t^{3}}{6}\Bigr)\,.
        \end{aligned}
    \end{equation}
    Here the last step is \eqref{eq:moment_upper_bound} with
    $p=2$ and $\cL=2e^{t/100}$. Combining
    \eqref{eq:window_term_bound} with
    \eqref{eq:Var_tilde_U_lower},
    \begin{equation*}
        \frac{\Var\Bigl(\sum_{i\in\cI(t)}
        \int_{\fT_{i}^{L}}\tilde{u}_{\fS_{i}}(t,x)\,dx\Bigr)}
        {\Var(\tilde{U}(t))}
        \le C\,t^{7/2}\,\frac{e^{3t/100}}{\ell(t)}
        \longrightarrow 0\,,
    \end{equation*}
    since $\ell(t)=e^{t(1+o(1))}$.
    For the second term, the $Q_{i}$ are disjoint, so
    these integrals are again i.i.d. Set
    $D_{1}:=(q_{1},\,q_{1}+e^{t/100})\subset Q_{1}$, so that
    $\fT_{1}^{L}\subset D_{1}$ and
    $\tilde{u}_{D_{1}}\le\tilde{u}_{Q_{1}}$ on $\fT_{1}^{L}$,
    with
    \begin{equation*}
        \tilde{u}_{Q_{1}}(t,x)-\tilde{u}_{D_{1}}(t,x)
        = \dE_{x}\Bigl[\exp\Bigl(\int_{0}^{t}
        \xi(W_{s})\,ds\Bigr)
        \,;\,\tau_{D_{1}}\le t<\tau_{Q_{1}}\Bigr]\,,
        \qquad x\in\fT_{1}^{L}\,.
    \end{equation*}
    On this event, $W$ reaches $q_{1}+e^{t/100}$ before
    time $t$, at distance at least $e^{t/100}-t^{4}$ from $x$.
    As in the proof of
    Lemma~\ref{lem:approx_U_to_tilde_U}\,\textup{(ii)}, the
    second moment of $\int_{\fT_{1}^{L}}
    \bigl[\tilde{u}_{Q_{1}}(t,x)
    -\tilde{u}_{D_{1}}(t,x)\bigr]\,dx$ is $o(e^{-t^{6}})$.
    Hence, as in \eqref{eq:window_term_bound},
    \begin{equation*}
        \begin{aligned}
            \Var\biggl(\sum_{i\in\cI(t)}\int_{\fT_{i}^{L}}
            \tilde{u}_{Q_{i}}(t,x)\,dx\biggr)
            &\le 2\,n(t)\,\dE\biggl[\Bigl(\int_{D_{1}}
            \tilde{u}_{D_{1}}(t,x)\,dx\Bigr)^{2}\biggr]
            + n(t)\,o\bigl(e^{-t^{6}}\bigr)\\
            &\le C\,n(t)\,e^{3t/100}\,t^{5/2}
            \exp\Bigl(\frac{t^{3}}{6}\Bigr)\,,
        \end{aligned}
    \end{equation*}
    and, combining this with \eqref{eq:Var_tilde_U_lower} as
    before,
    \begin{equation*}
        \frac{\Var\Bigl(\sum_{i\in\cI(t)}
        \int_{\fT_{i}^{L}}\tilde{u}_{Q_{i}}(t,x)\,dx\Bigr)}
        {\Var(\tilde{U}(t))}
        \le C\,t^{7/2}\,\frac{e^{3t/100}}{\ell(t)}
        \longrightarrow 0\,.
    \end{equation*}
    This proves \eqref{eq:centered_error_var_negligible}.
 
    We now conclude. By Minkowski's inequality and
    \eqref{eq:centered_error_var_negligible},
    \begin{equation*}
        \Bigl|\sqrt{\Var(U(t))}-\sqrt{\Var(\tilde{U}(t))}\Bigr|
        \le \dE\bigl[\bigl(X(t)-\tilde{X}(t)\bigr)^{2}\bigr]^{1/2}
        = o\Bigl(\sqrt{\Var(\tilde{U}(t))}\Bigr)\,,
    \end{equation*}
    so that $\Var(U(t))/\Var(\tilde{U}(t))\to1$. We decompose
    \begin{equation*}
        \frac{X(t)}{\sqrt{\Var(U(t))}}
        - \frac{\tilde{X}(t)}{\sqrt{\Var(\tilde{U}(t))}}
        = \frac{X(t)-\tilde{X}(t)}{\sqrt{\Var(U(t))}}
        + \tilde{X}(t)
        \biggl(\frac{1}{\sqrt{\Var(U(t))}}
        -\frac{1}{\sqrt{\Var(\tilde{U}(t))}}\biggr)\,.
    \end{equation*}
    By \eqref{eq:centered_error_var_negligible} and the
    convergence $\Var(U(t))/\Var(\tilde{U}(t))\to1$, the first
    term on the right-hand side satisfies
    \begin{equation*}
        \dE\Biggl[\biggl(
        \frac{X(t)-\tilde{X}(t)}{\sqrt{\Var(U(t))}}\biggr)^{2}
        \Biggr]
        = \frac{\dE\bigl[\bigl(X(t)-\tilde{X}(t)\bigr)^{2}\bigr]}
        {\Var(\tilde{U}(t))}\,
        \frac{\Var(\tilde{U}(t))}{\Var(U(t))}
        \longrightarrow 0\,,
    \end{equation*}
    so, by Chebyshev's inequality, it converges to $0$ in
    probability. The second term satisfies
    \begin{equation*}
        \dE\Biggl[\tilde{X}(t)^{2}
        \biggl(\frac{1}{\sqrt{\Var(U(t))}}
        -\frac{1}{\sqrt{\Var(\tilde{U}(t))}}\biggr)^{2}\Biggr]
        = \biggl(\sqrt{\frac{\Var(\tilde{U}(t))}{\Var(U(t))}}
        -1\biggr)^{2}
        \longrightarrow 0\,,
    \end{equation*}
    so it converges to $0$ in probability as well, which proves
    \eqref{eq:CLT_difference}.
\end{proof}

\begin{proof}[Proof of Theorem~\ref{thm:PAM_CLT}]
    By Lemma~\ref{lem:CLT_difference} and
    Theorem~\ref{thm:PAM_CLT_Dirichlet}, Slutsky's theorem
    gives
    \begin{equation*}
            \frac{U(t)-\dE[U(t)]}{\sqrt{\Var(U(t))}}
            = \Biggl(\frac{U(t)-\dE[U(t)]}{\sqrt{\Var(U(t))}}
            - \frac{\tilde{U}(t)-\dE[\tilde{U}(t)]}
            {\sqrt{\Var(\tilde{U}(t))}}\Biggr)
            + \frac{\tilde{U}(t)-\dE[\tilde{U}(t)]}
            {\sqrt{\Var(\tilde{U}(t))}}
            \xrightarrow{\ d\ } \cN(0,1)\,,
    \end{equation*}
    which is \eqref{eq:PAM_CLT}.
\end{proof}

\section{Proof of Theorem~\ref{thm:PAM_Stable}}
\label{sec:proof_PAM_Stable}
As in the previous two sections, we first prove the
stable limit theorem for $\tilde{U}(t)$ and then compare
$U(t)$ with $\tilde{U}(t)$. The limit law is identified through
a classical convergence criterion for triangular arrays, which
we now recall. A random variable is infinitely divisible if
and only if its characteristic function has the form
\begin{equation}\label{eq:LK_representation}
    \phi(z) = \exp\biggl(i\nu z - \frac{\sigma^{2}z^{2}}{2}
        + \int_{|x|>0}\Bigl(e^{izx}-1-\frac{izx}{1+x^{2}}\Bigr)
    \,d\fL(x)\biggr)\,,
    \qquad z\in\bR\,,
\end{equation}
where $\nu\in\bR$, $\sigma^{2}\ge0$, and $\fL$ is nondecreasing
on $(-\infty,0)$ and on $(0,\infty)$, with
$\lim_{x\to-\infty}\fL(x)=\lim_{x\to\infty}\fL(x)=0$ and
$\int_{0<|x|\le\ve}x^{2}\,d\fL(x)<\infty$ for every $\ve>0$. We
call $\fL$ the \emph{L\'evy--Khintchine spectral function} (see
\cite[Chapter~II, Theorem~5]{Pet75}), and denote by
$\dX_{\nu,\sigma^{2},\fL}$ the infinitely divisible random
variable with characteristic function
\eqref{eq:LK_representation}.
The following proposition states this criterion in the
notation of our setting (see
\cite[Chapter~IV, Theorems~7 and~8]{Pet75} and
\cite[Theorem~5]{BAMR19}).

\begin{proposition}\label{prop:stable_convergence_criterion}
    For each $t\ge0$, let $(\tilde{Y}_{i}(t))_{i\in\cI(t)}$ be
    i.i.d.\ random variables, where $n(t)=|\cI(t)|\to\infty$ as
    $t\to\infty$, and let $\tilde{Y}_{0}(t)$ denote a random
    variable with their common law. Assume that, for every
    $\ve>0$,
    \begin{equation}\label{eq:UAN_condition}
        \lim_{t\to\infty}\dP\bigl(|\tilde{Y}_{0}(t)|>\ve\bigr)
        = 0\,.
    \end{equation}
    Let $\fL$ be a L\'evy--Khintchine spectral function,
    $\nu\in\bR$, $\sigma^{2}\ge0$, and let $\tilde{A}(t)$ be a
    real-valued function. For $y>0$, set
    $Z_{y}(t):=\tilde{Y}_{0}(t)\1_{\{|\tilde{Y}_{0}(t)|\le y\}}$.
    Then the following are equivalent.
    \begin{enumerate}
        \item[\textup{(i)}] As $t\to\infty$,
        \begin{equation*}
            \sum_{i\in\cI(t)}\tilde{Y}_{i}(t)-\tilde{A}(t)
            \xrightarrow{\ d\ } \dX_{\nu,\sigma^{2},\fL}\,.
        \end{equation*}
        \item[\textup{(ii)}] For every continuity point $x$ of
        $\fL$,
        \begin{equation}
        \label{eq:stable_criterion_spectral_function}
            \fL(x) =
            \begin{cases}
                \phantom{-}\lim_{t\to\infty}
                n(t)\,\dP\bigl(\tilde{Y}_{0}(t)\le x\bigr)\,,
                & x<0\,,\\
                -\lim_{t\to\infty}
                n(t)\,\dP\bigl(\tilde{Y}_{0}(t)>x\bigr)\,,
                & x>0\,,
            \end{cases}
        \end{equation}
        \begin{equation}\label{eq:stable_criterion_variance}
            \sigma^{2}
            = \lim_{y\to0}\limsup_{t\to\infty}
            n(t)\,\Var\bigl(Z_{y}(t)\bigr)
            = \lim_{y\to0}\liminf_{t\to\infty}
            n(t)\,\Var\bigl(Z_{y}(t)\bigr)\,,
        \end{equation}
        and, for any $y>0$ such that $\pm y$ are continuity
        points of $\fL$,
        \begin{equation}\label{eq:stable_criterion_centering}
            \nu = \lim_{t\to\infty}
            \Bigl(n(t)\,\dE\bigl[Z_{y}(t)\bigr]
            -\tilde{A}(t)\Bigr)
            + \int_{|x|>y}\frac{x}{1+x^{2}}\,d\fL(x)
            - \int_{0<|x|\le y}\frac{x^{3}}{1+x^{2}}\,d\fL(x)\,.
        \end{equation}
    \end{enumerate}
\end{proposition}

Throughout this section, $L(t)$ is as in
\eqref{eq:box_size} with $\alpha\in(0,2)$.
We now apply Proposition~\ref{prop:stable_convergence_criterion}
to the family
\begin{equation}\label{eq:def_tilde_Y_i}
    \tilde{Y}_{i}(t) := \frac{\tilde{U}_{i}(t)}{B_{\alpha}(t)}\,,
    \qquad i\in\cI(t)\,,
\end{equation}
with $B_{\alpha}(t)$ as in \eqref{eq:stable_scale_B_alpha}, and
write $\tilde{Y}_{0}(t):=\tilde{U}_{0}(t)/B_{\alpha}(t)$.
Since
$\sum_{i\in\cI(t)}\tilde{Y}_{i}(t)=\tilde{U}(t)/B_{\alpha}(t)$,
the following theorem is the stable limit theorem for $\tilde{U}(t)$.
\begin{theorem}\label{thm:PAM_Stable_Dirichlet}
    Let $L(t)$ be as in \eqref{eq:box_size} with
    $\alpha\in(0,2)$. Then, as $t\to\infty$,
    \begin{equation}\label{eq:PAM_Stable_Dirichlet}
        \sum_{i\in\cI(t)}\tilde{Y}_{i}(t)-\tilde{A}(t)
        \xrightarrow{\ d\ } \dS_{\alpha}\,,
    \end{equation}
    where $\dS_{\alpha}$ is as in Theorem~\ref{thm:PAM_Stable}
    and
    \begin{equation}\label{eq:def_stable_centering}
        \tilde{A}(t) :=
        \begin{cases}
            0\,, & \alpha\in(0,1)\,,\\
            n(t)\,\dE\bigl[\tilde{Y}_{0}(t)\,;\,
            \tilde{Y}_{0}(t)\le1\bigr]\,, & \alpha=1\,,\\
            n(t)\,\dE\bigl[\tilde{Y}_{0}(t)\bigr]\,,
            & \alpha\in(1,2)\,.
        \end{cases}
    \end{equation}
\end{theorem}
We first verify condition
\eqref{eq:stable_criterion_spectral_function} with the spectral function
\begin{equation}\label{eq:tilde_Y_alpha_spectral_function}
    \fL(x) :=
    \begin{cases}
        0\,, & x<0\,,\\
        -x^{-\alpha}\,, & x>0\,.
    \end{cases}
\end{equation}
Since $\tilde{Y}_{0}(t)\ge0$, the condition for $x<0$
holds trivially, and it remains to verify the case $x>0$.
\begin{proposition}\label{prop:tilde_Y_positive_tail}
    For every $x>0$,
    \begin{equation}\label{eq:tilde_Y_positive_tail}
        \lim_{t\to\infty}
        n(t)\,\dP\bigl(\tilde{Y}_{0}(t)>x\bigr)
        = x^{-\alpha}\,.
    \end{equation}
\end{proposition}

\begin{proof}
    Throughout the proof, we denote by
    $(\lambda_{k},\varphi_{k})_{k\ge1}$ the eigenpairs of
    $\cH_{Q_{\ell(t)}}$, so that, by
    \eqref{eq:spatial_integral_spectral_rep} and
    $\varphi_{1}\ge0$,
    \begin{equation}\label{eq:spectral_representation_tilde_U}
        \tilde{U}_{0}(t)
        = \sum_{k\ge1}e^{-t\lambda_{k}}
        \Bigl(\int_{Q_{\ell(t)}}\varphi_{k}(x)\,dx\Bigr)^{2}
        \ge e^{-t\lambda_{1}}\,\|\varphi_{1}\|_{1}^{2}\,.
    \end{equation}
    Fix $\theta\in(0,\,1-2^{-2/3})$ and
    $\ve\in(0,\,\pi/\sqrt{2})$, and set
    \begin{equation}\label{eq:def_stable_level}
        \tilde{a}(t)
        := \frac{\log B_{\alpha}(t)-\log\ell(t)
        -t^{3/2}}{t}\,.
    \end{equation}
    By \eqref{eq:stable_scale_B_alpha} and
    $\log\ell(t)=t+o(1)$,
    \begin{equation}\label{eq:stable_level_expansion}
        \tilde{a}(t)=\frac{\alpha^{2}t^{2}}{16}+O(t^{1/2})\,,
        \qquad
        \tilde{a}(t)^{1/2}=\frac{\alpha t}{4}+O(t^{-1/2})\,.
    \end{equation}
    Define the events
    \begin{equation*}
        \begin{aligned}
            G_{1} &:= \bigl\{\lambda_{1}<-\tilde{a}(t)\bigr\}\,,
            \qquad
            G_{2} := \bigl\{\lambda_{2}\ge
            -(1-\theta)\,\tilde{a}(t)\bigr\}\,,\\
            G_{3} &:= \biggl\{\Bigl|\tilde{a}(t)^{1/4}
            \|\varphi_{1}\|_{1}-\frac{\pi}{\sqrt{2}}\Bigr|
            \le\ve\biggr\}\,,
        \end{aligned}
    \end{equation*}
    and decompose
    \begin{equation}\label{eq:tilde_Y_tail_decomposition}
        \begin{aligned}
            \dP\bigl(\tilde{Y}_{0}(t)>x\bigr)
            &= \dP\bigl(\tilde{Y}_{0}(t)>x,\,G_{1}^{c}\bigr)
            + \dP\bigl(\tilde{Y}_{0}(t)>x,\,G_{1},\,
            G_{2}^{c}\bigr)\\
            &\qquad
            + \dP\bigl(\tilde{Y}_{0}(t)>x,\,G_{1},\,G_{2},\,
            G_{3}^{c}\bigr)
            + \dP\bigl(\tilde{Y}_{0}(t)>x,\,G_{1},\,G_{2},\,
            G_{3}\bigr)\,.
        \end{aligned}
    \end{equation}
    We show that the first three terms are negligible after
    multiplication by $n(t)$, and that the last term produces
    the limit $x^{-\alpha}$.

    The first term of \eqref{eq:tilde_Y_tail_decomposition}
    vanishes for deterministic reasons.
    Since $\sum_{k\ge1}\bigl(\int_{Q_{\ell(t)}}\varphi_{k}(x)\,dx \bigr)^{2}=\ell(t)$
    by Parseval's identity, on $G_{1}^{c}$ we have
    \begin{equation}\label{eq:G_1_complement_bound}
        \tilde{Y}_{0}(t)
        \le \frac{\ell(t)\,e^{-t\lambda_{1}}}{B_{\alpha}(t)}
        \le \frac{\ell(t)\,e^{t\tilde{a}(t)}}{B_{\alpha}(t)}
        = e^{-t^{3/2}}\,,
    \end{equation}
    by the definition \eqref{eq:def_stable_level} of
    $\tilde{a}(t)$. Hence, for all sufficiently large $t$,
    \begin{equation}\label{eq:G_1_complement_negligible}
        n(t)\,\dP\bigl(\tilde{Y}_{0}(t)>x,\,G_{1}^{c}
        \bigr) = 0\,.
    \end{equation}
    For the second term, set
    $\tilde{b}(t):=(1-\theta)\,\tilde{a}(t)$.
    Corollary~\ref{cor:second_eigenvalue_left_tail} with
    $\cL=\ell(t)$ and $a=\tilde{b}(t)$ gives
    \begin{equation}\label{eq:G_2_complement_bound}
        \dP\bigl(G_{2}^{c}\bigr)
        = \dP\bigl(\lambda_{2}<-\tilde{b}(t)\bigr)
        \le \biggl(\frac{e\,\ell(t)}{m(\tilde{b}(t))}
        \biggr)^{2}\,.
    \end{equation}
    By \eqref{eq:stable_level_expansion} and
    \eqref{eq:first_explosion_mean_asymptotics},
    $\log m(\tilde{b}(t))
    =(1-\theta)^{3/2}\alpha^{3}t^{3}/24+O(t^{3/2})$. As
    $\log n(t)\le\log L(t)=\alpha^{3}t^{3}/24$ and
    $\log\ell(t)=O(t)$, \eqref{eq:G_2_complement_bound} 
    gives
    \begin{equation*}
        \log\Bigl(n(t)\,\dP\bigl(G_{2}^{c}\bigr)\Bigr)
        \le -\frac{2(1-\theta)^{3/2}-1}{24}\,\alpha^{3}t^{3}
        +O(t^{3/2})
        \longrightarrow -\infty
    \end{equation*}
    by the choice of $\theta$. In particular,
    \begin{equation}\label{eq:G_2_complement_negligible}
        \lim_{t\to\infty} n(t)\,\dP\bigl(G_{2}^{c}\bigr)
        = 0\,.
    \end{equation}
    For the third term,
    Proposition~\ref{prop:eigenfunction_L1_norm} with
    $\cL=\ell(t)$ and $a=\tilde{a}(t)$ gives, for all
    sufficiently large $t$,
    \begin{equation}\label{eq:G_3_complement_bound}
        \dP\bigl(G_{1}\cap G_{2}\cap G_{3}^{c}\bigr)
        \le \bigl(\ell(t)\,\tilde{a}(t)^{1/2}+2\bigr)
        \exp\Bigl(-\Bigl(\frac{8}{3}+c_{1}\Bigr)
        \tilde{a}(t)^{3/2}
        +c_{2}\,\tilde{a}(t)^{3/4}\Bigr)\,,
    \end{equation}
    where $c_{1},c_{2}>0$ depend only on $\theta$ and $\ve$.
    By \eqref{eq:stable_level_expansion},
    $\frac{8}{3}\tilde{a}(t)^{3/2}
    =\alpha^{3}t^{3}/24+O(t^{3/2})$. Since
    $\log n(t)\le\alpha^{3}t^{3}/24$,
    \eqref{eq:G_3_complement_bound} gives
    \begin{equation*}
        \log\Bigl(n(t)\,\dP\bigl(G_{1}\cap G_{2}\cap
        G_{3}^{c}\bigr)\Bigr)
        \le -\frac{c_{1}\,\alpha^{3}}{64}\,t^{3}+O(t^{3/2})
        \longrightarrow -\infty\,,
    \end{equation*}
    and hence
    \begin{equation}\label{eq:G_3_complement_negligible}
        \lim_{t\to\infty}
        n(t)\,\dP\bigl(G_{1}\cap G_{2}\cap G_{3}^{c}\bigr)
        = 0\,.
    \end{equation}
    Combining \eqref{eq:G_1_complement_negligible},
    \eqref{eq:G_2_complement_negligible}, and
    \eqref{eq:G_3_complement_negligible} with
    \eqref{eq:tilde_Y_tail_decomposition},
    \begin{equation}
    \label{eq:tilde_Y_tail_reduction_to_good_event}
        n(t)\,\dP\bigl(\tilde{Y}_{0}(t)>x\bigr)
        = n(t)\,\dP\bigl(\tilde{Y}_{0}(t)>x,\,G_{1},\,G_{2},\,
        G_{3}\bigr)
        + o(1)\,.
    \end{equation}
    On the event $G_{1}\cap G_{2}\cap G_{3}$, we bound
    $\tilde{U}_{0}(t)$ from both sides. The spectral
    representation \eqref{eq:spectral_representation_tilde_U}
    and the definition of $G_{3}$ give the lower bound
    \begin{equation}
    \label{eq:lower_bound_spectral_representation_tilde_U}
        \tilde{U}_{0}(t)
        \ge \|\varphi_{1}\|_{1}^{2}\,e^{-t\lambda_{1}}
        \ge \Bigl(\frac{\pi}{\sqrt{2}}-\ve\Bigr)^{2}
        \tilde{a}(t)^{-1/2}\,e^{-t\lambda_{1}}\,.
    \end{equation}
    For the upper bound, the sum over $k\ge2$ in
    \eqref{eq:spectral_representation_tilde_U} is at most
    $\ell(t)\,e^{-t\lambda_{2}}$ by Parseval's identity.
    Moreover $\lambda_{2}-\lambda_{1}>\theta\,\tilde{a}(t)$ on
    $G_{1}\cap G_{2}$. Hence
    \begin{equation}
    \label{eq:upper_bound_spectral_representation_tilde_U}
        \begin{aligned}
            \tilde{U}_{0}(t)
            &\le \|\varphi_{1}\|_{1}^{2}\,e^{-t\lambda_{1}}
            \biggl(1+\frac{\ell(t)}{\|\varphi_{1}\|_{1}^{2}}\,
            e^{-t(\lambda_{2}-\lambda_{1})}\biggr)\\
            &\le \Bigl(\frac{\pi}{\sqrt{2}}+\ve\Bigr)^{2}
            \tilde{a}(t)^{-1/2}\,e^{-t\lambda_{1}}
            \biggl(1+\Bigl(\frac{\pi}{\sqrt{2}}-\ve\Bigr)^{-2}
            \ell(t)\,\tilde{a}(t)^{1/2}\,
            e^{-\theta t\tilde{a}(t)}\biggr)\,,
        \end{aligned}
    \end{equation}
    where the second factor is $1+o(1)$, since
    $\log\ell(t)=O(t)$ while $\theta t\tilde{a}(t)\ge ct^{3}$
    for some $c>0$.

    We now bound the right-hand side of
    \eqref{eq:tilde_Y_tail_reduction_to_good_event} from
    below. Define
    \begin{equation}\label{eq:def_H_t}
        H_{t} := \biggl\{
        \frac{(\pi/\sqrt{2}-\ve)^{2}\,e^{-t\lambda_{1}}}
        {\tilde{a}(t)^{1/2}\,B_{\alpha}(t)}>x\biggr\}\,.
    \end{equation}
    By \eqref{eq:lower_bound_spectral_representation_tilde_U},
    $H_{t}\cap G_{3}\subset
    \{\tilde{Y}_{0}(t)>x\}\cap G_{3}$. Together with the
    union bound, this gives
    \begin{equation*}
        \dP\bigl(\tilde{Y}_{0}(t)>x,\,G_{1},\,G_{2},\,
        G_{3}\bigr)
        \ge \dP\bigl(H_{t}\bigr)
        - \dP\bigl(H_{t}\cap G_{1}^{c}\bigr)
        - \dP\bigl(G_{2}^{c}\bigr)
        - \dP\bigl(G_{1}\cap G_{2}\cap G_{3}^{c}\bigr)\,.
    \end{equation*}
    As in \eqref{eq:G_1_complement_bound},
    $H_{t}\cap G_{1}^{c}$ is empty for all sufficiently large
    $t$. Multiplying by $n(t)$ and using
    \eqref{eq:G_2_complement_negligible} and
    \eqref{eq:G_3_complement_negligible},
    \begin{equation}\label{eq:liminf_reduction}
        \liminf_{t\to\infty}
        n(t)\,\dP\bigl(\tilde{Y}_{0}(t)>x,\,G_{1},\,G_{2},\,
        G_{3}\bigr)
        \ge \liminf_{t\to\infty}
        n(t)\,\dP\bigl(H_{t}\bigr)\,.
    \end{equation}
    By the definition \eqref{eq:def_H_t} of $H_{t}$,
    \begin{equation}\label{eq:level_event_identification}
        \dP\bigl(H_{t}\bigr)
        = \dP\bigl(\lambda_{1}<-a_{1}(t)\bigr)\,,
        \qquad
        a_{1}(t)
        := \frac{1}{t}\log\biggl(
        \frac{\tilde{a}(t)^{1/2}\,B_{\alpha}(t)\,x}
        {(\pi/\sqrt{2}-\ve)^{2}}\biggr)\,.
    \end{equation}
    From \eqref{eq:stable_scale_B_alpha} and
    \eqref{eq:stable_level_expansion}, we obtain
    \begin{equation}\label{eq:level_a_1_expansion}
        \begin{aligned}
            a_{1}(t)
            &= \frac{\alpha^{2}t^{2}}{16}
            + \frac{1}{t}\biggl(\frac{1}{\alpha}
            \log\Bigl(\frac{\alpha t}{4\pi}\Bigr)
            + \log x
            - 2\log\Bigl(1-\frac{\sqrt{2}\,\ve}{\pi}\Bigr)
            \biggr) + O(t^{-5/2})\,,\\
            \frac{8}{3}\,a_{1}(t)^{3/2}
            &= \frac{\alpha^{3}t^{3}}{24}
            + \log\Bigl(\frac{\alpha t}{4\pi}\Bigr)
            + \alpha\log x
            - 2\alpha\log\Bigl(1-\frac{\sqrt{2}\,\ve}{\pi}\Bigr)
            + o(1)\,.
        \end{aligned}
    \end{equation}
    Since
    $a_{1}(t)\exp\bigl(4(\log a_{1}(t))^{2}\bigr)
    =e^{O((\log t)^{2})}$ and
    $\log m(a_{1}(t))\sim\alpha^{3}t^{3}/24$,
    $\ell(t)$ satisfies the window condition
    \eqref{eq:admissible_window} at $a_{1}(t)$, so
    Theorem~\ref{thm:sharp_lower_tail_lowest_eigenvalue}
    applies with $a=a_{1}(t)$ and $\cL_{a}=\ell(t)$. Together
    with \eqref{eq:level_event_identification},
    \eqref{eq:level_a_1_expansion}, and
    $n(t)\,\ell(t)=L(t)=e^{\alpha^{3}t^{3}/24}$, it gives
    \begin{equation}\label{eq:level_event_asymptotics}
        \begin{aligned}
            n(t)\,\dP\bigl(\lambda_{1}<-a_{1}(t)\bigr)
            &= n(t)\,\frac{\ell(t)\sqrt{a_{1}(t)}}{\pi}
            \exp\Bigl(-\frac{8}{3}a_{1}(t)^{3/2}\Bigr)
            (1+o(1))\\
            &= x^{-\alpha}
            \Bigl(1-\frac{\sqrt{2}\,\ve}{\pi}\Bigr)^{2\alpha}
            (1+o(1))\,,
        \end{aligned}
    \end{equation}
    where $\sqrt{a_{1}(t)}/\pi
    =\bigl(\alpha t/(4\pi)\bigr)(1+o(1))$ cancels the term
    $\log\bigl(\alpha t/(4\pi)\bigr)$ in the exponent. Combining
    \eqref{eq:tilde_Y_tail_reduction_to_good_event},
    \eqref{eq:liminf_reduction}, and
    \eqref{eq:level_event_asymptotics} yields
    \begin{equation}\label{eq:tilde_Y_liminf_bound}
        \liminf_{t\to\infty}
        n(t)\,\dP\bigl(\tilde{Y}_{0}(t)>x\bigr)
        \ge x^{-\alpha}
        \biggl(1-\frac{\sqrt{2}\,\ve}{\pi}\biggr)^{2\alpha}\,.
    \end{equation}

    We next bound the right-hand side of
    \eqref{eq:tilde_Y_tail_reduction_to_good_event} from
    above. Since the
    second factor in
    \eqref{eq:upper_bound_spectral_representation_tilde_U} is
    $1+o(1)$, for every $\delta>0$ and all sufficiently large
    $t$,
    \begin{equation}\label{eq:limsup_reduction}
        \dP\bigl(\tilde{Y}_{0}(t)>x,\,G_{1},\,G_{2},\,
        G_{3}\bigr)
        \le \dP\biggl(
        \frac{(\pi/\sqrt{2}+\ve)^{2}\,(1+\delta)\,
        e^{-t\lambda_{1}}}
        {\tilde{a}(t)^{1/2}\,B_{\alpha}(t)}>x\biggr)
        = \dP\bigl(\lambda_{1}<-a_{2}(t)\bigr)\,,
    \end{equation}
    where
    \begin{equation}\label{eq:def_level_a_2}
        a_{2}(t)
        := \frac{1}{t}\log\biggl(
        \frac{\tilde{a}(t)^{1/2}\,B_{\alpha}(t)\,x}
        {(\pi/\sqrt{2}+\ve)^{2}\,(1+\delta)}\biggr)\,.
    \end{equation}
    The expansion \eqref{eq:level_a_1_expansion} and the
    asymptotics \eqref{eq:level_event_asymptotics} apply
    to $a_{2}(t)$ with $-\ve$ replaced by $+\ve$. The factor
    $(1+\delta)$ contributes $(1+\delta)^{\alpha}$, and we
    obtain
    \begin{equation}\label{eq:limsup_level_asymptotics}
        n(t)\,\dP\bigl(\lambda_{1}<-a_{2}(t)\bigr)
        = x^{-\alpha}
        \Bigl(1+\frac{\sqrt{2}\,\ve}{\pi}\Bigr)^{2\alpha}
        (1+\delta)^{\alpha}\,(1+o(1))\,.
    \end{equation}
    Combining
    \eqref{eq:tilde_Y_tail_reduction_to_good_event},
    \eqref{eq:limsup_reduction}, and
    \eqref{eq:limsup_level_asymptotics}, and letting
    $\delta\downarrow0$, we obtain
    \begin{equation}\label{eq:tilde_Y_limsup_bound}
        \limsup_{t\to\infty}
        n(t)\,\dP\bigl(\tilde{Y}_{0}(t)>x\bigr)
        \le x^{-\alpha}
        \biggl(1+\frac{\sqrt{2}\,\ve}{\pi}\biggr)^{2\alpha}\,.
    \end{equation}
    Since $\ve\in(0,\pi/\sqrt{2})$ was arbitrary,
    \eqref{eq:tilde_Y_liminf_bound} and
    \eqref{eq:tilde_Y_limsup_bound} prove
    \eqref{eq:tilde_Y_positive_tail}.
\end{proof}

It remains to identify $\sigma^{2}$ in
\eqref{eq:stable_criterion_variance} and $\nu$ in
\eqref{eq:stable_criterion_centering}. These follow
from the truncated moments of $\tilde{Y}_{0}(t)$ in the
following lemma, which is the analogue of
\cite[Lemma~8]{BAMR19}.

\begin{lemma}\label{lem:truncated_moments}
    For every $y>0$, the following hold as $t\to\infty$.
    \begin{enumerate}
        \item[\textup{(i)}] The truncated second moment
        satisfies
        \begin{equation}\label{eq:truncated_second_moment}
            n(t)\,\dE\bigl[\tilde{Y}_{0}(t)^{2}\,;\,
            \tilde{Y}_{0}(t)\le y\bigr]
            \longrightarrow
            \frac{\alpha}{2-\alpha}\,y^{2-\alpha}\,.
        \end{equation}
        \item[\textup{(ii)}] The first moments
        satisfy
        \begin{equation}\label{eq:truncated_first_moment}
            \begin{cases}
                \;n(t)\,\dE\bigl[\tilde{Y}_{0}(t)\,;\,
                \tilde{Y}_{0}(t)\le y\bigr]
                \longrightarrow
                \dfrac{\alpha}{1-\alpha}\,y^{1-\alpha}\,,
                & \alpha\in(0,1)\,,\\[10pt]
                \;n(t)\,\Bigl(\dE\bigl[\tilde{Y}_{0}(t)\,;\,
                \tilde{Y}_{0}(t)\le y\bigr]
                -\dE\bigl[\tilde{Y}_{0}(t)\,;\,
                \tilde{Y}_{0}(t)\le1\bigr]\Bigr)
                \longrightarrow \log y\,,
                & \alpha=1\,,\\[10pt]
                \;n(t)\,\dE\bigl[\tilde{Y}_{0}(t)\,;\,
                \tilde{Y}_{0}(t)>y\bigr]
                \longrightarrow
                \dfrac{\alpha}{\alpha-1}\,y^{1-\alpha}\,,
                & \alpha\in(1,2)\,.
            \end{cases}
        \end{equation}
    \end{enumerate}
\end{lemma}
We prove Lemma~\ref{lem:truncated_moments} by
integrating the tail of $\tilde{Y}_{0}(t)$.
To pass to the limit under the integral, we supplement
the pointwise tail asymptotics of
Proposition~\ref{prop:tilde_Y_positive_tail}
with a uniform tail bound and estimates on the
contributions from very small and very large values
of $\tilde{Y}_{0}(t)$.
The following lemma provides these estimates,
using the nonasymptotic eigenvalue bound in
Lemma~\ref{lem:left_tail_upper_bound_fixed_L}.

\begin{lemma}\label{lem:uniform_tail_bound}
    The following estimates hold.
    \begin{enumerate}
        \item[\textup{(i)}] There exist $C,t_{0}>0$, depending
        only on $\alpha$, such that for all $t\ge t_{0}$ and
        all $x\in[e^{-t^{3/2}},\,e^{t^{3/2}}]$,
        \begin{equation*}
            n(t)\,\dP\bigl(\tilde{Y}_{0}(t)>x\bigr)
            \le C\,x^{-\alpha}\,.
        \end{equation*}
        \item[\textup{(ii)}] As $t\to\infty$,
        \begin{equation*}
            n(t)\,\dE\bigl[\tilde{Y}_{0}(t)^{2}\,;\,
            \tilde{Y}_{0}(t)\le e^{-t^{3/2}}\bigr]
            \longrightarrow 0\,.
        \end{equation*}
        \item[\textup{(iii)}] If $\alpha\in(0,1)$, as
        $t\to\infty$,
        \begin{equation*}
            n(t)\,\dE\bigl[\tilde{Y}_{0}(t)\,;\,
            \tilde{Y}_{0}(t)\le e^{-t^{3/2}}\bigr]
            \longrightarrow 0\,.
        \end{equation*}
        \item[\textup{(iv)}] If $\alpha\in(1,2)$, as
        $t\to\infty$,
        \begin{equation*}
            n(t)\,\dE\bigl[\tilde{Y}_{0}(t)\,;\,
            \tilde{Y}_{0}(t)\ge e^{t^{3/2}}\bigr]
            \longrightarrow 0\,.
        \end{equation*}
    \end{enumerate}
\end{lemma}

\begin{proof}
    Throughout, $(\lambda_{k},\varphi_{k})_{k\ge1}$ denote the
    eigenpairs of $\cH_{Q_{\ell(t)}}$, and $C>0$ depends only
    on $\alpha$ and may change from line to line.

    We begin with \textup{(i)}. Fix
    $\theta\in(0,\,1-2^{-2/3})$ and $\ve:=\pi/(2\sqrt{2})$, and
    define $\tilde{a}(t)$ and the events $G_{1}$, $G_{2}$,
    $G_{3}$ as in the proof of
    Proposition~\ref{prop:tilde_Y_positive_tail}. These do not
    involve $x$, and the arguments for
    \eqref{eq:G_2_complement_negligible} and
    \eqref{eq:G_3_complement_negligible} give, for all
    sufficiently large $t$,
    \begin{equation}\label{eq:uniform_bad_events}
        n(t)\,\dP\bigl(G_{2}^{c}\bigr)
        + n(t)\,\dP\bigl(G_{1},\,G_{2},\,G_{3}^{c}\bigr)
        \le e^{-ct^{3}}\,,
    \end{equation}
    with $c>0$ depending only on $\alpha$. For every
    $x>0$, \eqref{eq:limsup_reduction} with $\delta=1$ gives
    \begin{equation}\label{eq:uniform_good_inclusion}
        \dP\bigl(\tilde{Y}_{0}(t)>x,\,G_{1},\,G_{2},\,
        G_{3}\bigr)
        \le \dP\bigl(\lambda_{1}<-a_{2}(t)\bigr)\,,
    \end{equation}
    with $a_{2}(t)$ as in \eqref{eq:def_level_a_2}.

    Let $x\in[e^{-t^{3/2}},\,e^{t^{3/2}}]$. By
    \eqref{eq:G_1_complement_bound},
    $\tilde{Y}_{0}(t)\le e^{-t^{3/2}}\le x$ on
    $G_{1}^{c}$, so $\{\tilde{Y}_{0}(t)>x\}\cap G_{1}^{c}=\emptyset$.
    Since $x\le e^{t^{3/2}}$, for all sufficiently large $t$ we also
    have $e^{-ct^{3}}\le x^{-\alpha}$. Since
    $|\log x|\le t^{3/2}$, expanding as in
    \eqref{eq:level_a_1_expansion} gives, with an error uniform
    in $x$,
    \begin{equation}\label{eq:uniform_level_expansion}
        \frac{8}{3}\,a_{2}(t)^{3/2}
        = \frac{\alpha^{3}t^{3}}{24}
        + \log\Bigl(\frac{\alpha t}{4\pi}\Bigr)
        + \alpha\log x + O(1)\,.
    \end{equation}
    Indeed, the second-order term of the expansion is
    of order $(\log x)^{2}/t^{3}$, which stays bounded
    when $|\log x|\le t^{3/2}$. This is the reason for the
    range in \textup{(i)}.
    Hence
    Lemma~\ref{lem:left_tail_upper_bound_fixed_L},
    \eqref{eq:first_explosion_mean_asymptotics},
    \eqref{eq:uniform_level_expansion},
    $n(t)\,\ell(t)=L(t)=e^{\alpha^{3}t^{3}/24}$, and
    $\sqrt{a_{2}(t)}\le\alpha t$ give
    \begin{equation}\label{eq:uniform_good_bound}
        n(t)\,\dP\bigl(\lambda_{1}<-a_{2}(t)\bigr)
        \le C\,e^{\alpha^{3}t^{3}/24}\,\sqrt{a_{2}(t)}\,
        e^{-\frac{8}{3}a_{2}(t)^{3/2}}
        \le C\,x^{-\alpha}\,.
    \end{equation}
    Combining \eqref{eq:uniform_bad_events},
    \eqref{eq:uniform_good_inclusion}, and
    \eqref{eq:uniform_good_bound} with the decomposition
    \eqref{eq:tilde_Y_tail_decomposition} proves \textup{(i)}.

    We next prove \textup{(ii)}. We split the
    expectation at $e^{-\alpha^{2}t^{3}/24}$, writing
    \begin{equation}\label{eq:deep_zone_split}
        \begin{aligned}
            n(t)\,\dE\bigl[\tilde{Y}_{0}(t)^{2}\,;\,
            \tilde{Y}_{0}(t)\le e^{-t^{3/2}}\bigr]
            &= n(t)\,\dE\bigl[\tilde{Y}_{0}(t)^{2}\,;\,
            \tilde{Y}_{0}(t)\le e^{-\alpha^{2}t^{3}/24}\bigr]\\
            &\qquad
            + n(t)\,\dE\bigl[\tilde{Y}_{0}(t)^{2}\,;\,
            e^{-\alpha^{2}t^{3}/24}<\tilde{Y}_{0}(t)
            \le e^{-t^{3/2}}\bigr]\,.
        \end{aligned}
    \end{equation}
    Since $n(t)\le e^{\alpha^{3}t^{3}/24}$ and
    $\alpha<2$, the first term is at most
    $n(t)\,e^{-\alpha^{2}t^{3}/12}
    \le e^{\alpha^{2}(\alpha-2)t^{3}/24}\longrightarrow0$.

    To bound the second term, we control the tail of
    $\tilde{Y}_{0}(t)$ below the range of \textup{(i)}. For
    $x>0$, set
    \begin{equation}\label{eq:def_crude_level}
        \tilde{a}_{x}(t)
        := \frac{\log B_{\alpha}(t)-\log\ell(t)+\log x}{t}\,,
    \end{equation}
    so that, bounding the spectral representation by Parseval's
    identity as in \eqref{eq:G_1_complement_bound},
    \begin{equation}\label{eq:crude_inclusion}
        \bigl\{\tilde{Y}_{0}(t)>x\bigr\}
        \subset \bigl\{\lambda_{1}<-\tilde{a}_{x}(t)\bigr\}\,.
    \end{equation}
    Since $\log\ell(t)=t+o(1)$ by \eqref{eq:block_number}, the
    definitions \eqref{eq:stable_scale_B_alpha} and
    \eqref{eq:def_crude_level} give
    \begin{equation}\label{eq:crude_level_value}
        \tilde{a}_{x}(t)
        = \frac{\alpha^{2}t^{2}}{16}
        + \frac{\log x-t}{t}
        + O\Bigl(\frac{\log t}{t}\Bigr)\,,
    \end{equation}
    where the error term does not depend on $x$. In
    particular, $\tilde{a}_{x}(t)\to\infty$ uniformly on
    $[e^{-\alpha^{2}t^{3}/24},e^{t^{3}}]$. Hence
    Lemma~\ref{lem:left_tail_upper_bound_fixed_L},
    \eqref{eq:first_explosion_mean_asymptotics}, and
    \eqref{eq:crude_inclusion} give
    \begin{equation*}
        n(t)\,\dP\bigl(\tilde{Y}_{0}(t)>x\bigr)
        \le C\,n(t)\,\ell(t)\,\sqrt{\tilde{a}_{x}(t)}\,
        e^{-\frac{8}{3}\tilde{a}_{x}(t)^{3/2}}\,,
        \qquad x\in[e^{-\alpha^{2}t^{3}/24},\,e^{t^{3}}]\,.
    \end{equation*}
    By \eqref{eq:crude_level_value},
    $\tilde{a}_{x}(t)^{1/2}\le Ct$ on this range. For the
    exponential factor, $\log x$ may now be of order $t^{3}$,
    so the expansion \eqref{eq:level_a_1_expansion} is no
    longer available. Instead we use the tangent bound at
    $r=\alpha^{2}t^{2}/16$ for the convex function
    $r\mapsto\frac{8}{3}r^{3/2}$,
    \begin{equation}\label{eq:cost_convexity}
        \frac{8}{3}\,r^{3/2}
        \ge \frac{\alpha^{3}t^{3}}{24}
        + \alpha t\,\Bigl(r-\frac{\alpha^{2}t^{2}}{16}\Bigr)\,,
        \qquad r\ge0\,.
    \end{equation}
    Combining \eqref{eq:cost_convexity} at
    $r=\tilde{a}_{x}(t)$ with \eqref{eq:crude_level_value} and
    $n(t)\,\ell(t)=L(t)=e^{\alpha^{3}t^{3}/24}$, we obtain
    \begin{equation}\label{eq:crude_zone_bound}
        n(t)\,\dP\bigl(\tilde{Y}_{0}(t)>x\bigr)
        \le C\,t\,e^{(1+\alpha)t}\,x^{-\alpha}\,,
        \qquad x\in[e^{-\alpha^{2}t^{3}/24},\,e^{t^{3}}]\,.
    \end{equation}
    By Fubini's theorem and the monotonicity of
    $x\mapsto\dP(\tilde{Y}_{0}(t)>x)$, the second term
    of \eqref{eq:deep_zone_split} is at most
    \begin{equation}\label{eq:crude_zone_second_moment}
        \int_{0}^{e^{-t^{3/2}}}2x\cdot
        C\,t\,e^{(1+\alpha)t}\,x^{-\alpha}\,dx
        = \frac{2C}{2-\alpha}\,t\,e^{(1+\alpha)t}\,
        e^{-(2-\alpha)t^{3/2}}
        \longrightarrow 0\,,
    \end{equation}
    where the integral converges at $0$ since
    $\alpha<2$.
    This proves \textup{(ii)}.

    We turn to \textup{(iii)}. Let $\alpha\in(0,1)$. As
    in \eqref{eq:deep_zone_split}, we split the expectation at
    $e^{-\alpha^{2}t^{3}/24}$. Since $\alpha<1$, the
    first term is at most
    $n(t)\,e^{-\alpha^{2}t^{3}/24}
    \le e^{\alpha^{2}(\alpha-1)t^{3}/24}\longrightarrow0$.
    As in \eqref{eq:crude_zone_second_moment}, the
    second term is at most
    \begin{equation*}
        \int_{0}^{e^{-t^{3/2}}}
        C\,t\,e^{(1+\alpha)t}\,x^{-\alpha}\,dx
        = \frac{C}{1-\alpha}\,t\,e^{(1+\alpha)t}\,
        e^{-(1-\alpha)t^{3/2}}
        \longrightarrow 0\,,
    \end{equation*}
    where the integral converges at $0$ since
    $\alpha<1$.
    This proves \textup{(iii)}.

    Finally, we prove \textup{(iv)}. Let
    $\alpha\in(1,2)$. Since \eqref{eq:crude_zone_bound} holds
    only for $x\le e^{t^{3}}$, we split the expectation at
    $e^{t^{3}}$, writing
    \begin{equation*}
        n(t)\,\dE\bigl[\tilde{Y}_{0}(t)\,;\,
        \tilde{Y}_{0}(t)\ge e^{t^{3/2}}\bigr]
        = n(t)\,\dE\bigl[\tilde{Y}_{0}(t)\,;\,
        e^{t^{3/2}}\le\tilde{Y}_{0}(t)\le e^{t^{3}}\bigr]
        + n(t)\,\dE\bigl[\tilde{Y}_{0}(t)\,;\,
        \tilde{Y}_{0}(t)>e^{t^{3}}\bigr]\,.
    \end{equation*}
    By Fubini's theorem and \eqref{eq:crude_zone_bound},
    the first term is at most
    \begin{equation*}
        e^{t^{3/2}}\,n(t)\,
        \dP\bigl(\tilde{Y}_{0}(t)\ge e^{t^{3/2}}\bigr)
        + n(t)\int_{e^{t^{3/2}}}^{e^{t^{3}}}
        \dP\bigl(\tilde{Y}_{0}(t)>x\bigr)\,dx
        \le \frac{C\alpha}{\alpha-1}\,
        t\,e^{(1+\alpha)t}\,e^{-(\alpha-1)t^{3/2}}
        \longrightarrow 0\,,
    \end{equation*}
    where the integral converges at infinity since
    $\alpha>1$.
    For the second term,
    the moment bound \eqref{eq:moment_upper_bound} with $p=2$
    and the definition \eqref{eq:stable_scale_B_alpha} of
    $B_{\alpha}(t)$ give
    $n(t)\,\dE\bigl[\tilde{Y}_{0}(t)^{2}\bigr]
    \le e^{(\alpha+1)(\alpha-2)^{2}t^{3}/24+O(t)}$. Hence
    \begin{equation*}
        n(t)\,\dE\bigl[\tilde{Y}_{0}(t)\,;\,
        \tilde{Y}_{0}(t)>e^{t^{3}}\bigr]
        \le e^{-t^{3}}\,n(t)\,
        \dE\bigl[\tilde{Y}_{0}(t)^{2}\bigr]
        \longrightarrow 0\,,
    \end{equation*}
    since $(\alpha+1)(\alpha-2)^{2}<24$.
    This proves
    \textup{(iv)} and completes the proof.
\end{proof}

We now provide the proof of Lemma~\ref{lem:truncated_moments}. 

\begin{proof}[Proof of Lemma~\ref{lem:truncated_moments}]
    We begin with \textup{(i)}. Fix $y>0$. For every $t$
    with $e^{-t^{3/2}}<y$, Fubini's theorem gives
    \begin{equation}\label{eq:truncated_second_moment_fubini}
        \begin{aligned}
            n(t)\,\dE\bigl[\tilde{Y}_{0}(t)^{2}\,;\,
            e^{-t^{3/2}}<\tilde{Y}_{0}(t)\le y\bigr]
            &= \int_{e^{-t^{3/2}}}^{y}2x\,n(t)\,
            \dP\bigl(\tilde{Y}_{0}(t)>x\bigr)\,dx\\
            &\qquad
            + e^{-2t^{3/2}}\,n(t)\,
            \dP\bigl(\tilde{Y}_{0}(t)>e^{-t^{3/2}}\bigr)
            - y^{2}\,n(t)\,
            \dP\bigl(\tilde{Y}_{0}(t)>y\bigr)\,.
        \end{aligned}
    \end{equation}
    We evaluate the first term on the right-hand side by
    dominated convergence. Set
    \begin{equation*}
        g_{t}(x):=2x\,n(t)\,
        \dP\bigl(\tilde{Y}_{0}(t)>x\bigr)\,
        \1_{\{x>e^{-t^{3/2}}\}}\,,
        \qquad x\in(0,y]\,.
    \end{equation*}
    For every $x\in(0,y]$, the indicator equals $1$ for
    all sufficiently large $t$, so
    Proposition~\ref{prop:tilde_Y_positive_tail} gives
    $g_{t}(x)\to2x^{1-\alpha}$ as $t\to\infty$.
    By Lemma~\ref{lem:uniform_tail_bound}\,\textup{(i)},
    $g_{t}(x)\le2C\,x^{1-\alpha}$ on $(0,y]$ for all
    sufficiently large $t$.
    Since $\alpha<2$, the
    dominating function is integrable on $(0,y]$. Hence
    \begin{equation*}
        \int_{e^{-t^{3/2}}}^{y}2x\,n(t)\,
        \dP\bigl(\tilde{Y}_{0}(t)>x\bigr)\,dx
        = \int_{0}^{y}g_{t}(x)\,dx
        \longrightarrow \int_{0}^{y}2x^{1-\alpha}\,dx
        = \frac{2}{2-\alpha}\,y^{2-\alpha}\,.
    \end{equation*}
    By Lemma~\ref{lem:uniform_tail_bound}\,\textup{(i)},
    the second term on the right-hand side of
    \eqref{eq:truncated_second_moment_fubini} is at most
    $C\,e^{-(2-\alpha)t^{3/2}}\to0$. It follows from
    Proposition~\ref{prop:tilde_Y_positive_tail} that the third
    term converges to $-y^{2-\alpha}$. Combining these limits
    with Lemma~\ref{lem:uniform_tail_bound}\,\textup{(ii)}, we
    obtain
    \begin{equation*}
        n(t)\,\dE\bigl[\tilde{Y}_{0}(t)^{2}\,;\,
        \tilde{Y}_{0}(t)\le y\bigr]
        \longrightarrow
        \frac{2}{2-\alpha}\,y^{2-\alpha}-y^{2-\alpha}
        = \frac{\alpha}{2-\alpha}\,y^{2-\alpha}\,,
    \end{equation*}
    which is \eqref{eq:truncated_second_moment}.

    We next prove \textup{(ii)}, beginning with
    $\alpha\in(0,1)$. The argument of \textup{(i)} applies to
    the decomposition
    \begin{equation*}
        \begin{aligned}
            n(t)\,\dE\bigl[\tilde{Y}_{0}(t)\,;\,
            e^{-t^{3/2}}<\tilde{Y}_{0}(t)\le y\bigr]
            &= \int_{e^{-t^{3/2}}}^{y}n(t)\,
            \dP\bigl(\tilde{Y}_{0}(t)>x\bigr)\,dx\\
            &\qquad
            + e^{-t^{3/2}}\,n(t)\,
            \dP\bigl(\tilde{Y}_{0}(t)>e^{-t^{3/2}}\bigr)
            - y\,n(t)\,\dP\bigl(\tilde{Y}_{0}(t)>y\bigr)
        \end{aligned}
    \end{equation*}
    with two changes. The dominating function is
    $C\,x^{-\alpha}$, integrable on $(0,y]$ since
    $\alpha<1$, and
    Lemma~\ref{lem:uniform_tail_bound}\,\textup{(iii)} replaces
    Lemma~\ref{lem:uniform_tail_bound}\,\textup{(ii)}. Hence
    \begin{equation*}
        n(t)\,\dE\bigl[\tilde{Y}_{0}(t)\,;\,
        \tilde{Y}_{0}(t)\le y\bigr]
        \longrightarrow
        \int_{0}^{y}x^{-\alpha}\,dx-y^{1-\alpha}
        = \frac{\alpha}{1-\alpha}\,y^{1-\alpha}\,,
    \end{equation*}
    which proves \eqref{eq:truncated_first_moment} for
    $\alpha\in(0,1)$.

    Let $\alpha=1$. Since the case $y>1$ is analogous
    and the case $y=1$ is trivial, we assume $y<1$. The difference in
    \eqref{eq:truncated_first_moment} equals
    $-n(t)\,\dE\bigl[\tilde{Y}_{0}(t)\,;\,
    y<\tilde{Y}_{0}(t)\le1\bigr]$, and Fubini's theorem gives
    \begin{equation*}
        \begin{aligned}
            n(t)\,\dE\bigl[\tilde{Y}_{0}(t)\,;\,
            y<\tilde{Y}_{0}(t)\le1\bigr]
            &= \int_{y}^{1}n(t)\,
            \dP\bigl(\tilde{Y}_{0}(t)>x\bigr)\,dx\\
            &\qquad
            + y\,n(t)\,\dP\bigl(\tilde{Y}_{0}(t)>y\bigr)
            - n(t)\,\dP\bigl(\tilde{Y}_{0}(t)>1\bigr)\,.
        \end{aligned}
    \end{equation*}
    On $[y,1]$, the integrand converges pointwise to $x^{-1}$
    by Proposition~\ref{prop:tilde_Y_positive_tail} and is
    bounded by $C\,y^{-1}$ by
    Lemma~\ref{lem:uniform_tail_bound}\,\textup{(i)} for
    all sufficiently large $t$, so dominated convergence gives
    \begin{equation*}
        \int_{y}^{1}n(t)\,
        \dP\bigl(\tilde{Y}_{0}(t)>x\bigr)\,dx
        \longrightarrow \int_{y}^{1}x^{-1}\,dx
        = -\log y\,.
    \end{equation*}
    By Proposition~\ref{prop:tilde_Y_positive_tail}, the
    remaining two terms converge to $1$ and $-1$, and cancel in
    the limit. Hence the difference converges to $\log y$,
    which proves \eqref{eq:truncated_first_moment} for
    $\alpha=1$.

    It remains to prove
    \eqref{eq:truncated_first_moment} for $\alpha\in(1,2)$. For
    every $t$ with $e^{t^{3/2}}>y$, Fubini's theorem gives
    \begin{equation*}
        \begin{aligned}
            n(t)\,\dE\bigl[\tilde{Y}_{0}(t)\,;\,
            y<\tilde{Y}_{0}(t)\le e^{t^{3/2}}\bigr]
            &= \int_{y}^{e^{t^{3/2}}}n(t)\,
            \dP\bigl(\tilde{Y}_{0}(t)>x\bigr)\,dx
            + y\,n(t)\,\dP\bigl(\tilde{Y}_{0}(t)>y\bigr)\\
            &\qquad
            - e^{t^{3/2}}\,n(t)\,
            \dP\bigl(\tilde{Y}_{0}(t)>e^{t^{3/2}}\bigr)\,.
        \end{aligned}
    \end{equation*}
    The argument of \textup{(i)} applies with two
    changes. The dominating function is $C\,x^{-\alpha}$,
    integrable on $[y,\infty)$ since $\alpha>1$, and
    Lemma~\ref{lem:uniform_tail_bound}\,\textup{(iv)} replaces
    Lemma~\ref{lem:uniform_tail_bound}\,\textup{(ii)}. Hence
    \begin{equation*}
        n(t)\,\dE\bigl[\tilde{Y}_{0}(t)\,;\,
        \tilde{Y}_{0}(t)>y\bigr]
        \longrightarrow
        \int_{y}^{\infty}x^{-\alpha}\,dx+y^{1-\alpha}
        = \frac{\alpha}{\alpha-1}\,y^{1-\alpha}\,,
    \end{equation*}
    which proves \eqref{eq:truncated_first_moment} for
    $\alpha\in(1,2)$ and completes the proof.
\end{proof}

With Proposition~\ref{prop:tilde_Y_positive_tail} and
Lemma~\ref{lem:truncated_moments}, we can now verify all
hypotheses of
Proposition~\ref{prop:stable_convergence_criterion}.
\begin{proof}[Proof of Theorem~\ref{thm:PAM_Stable_Dirichlet}]
    We apply
    Proposition~\ref{prop:stable_convergence_criterion} with
    the spectral function $\fL$ of
    \eqref{eq:tilde_Y_alpha_spectral_function}.
    By Proposition~\ref{prop:tilde_Y_positive_tail},
    for every $x>0$,
    \begin{equation}\label{eq:tail_condition_verified}
        \lim_{t\to\infty}n(t)\,
        \dP\bigl(\tilde{Y}_{0}(t)>x\bigr)
        = x^{-\alpha} = -\fL(x)\,.
    \end{equation}
    This proves
    \eqref{eq:stable_criterion_spectral_function}.
    Since $n(t)\to\infty$,
    \eqref{eq:tail_condition_verified} also gives
    $\dP(\tilde{Y}_{0}(t)>\ve)\to0$ for every $\ve>0$, which is
    \eqref{eq:UAN_condition}.

    We turn to $\sigma^{2}$. Since
    $0\le\Var(Z_{y}(t))\le\dE[Z_{y}(t)^{2}]
    =\dE[\tilde{Y}_{0}(t)^{2}\,;\,\tilde{Y}_{0}(t)\le y]$,
    Lemma~\ref{lem:truncated_moments}\,\textup{(i)} gives
    \begin{equation*}
        0\le
        \lim_{y\to0}\liminf_{t\to\infty}
        n(t)\,\Var\bigl(Z_{y}(t)\bigr)
        \le
        \lim_{y\to0}\limsup_{t\to\infty}
        n(t)\,\Var\bigl(Z_{y}(t)\bigr)
        \le \lim_{y\to0}
        \frac{\alpha}{2-\alpha}\,y^{2-\alpha}
        = 0\,,
    \end{equation*}
    so \eqref{eq:stable_criterion_variance} holds with
    $\sigma^{2}=0$.

    Following \cite[Proposition~6.4]{BABM05} (see also
    \cite[Section~6.6]{BAMR19}), we verify
    \eqref{eq:stable_criterion_centering} with
    \begin{equation}\label{eq:tilde_Y_stable_drift}
        \nu = \begin{cases}
            \dfrac{\alpha\pi}{2\cos(\frac{\alpha\pi}{2})}\,,
            & \alpha\in(0,1)\cup(1,2)\,,\\[6pt]
            0\,, & \alpha=1\,.
        \end{cases}
    \end{equation}
    Let $y>0$ be arbitrary. Since
    $\tilde{Y}_{0}(t)\ge0$, we have $\dE[Z_{y}(t)]
    =\dE[\tilde{Y}_{0}(t)\,;\,\tilde{Y}_{0}(t)\le y]$. On
    $(0,\infty)$, $d\fL(x)=\alpha x^{-\alpha-1}\,dx$.
    For the case $\alpha\in(0,1)$, we have
    $\tilde{A}(t)=0$, and
    Lemma~\ref{lem:truncated_moments}\,\textup{(ii)} turns
    \eqref{eq:stable_criterion_centering} into
    \begin{equation}\label{eq:nu_computation_small_alpha}
        \nu = \frac{\alpha}{1-\alpha}\,y^{1-\alpha}
        + \alpha\int_{y}^{\infty}
        \frac{x^{-\alpha}}{1+x^{2}}\,dx
        - \alpha\int_{0}^{y}
        \frac{x^{2-\alpha}}{1+x^{2}}\,dx\,.
    \end{equation}
    Substituting
    $\frac{x^{2-\alpha}}{1+x^{2}}
    =x^{-\alpha}-\frac{x^{-\alpha}}{1+x^{2}}$ in the second
    integral and using
    $\alpha\int_{0}^{y}x^{-\alpha}\,dx
    =\frac{\alpha}{1-\alpha}\,y^{1-\alpha}$, we reduce the
    right-hand side of \eqref{eq:nu_computation_small_alpha}
    to
    \begin{equation*}
        \alpha\int_{0}^{\infty}
        \frac{x^{-\alpha}}{1+x^{2}}\,dx
        = \frac{\alpha\pi}{2\cos(\alpha\pi/2)}\,.
    \end{equation*}
    In particular, the right-hand side does not depend
    on $y$.
    For the case $\alpha\in(1,2)$, we have
    $\tilde{A}(t)=n(t)\,\dE\bigl[\tilde{Y}_{0}(t)\bigr]$, and
    Lemma~\ref{lem:truncated_moments}\,\textup{(ii)} gives
    \begin{equation*}
        n(t)\,\dE\bigl[Z_{y}(t)\bigr]-\tilde{A}(t)
        = -n(t)\,\dE\bigl[\tilde{Y}_{0}(t)\,;\,
        \tilde{Y}_{0}(t)>y\bigr]
        \longrightarrow
        -\frac{\alpha}{\alpha-1}\,y^{1-\alpha}
        = -\alpha\int_{y}^{\infty}x^{-\alpha}\,dx\,.
    \end{equation*}
        The same substitution, applied to the first
    integral in \eqref{eq:stable_criterion_centering}, again
    yields
    \begin{equation*}
        \nu=-\alpha\int_{0}^{\infty}
        \frac{x^{2-\alpha}}{1+x^{2}}\,dx
        =\frac{\alpha\pi}{2\cos(\alpha\pi/2)}\,.
    \end{equation*}
    For the case $\alpha=1$, we have
    $\tilde{A}(t)=n(t)\,\dE\bigl[Z_{1}(t)\bigr]$, so
    Lemma~\ref{lem:truncated_moments}\,\textup{(ii)} gives
    $n(t)\,\dE[Z_{y}(t)]-\tilde{A}(t)\to\log y$. The integrals
    in \eqref{eq:stable_criterion_centering} equal
    \begin{equation*}
        \int_{|x|>y}\frac{x}{1+x^{2}}\,d\fL(x)
        - \int_{0<|x|\le y}\frac{x^{3}}{1+x^{2}}\,d\fL(x)
        = \int_{y}^{\infty}\frac{dx}{x(1+x^{2})}
        - \int_{0}^{y}\frac{x\,dx}{1+x^{2}}
        = -\log y\,,
    \end{equation*}
    where the last equality follows from the antiderivatives
    $\log\frac{x}{\sqrt{1+x^{2}}}$ and
    $\frac{1}{2}\log(1+x^{2})$. Hence $\nu=0$.

    Therefore
    Proposition~\ref{prop:stable_convergence_criterion} gives
    \begin{equation*}
        \sum_{i\in\cI(t)}\tilde{Y}_{i}(t)-\tilde{A}(t)
        \xrightarrow{\ d\ } \dX_{\nu,0,\fL}\,,
    \end{equation*}
    where, by \eqref{eq:LK_representation}, the characteristic
    function of $\dX_{\nu,0,\fL}$ is
    \begin{equation}
    \label{eq:tilde_Y_infinitely_divisible_char_fn}
        \phi(z) = \exp\biggl(i\nu z
        + \alpha\int_{0}^{\infty}
        \Bigl(e^{izx}-1-\frac{izx}{1+x^{2}}\Bigr)
        \frac{dx}{x^{\alpha+1}}\biggr)\,,
    \end{equation}
    with $\nu$ as in \eqref{eq:tilde_Y_stable_drift}.
    The function
    \eqref{eq:tilde_Y_infinitely_divisible_char_fn} is the
    characteristic function (6.6)--(6.7) of
    \cite{BABM05}. By \cite[Theorem~6.2]{BABM05}, it
    corresponds to the stable law with exponent $\alpha$ 
    and skewness parameter $1$, whose canonical form is 
    \eqref{eq:stable_char_fn}. Hence
    $\dX_{\nu,0,\fL}$ has law $\dS_{\alpha}$, and
    \eqref{eq:PAM_Stable_Dirichlet} follows.
\end{proof}

We now compare the centered and normalized sums appearing in
Theorems~\ref{thm:PAM_Stable}
and~\ref{thm:PAM_Stable_Dirichlet}. As in
\eqref{eq:def_tilde_Y_i}, set
\begin{equation*}
    Y_{i}(t) := \frac{U_{i}(t)}{B_{\alpha}(t)}\,,
    \qquad
    Y_{0}(t) := \frac{U_{0}(t)}{B_{\alpha}(t)}\,.
\end{equation*}
Since the blocks $Q_{i}$, $i\in\cI(t)$, partition
$Q_{L(t)}$, we have
$\sum_{i\in\cI(t)}Y_{i}(t)=U(t)/B_{\alpha}(t)$. Each
$Y_{i}(t)$ has the same law as $Y_{0}(t)$.

\begin{lemma}\label{lem:Stable_difference}
    Let $L(t)$ be as in \eqref{eq:box_size}. As $t\to\infty$, the
    following hold.
    \begin{enumerate}
        \item[\textup{(i)}] If $\alpha\in(0,1)$, then
        \begin{equation*}
            \sum_{i\in\cI(t)}
            \bigl(Y_{i}(t)-\tilde{Y}_{i}(t)\bigr)
            \xrightarrow{\ \dP\ } 0\,.
        \end{equation*}
        \item[\textup{(ii)}] If $\alpha=1$, then
        \begin{equation}\label{eq:stable_approximation_error_alpha_1}
            \sum_{i\in\cI(t)}
            \bigl(Y_{i}(t)-\tilde{Y}_{i}(t)\bigr)
            - n(t)\Bigl(
            \dE\bigl[Y_{0}(t)\,;\,Y_{0}(t)\le1\bigr]
            -\dE\bigl[\tilde{Y}_{0}(t)\,;\,
            \tilde{Y}_{0}(t)\le1\bigr]\Bigr)
            \xrightarrow{\ \dP\ } 0\,.
        \end{equation}
        \item[\textup{(iii)}] If $\alpha\in(1,2)$, then
        \begin{equation*}
            \sum_{i\in\cI(t)}
            \bigl(Y_{i}(t)-\tilde{Y}_{i}(t)\bigr)
            - n(t)\Bigl(\dE\bigl[Y_{0}(t)\bigr]
            -\dE\bigl[\tilde{Y}_{0}(t)\bigr]\Bigr)
            \xrightarrow{\ \dP\ } 0\,.
        \end{equation*}
    \end{enumerate}
\end{lemma}

\begin{proof}
    Throughout the proof, we use the following
    estimate, valid for every $\alpha\in(0,2)$. By
    \eqref{eq:moment_upper_bound} with $p=\alpha$ and
    $\cL=2e^{t/100}$,
    \begin{equation}\label{eq:window_contribution_negligible}
        \frac{n(t)}{B_{\alpha}(t)^{\alpha}}
        \dE\biggl[\Bigl(\int_{\fS_{1}}
        \tilde{u}_{\fS_{1}}(t,x)\,dx\Bigr)^{\alpha}\biggr]
        \le C\,t^{\alpha+3/2}\,
        \frac{e^{(\alpha+1)t/100}}{\ell(t)}\,
        L(t)\exp\Bigl(-\frac{\alpha^{3}}{24}\,t^{3}\Bigr)
        \longrightarrow 0\,,
    \end{equation}
    since $L(t)=e^{\alpha^{3}t^{3}/24}$ and
    $\ell(t)=e^{t(1+o(1))}$.

    We begin with \textup{(i)}. Since
    $U_{i}(t)\ge\tilde{U}_{i}(t)$ for each $i\in\cI(t)$,
    it suffices to prove
    \begin{equation*}
        \lim_{t\to\infty}\frac{1}{B_{\alpha}(t)^{\alpha}}
        \dE\Bigl[\Bigl(\sum_{i\in\cI(t)}
        \bigl(U_{i}(t)-\tilde{U}_{i}(t)\bigr)\Bigr)^{\alpha}
        \Bigr]
        = 0\,.
    \end{equation*}
    By \eqref{eq:interior_boundary_split}, the
    nonnegativity of the Dirichlet solutions, and the
    inclusions \eqref{eq:layer_window_inclusion}, we have
    \begin{equation}\label{eq:boundary_layer_window_split}
        \begin{aligned}
            \sum_{i\in\cI(t)}
            \bigl(U_{i}(t)-\tilde{U}_{i}(t)\bigr)
            &\le \sum_{i\in\cI(t)}\biggl(
            \int_{\fR_{i}}
            \bigl[u(t,x)-\tilde{u}_{Q_{i}}(t,x)\bigr]\,dx
            + \int_{\fT_{i}^{L}}
            \bigl[u(t,x)-\tilde{u}_{\fS_{i}}(t,x)\bigr]\,dx\\
            &\qquad
            + \int_{\fT_{i}^{R}}
            \bigl[u(t,x)-\tilde{u}_{\fS_{i+1}}(t,x)\bigr]\,dx
            \biggr)
            + 2\sum_{i=1}^{n(t)+1}
            \int_{\fS_{i}}\tilde{u}_{\fS_{i}}(t,x)\,dx\,.
        \end{aligned}
    \end{equation}
    Since $\alpha<1$ and each term on the right-hand
    side is nonnegative, subadditivity of
    $x\mapsto x^{\alpha}$ and
    Lemma~\ref{lem:approx_U_to_tilde_U}\,\textup{(i)}--\textup{(iii)}
    with $p=\alpha$ give
    \begin{equation*}
        \dE\Bigl[\Bigl(\sum_{i\in\cI(t)}
        \bigl(U_{i}(t)-\tilde{U}_{i}(t)\bigr)\Bigr)^{\alpha}
        \Bigr]
        \le n(t)\,o\bigl(e^{-t^{6}}\bigr)
        + 2\sum_{i=1}^{n(t)+1}
        \dE\biggl[\Bigl(\int_{\fS_{i}}
        \tilde{u}_{\fS_{i}}(t,x)\,dx\Bigr)^{\alpha}\biggr]\,.
    \end{equation*}
    Dividing by $B_{\alpha}(t)^{\alpha}$ and using
\eqref{eq:window_contribution_negligible}
completes the proof of \textup{(i)}.

    We next prove \textup{(iii)}. For each $i\in\cI(t)$, set
    \begin{equation*}
        V_{i}^{L}(t) := \int_{\fT_{i}^{L}}
        \bigl[\tilde{u}_{\fS_{i}}(t,x)
        -\tilde{u}_{Q_{i}}(t,x)\bigr]\,dx\,,
        \qquad
        V_{i}^{R}(t) := \int_{\fT_{i}^{R}}
        \bigl[\tilde{u}_{\fS_{i+1}}(t,x)
        -\tilde{u}_{Q_{i}}(t,x)\bigr]\,dx\,.
    \end{equation*}
    With this notation, the decomposition
    \eqref{eq:U_tildeU_error_decomposition} becomes
    \begin{equation*}
        \begin{aligned}
            \sum_{i\in\cI(t)}
            \bigl(U_{i}(t)-\tilde{U}_{i}(t)\bigr)
            &= \sum_{i\in\cI(t)}\biggl(
            \int_{\fR_{i}}
            \bigl[u(t,x)-\tilde{u}_{Q_{i}}(t,x)\bigr]\,dx
            + \int_{\fT_{i}^{L}}
            \bigl[u(t,x)-\tilde{u}_{\fS_{i}}(t,x)\bigr]\,dx\\
            &\qquad
            + \int_{\fT_{i}^{R}}
            \bigl[u(t,x)-\tilde{u}_{\fS_{i+1}}(t,x)\bigr]\,dx
            \biggr)
            + \sum_{i\in\cI(t)}
            \bigl(V_{i}^{L}(t)+V_{i}^{R}(t)\bigr)\,.
        \end{aligned}
    \end{equation*}
    By Lemma~\ref{lem:approx_U_to_tilde_U}
\textup{(i)}--\textup{(iii)} with $p=1$,
the first sum on the right-hand side converges
to zero in $L^{1}$ after centering and division
by $B_{\alpha}(t)$. It therefore suffices to verify
    that
    \begin{equation}\label{eq:centered_window_errors_negligible}
        \frac{1}{B_{\alpha}(t)}\sum_{i\in\cI(t)}
        \bigl(V_{i}^{L}(t)-\dE[V_{i}^{L}(t)]\bigr)
        \xrightarrow{\ \dP\ } 0\,,
        \qquad
        \frac{1}{B_{\alpha}(t)}\sum_{i\in\cI(t)}
        \bigl(V_{i}^{R}(t)-\dE[V_{i}^{R}(t)]\bigr)
        \xrightarrow{\ \dP\ } 0\,.
    \end{equation}
    We prove the first convergence. The second is obtained in
    the same way. Each $V_{i}^{L}(t)$ is measurable with
    respect to the noise on $Q_{i}\cup\fS_{i}$. Since
    $|\fS_{i}|=2e^{t/100}$ and $|Q_{i}|=\ell(t)$, for all
    sufficiently large $t$,
    \begin{equation}\label{eq:parity_disjointness}
        \bigl(Q_{i}\cup\fS_{i}\bigr)\cap
        \bigl(Q_{j}\cup\fS_{j}\bigr)=\emptyset\,,
        \qquad |i-j|\ge2\,.
    \end{equation}
    We split the sum in
    \eqref{eq:centered_window_errors_negligible} into the sums
    over even and odd $i$. By \eqref{eq:parity_disjointness},
    each partial sum consists of independent random variables,
    and it suffices to treat each of them. By Markov's
    inequality with exponent $\alpha$, the von Bahr--Esseen
    inequality \cite[Theorem~2]{vBE65}, and the convexity of
    $x\mapsto x^{\alpha}$, for every $\ve>0$,
    \begin{equation*}
        \dP\biggl(\Bigl|\sum_{i}
        \bigl(V_{i}^{L}(t)-\dE[V_{i}^{L}(t)]\bigr)\Bigr|
        >\ve\,B_{\alpha}(t)\biggr)
        \le \frac{C}{\ve^{\alpha}B_{\alpha}(t)^{\alpha}}
        \sum_{i}\dE\bigl[|V_{i}^{L}(t)|^{\alpha}\bigr]\,.
    \end{equation*}
    Since $0\le\tilde{u}_{Q_{i}}\le u$ on $\fT_{i}^{L}$, we
    have
    \begin{equation*}
        |V_{i}^{L}(t)|
        \le\int_{\fT_{i}^{L}}\tilde{u}_{\fS_{i}}\,dx
        +\int_{\fT_{i}^{L}}u\,dx
        \le2\int_{\fT_{i}^{L}}\tilde{u}_{\fS_{i}}\,dx
        +\int_{\fT_{i}^{L}}
        \bigl[u-\tilde{u}_{\fS_{i}}\bigr]\,dx\,,
    \end{equation*}
    and hence
    \begin{equation*}
        \dE\bigl[|V_{i}^{L}(t)|^{\alpha}\bigr]
        \le C\,\dE\biggl[\Bigl(\int_{\fT_{i}^{L}}
        \tilde{u}_{\fS_{i}}(t,x)\,dx\Bigr)^{\alpha}\biggr]
        + C\,\dE\biggl[\Bigl(\int_{\fT_{i}^{L}}
        \bigl[u(t,x)-\tilde{u}_{\fS_{i}}(t,x)\bigr]\,dx
        \Bigr)^{\alpha}\biggr]\,,
    \end{equation*}
    where the second term is $o(e^{-t^{6}})$ by
    Lemma~\ref{lem:approx_U_to_tilde_U}\,\textup{(ii)}.
    Since $\tilde{u}_{\fS_{i}}\ge0$ and
    $\fT_{i}^{L}\subset\fS_{i}$, we may enlarge the domain of
    integration in the first term to $\fS_{i}$. Since the
    integrals over $\fS_{i}$ have the same law,
    \begin{equation*}
        \frac{1}{B_{\alpha}(t)^{\alpha}}
        \sum_{i}\dE\bigl[|V_{i}^{L}(t)|^{\alpha}\bigr]
        \le \frac{C\,n(t)}{B_{\alpha}(t)^{\alpha}}
        \dE\biggl[\Bigl(\int_{\fS_{1}}
        \tilde{u}_{\fS_{1}}(t,x)\,dx\Bigr)^{\alpha}\biggr]
        + o(1)
        \longrightarrow 0\,,
    \end{equation*}
    by \eqref{eq:window_contribution_negligible}. This
    proves \textup{(iii)}.

    Finally, we prove \textup{(ii)}. Since
    $Y_{0}(t)\ge\tilde{Y}_{0}(t)$, the event
    $\{\tilde{Y}_{0}(t)\le1\}$ is the disjoint union of
    $\{Y_{0}(t)\le1\}$ and
    $\{\tilde{Y}_{0}(t)\le1<Y_{0}(t)\}$, so
    \begin{equation*}
        \begin{aligned}
            \dE\bigl[Y_{0}(t)\,;\,Y_{0}(t)\le1\bigr]
            -\dE\bigl[\tilde{Y}_{0}(t)\,;\,
            \tilde{Y}_{0}(t)\le1\bigr]
            &= \dE\bigl[Y_{0}(t)-\tilde{Y}_{0}(t)\,;\,
            Y_{0}(t)\le1\bigr]\\
            &\qquad
            -\dE\bigl[\tilde{Y}_{0}(t)\,;\,
            \tilde{Y}_{0}(t)\le1<Y_{0}(t)\bigr]\,.
        \end{aligned}
    \end{equation*}
    Hence the left-hand side of
    \eqref{eq:stable_approximation_error_alpha_1} is bounded in
    absolute value by
    \begin{equation*}
        \sum_{i\in\cI(t)}\bigl(Y_{i}(t)-\tilde{Y}_{i}(t)\bigr)
        + n(t)\,\dE\bigl[Y_{0}(t)-\tilde{Y}_{0}(t)\bigr]
        + n(t)\,\dE\bigl[\tilde{Y}_{0}(t)\,;\,
        \tilde{Y}_{0}(t)\le1<Y_{0}(t)\bigr]\,.
    \end{equation*}
    Taking expectations in \eqref{eq:boundary_layer_window_split}
    and applying Lemma~\ref{lem:approx_U_to_tilde_U}\,\textup{(i)}--\textup{(iii)} with $p=1$
    and \eqref{eq:window_contribution_negligible},
    we obtain 
    \[ 
    n(t)\,\dE[Y_0(t)-\tilde{Y}_0(t)]\to0\,.
    \] Hence, the first two terms in the preceding bound converge to zero in probability, and it remains to bound the third.

    Fix $\eta\in(0,1)$. On
    $\{\tilde{Y}_{0}(t)\le1<Y_{0}(t)\}$, we have
    $\tilde{Y}_{0}(t)\le1$, and either
    $Y_{0}(t)-\tilde{Y}_{0}(t)>\eta$ or
    $\tilde{Y}_{0}(t)>1-\eta$. Hence
    \begin{equation*}
        \dE\bigl[\tilde{Y}_{0}(t)\,;\,
        \tilde{Y}_{0}(t)\le1<Y_{0}(t)\bigr]
        \le \dP\bigl(Y_{0}(t)-\tilde{Y}_{0}(t)>\eta\bigr)
        + \dP\bigl(1-\eta<\tilde{Y}_{0}(t)\le1\bigr)\,.
    \end{equation*}
    By Markov's inequality,
    $n(t)\,\dP\bigl(Y_{0}(t)-\tilde{Y}_{0}(t)>\eta\bigr)
    \le\eta^{-1}\,n(t)\,
    \dE\bigl[Y_{0}(t)-\tilde{Y}_{0}(t)\bigr]\to0$, and
    Proposition~\ref{prop:tilde_Y_positive_tail} gives
    \begin{equation*}
        \lim_{t\to\infty}
        n(t)\,\dP\bigl(1-\eta<\tilde{Y}_{0}(t)\le1\bigr)
        = \frac{1}{1-\eta}-1
        = \frac{\eta}{1-\eta}\,.
    \end{equation*}
    Hence
    \begin{equation*}
        \limsup_{t\to\infty}\,
        n(t)\,\dE\bigl[\tilde{Y}_{0}(t)\,;\,
        \tilde{Y}_{0}(t)\le1<Y_{0}(t)\bigr]
        \le \frac{\eta}{1-\eta}\,,
    \end{equation*}
    and letting $\eta\downarrow0$ proves \textup{(ii)}
    and completes the proof.
\end{proof}

We are now ready to prove Theorem~\ref{thm:PAM_Stable}.
\begin{proof}[Proof of Theorem~\ref{thm:PAM_Stable}]
    By stationarity, Fubini's theorem, \eqref{eq:stable_centering},
    and $n(t)\ell(t)=L(t)$,
    \begin{equation*}
        \frac{L(t)\,A(t)}{B_{\alpha}(t)}
        = \begin{cases}
            0\,, & \alpha\in(0,1)\,,\\[4pt]
            n(t)\,\dE\bigl[Y_{0}(t)\,;\,Y_{0}(t)\le1\bigr]\,,
            & \alpha=1\,,\\[4pt]
            n(t)\,\dE\bigl[Y_{0}(t)\bigr]\,,
            & \alpha\in(1,2)\,,
        \end{cases}
    \end{equation*}
    which is $\tilde{A}(t)$ of \eqref{eq:def_stable_centering}
    with $\tilde{Y}_{0}(t)$
    replaced by $Y_{0}(t)$. Since
    $\sum_{i\in\cI(t)}Y_{i}(t)=U(t)/B_{\alpha}(t)$,
    \begin{equation*}
        \frac{U(t)-L(t)\,A(t)}{B_{\alpha}(t)}
        - \Biggl(\sum_{i\in\cI(t)}\tilde{Y}_{i}(t)
        - \tilde{A}(t)\Biggr)
        = \sum_{i\in\cI(t)}
        \bigl(Y_{i}(t)-\tilde{Y}_{i}(t)\bigr)
        - \Biggl(\frac{L(t)\,A(t)}{B_{\alpha}(t)}
        - \tilde{A}(t)\Biggr)\,,
    \end{equation*}
    which is the left-hand side of the corresponding case of
    Lemma~\ref{lem:Stable_difference}, and hence converges to
    $0$ in probability. By
    Theorem~\ref{thm:PAM_Stable_Dirichlet} and Slutsky's
    theorem, \eqref{eq:stable_convergence} follows.
\end{proof}

\section*{Acknowledgement} JY was supported by a KIAS Individual Grant (HP090401) and by the National Science Foundation under Grant No. DMS-2424139 while in residence at the Simons Laufer Mathematical Sciences Institute in Berkeley, California, during the Fall 2025 semester. KK and UK were supported in part by the National Research Foundation of Korea (RS-2026-25479024).

\appendix
\section{Auxiliary estimates for
Theorem~\ref{thm:sharp_lower_tail_lowest_eigenvalue}}
\label{apx:Auxiliary_estimates}
In this appendix, we prove several auxiliary estimates
used in the proof of Theorem~\ref{thm:sharp_lower_tail_lowest_eigenvalue}.
Recall from Section~\ref{sec:Anderson_Hamiltonian_finite_interval}
the definitions of $\delta(a)$, $x_{\pm}$, $I(a)$,
$T_{y}$, and $m(a)$ in \eqref{eq:def_delta_x_pm},
\eqref{eq:def_interval_I}, \eqref{eq:def_hitting_time}, and
\eqref{eq:first_explosion_mean}, the hitting time
$H=T_{x_{-}}\wedge T_{x_{+}}$, the first explosion time
$\zeta_{a}^{(1)}$, and the notation $\dP_{z}$ and $\dE_{z}$.
Throughout, the diffusion $X_{a}$
denotes the solution of \eqref{eq:Riccati_SDE}, and we write
\begin{equation*}
    \rho := \inf\bigl\{t\ge0\,:\,X_{a}(t)\notin I(a)\bigr\}\,.
\end{equation*}
Under $\dP_{x_{+}}$, $\rho$ coincides with $\rho_{1}$ of
\eqref{eq:def_rho_sigma}.
We begin with the first two moments of $\rho$ for
$X_{a}$ started from $x_{+}$.

\begin{lemma}\label{lem:moment_bounds_exit_time_rho}
    For every $a>1$, the following bounds hold.
    \begin{enumerate}
        \item[\textup{(i)}]
        $\dE_{x_{+}}[\rho] \le \dfrac{1}{2\sqrt{a}}$\,.
        \item[\textup{(ii)}]
        $\dE_{x_{+}}\bigl[\rho^{2}\bigr]
        \le \dfrac{1}{4a}+\dfrac{1}{2a(\log a)^{2}}$\,.
    \end{enumerate}
\end{lemma}

\begin{proof}
    For $s<\rho$ we have $X_{a}(s)\ge\sqrt{a}+\delta(a)/2$, and
    hence $a-X_{a}(s)^{2}\le-\sqrt{a}\,\delta(a)$.
    Integrating the equation \eqref{eq:Riccati_SDE} for
    $X_{a}$ up to $t\wedge\rho$, and using this drift bound
    together with $X_{a}(t\wedge\rho)\ge x_{+}-\delta(a)/2$, we
    obtain
    \begin{equation}\label{eq:apx_rho_pathwise_bound}
        t\wedge\rho
        \le \frac{1}{\sqrt{a}\,\delta(a)}
        \Bigl(\frac{\delta(a)}{2}+B(t\wedge\rho)\Bigr)\,.
    \end{equation}
    Since $t\wedge\rho$ is a bounded stopping time, optional
    stopping gives $\dE_{x_{+}}[B(t\wedge\rho)]=0$ and
    $\dE_{x_{+}}\bigl[B(t\wedge\rho)^{2}\bigr]
    =\dE_{x_{+}}[t\wedge\rho]$. Taking expectations in \eqref{eq:apx_rho_pathwise_bound} and
    letting $t\to\infty$, we obtain \textup{(i)} by monotone
    convergence.
    For \textup{(ii)}, squaring \eqref{eq:apx_rho_pathwise_bound}
    and taking expectations, we obtain
    \begin{equation*}
        \dE_{x_{+}}\bigl[(t\wedge\rho)^{2}\bigr]
        \le \frac{1}{a\,\delta(a)^{2}}
        \Bigl(\frac{\delta(a)^{2}}{4}
        +\dE_{x_{+}}[t\wedge\rho]\Bigr)
        \le \frac{1}{4a}
        +\frac{1}{a\,\delta(a)^{2}}\cdot\frac{1}{2\sqrt{a}}
        = \frac{1}{4a}+\frac{1}{2a(\log a)^{2}}\,.
    \end{equation*}
    Letting $t\to\infty$, we obtain \textup{(ii)} by monotone
    convergence.
\end{proof}

The next lemma bounds moments of $H$ for $X_{a}$ started from
either endpoint of $I(a)$.

\begin{lemma}\label{lem:H_moment_bounds}
    There exist $C_{1},C_{2}>0$ such that, for all sufficiently
    large $a$, the following bounds hold.
    \begin{enumerate}
        \item[\textup{(i)}]
        $\dE_{x_{+}+\delta(a)/2}\bigl[H^{2}\bigr]
        \le \dfrac{1}{16a}+\dfrac{1}{16a(\log a)^{2}}$\,.
        \item[\textup{(ii)}]
        $\dE_{x_{+}-\delta(a)/2}\bigl[H^{2}\bigr]
        \le C_{1}\exp\bigl(4(\log a)^{2}\bigr)\,
        a^{1/2}(\log a)^{2}$\,.
        \item[\textup{(iii)}]
        $\dE_{x_{+}-\delta(a)/2}\bigl[H\bigr]
        \ge \dfrac{\delta(a)^{2}}{8}$\,.
        \item[\textup{(iv)}]
        $\dE_{x_{+}-\delta(a)/2}\bigl[H\mid T_{x_{-}}<T_{x_{+}}\bigr]
        \le C_{2}\exp\bigl(2(\log a)^{2}\bigr)\,a^{1/4}\log a$\,.
    \end{enumerate}
\end{lemma}

\begin{proof}
    We begin with \textup{(i)}, which follows as in the proof of
    Lemma~\ref{lem:moment_bounds_exit_time_rho}. For $t<H$ the
    diffusion started from $x_{+}+\delta(a)/2$ stays above
    $x_{+}$, so the drift is bounded above by
    $-2\sqrt{a}\,\delta(a)$, and integrating
    \eqref{eq:Riccati_SDE} up to $t\wedge H$ gives
    \begin{equation*}
        t\wedge H
        \le \frac{1}{2\sqrt{a}\,\delta(a)}
        \Bigl(\frac{\delta(a)}{2}+B(t\wedge H)\Bigr)\,.
    \end{equation*}
    The same optional stopping argument gives
    $\dE_{x_{+}+\delta(a)/2}[H]\le\frac{1}{4\sqrt{a}}$ and then
    \textup{(i)}.
    We now prove \textup{(ii)} in two steps. We first bound
    $\dE_{x}[H]$ uniformly in $x\in[x_{-},x_{+}]$, and then
    pass to the second moment. Define
    \begin{equation*}
        \cT(x) := \dE_{x}[H]\,,\qquad x\in[x_{-},\,x_{+}]\,.
    \end{equation*}
    By Dynkin's formula applied up to $H$, the
    function $\cT$ satisfies
    \begin{equation*}
        \frac{1}{2}\cT''(x)+(a-x^{2})\,\cT'(x) = -1\,,
        \qquad x\in(x_{-},x_{+})\,,
    \end{equation*}
    with boundary conditions $\cT(x_{-})=\cT(x_{+})=0$. Recall the
    potential $V$ from \eqref{eq:Riccati_potential}, and introduce
    the scale function $q$ and the speed density $w$ (see, e.g.,
    \cite[Section~5.5]{KS91}),
    \begin{equation*}
        q(x) := \int_{x_{-}}^{x}e^{2V(y)}\,dy\,,
        \qquad
        w(x) := 2e^{-2V(x)}\,.
    \end{equation*}
    Then $q(x_{-})=0$, and we can write $\cT$ as
    \begin{equation}\label{eq:apx_Green_representation}
        \cT(x) = \int_{x_{-}}^{x_{+}}G(x,y)\,w(y)\,dy\,,
        \qquad
        G(x,y) := \frac{q(x\wedge y)
        \bigl(q(x_{+})-q(x\vee y)\bigr)}{q(x_{+})}\,.
    \end{equation}
    Since $q$ is increasing, $G(x,y)\le q(x_{+})-q(y)$ for all
    $y\in[x_{-},x_{+}]$. Moreover,
    \begin{equation}\label{eq:apx_kernel_identity}
        \bigl(q(x_{+})-q(y)\bigr)e^{-2V(y)}
        =\int_{y}^{x_{+}}
        \exp\Bigl(2\bigl(V(z)-V(y)\bigr)\Bigr)\,dz\,.
    \end{equation}
    Hence
    \begin{equation*}
        \sup_{x\in[x_{-},\,x_{+}]}\cT(x)
        \le 2\,\cD(a)\,,
    \end{equation*}
    where
    \begin{equation*}
        \cD(a) := \int_{x_{-}}^{x_{+}}\!\!\int_{y}^{x_{+}}
        \exp\Bigl(2\bigl(V(z)-V(y)\bigr)\Bigr)\,dz\,dy\,.
    \end{equation*}
    Note that
    \begin{equation*}
        V(z)-V(y) \le
        \begin{cases}
            V(-\sqrt{a}\,)-V(x_{-})\,,
            &y\in[x_{-},-\sqrt{a}\,]\,,\ z\in[y,x_{+}]\,,\\
            0\,,
            &y\in[-\sqrt{a},\sqrt{a}\,]\,,\ z\in[y,\sqrt{a}\,]\,,\\
            V(x_{+})-V(\sqrt{a}\,)\,,
            &y\in[-\sqrt{a},\sqrt{a}\,]\,,\ z\in[\sqrt{a},x_{+}]\,,\\
            V(x_{+})-V(\sqrt{a}\,)\,,
            &y\in[\sqrt{a},x_{+}]\,,\ z\in[y,x_{+}]\,,
        \end{cases}
    \end{equation*}
    and that, since $V$ is odd,
    \begin{equation*}
        V(-\sqrt{a}\,)-V(x_{-})
        = V(x_{+})-V(\sqrt{a}\,)
        = \sqrt{a}\,\delta(a)^{2}+\frac{\delta(a)^{3}}{3}\,.
    \end{equation*}
    Since the region in the second case has area at most $2a$,
    while the remaining regions have total area at most
    $2\delta(a)\bigl(\sqrt{a}+\delta(a)\bigr)
    +2\sqrt{a}\,\delta(a)+\delta(a)^{2}$, it follows that
    \begin{equation}\label{eq:apx_double_integral_bound}
        \begin{aligned}
            \cD(a)
            &\le \exp\Bigl(2\sqrt{a}\,\delta(a)^{2}
            +\frac{2\delta(a)^{3}}{3}\Bigr)
            \Bigl(2\delta(a)\bigl(\sqrt{a}+\delta(a)\bigr)
            +2\sqrt{a}\,\delta(a)+\delta(a)^{2}\Bigr)+2a\\
            &\le C\exp\bigl(2(\log a)^{2}\bigr)\,a^{1/4}\log a
        \end{aligned}
    \end{equation}
    for some $C>0$ and all sufficiently large $a$.
    We now turn to the second moment of $H$.
    Set
    \begin{equation}\label{eq:def_C_sup_H}
        \cC(a) := 2\sup_{y\in[x_{-},\,x_{+}]}\dE_{y}[H]
        \;\le\;4\,\cD(a)\,.
    \end{equation}
    For every $y\in[x_{-},\,x_{+}]$, Markov's inequality yields
    \begin{equation*}
        \dP_{y}\bigl(H>\cC(a)\bigr) \le \frac{1}{2}\,.
    \end{equation*}
    We claim that
    \begin{equation}\label{eq:H_geometric_tail}
        \sup_{y\in[x_{-},\,x_{+}]}\dP_{y}\bigl(H>n\cC(a)\bigr)
        \le 2^{-n}\,,\qquad n\ge1\,.
    \end{equation}
    Indeed, on the event $\{H>\cC(a)\}$ we have
    $X_{a}(\cC(a))\in(x_{-},x_{+})$. Hence, for every
    $y\in[x_{-},x_{+}]$, the Markov property at time $\cC(a)$
    gives
    \begin{equation*}
        \begin{aligned}
            \dP_{y}\bigl(H>n\cC(a)\bigr)
            &\le \dE_{y}\Bigl[\1_{\{H>\cC(a)\}}\,
            \dP_{X_{a}(\cC(a))}\bigl(H>(n-1)\cC(a)\bigr)\Bigr]\\
            &\le \frac{1}{2}\,
            \sup_{z\in[x_{-},\,x_{+}]}
            \dP_{z}\bigl(H>(n-1)\cC(a)\bigr)\,,
        \end{aligned}
    \end{equation*}
    and \eqref{eq:H_geometric_tail} follows by induction. 
    Since $\dE_{y}[H^{2}]=\int_{0}^{\infty}2s\,\dP_{y}(H>s)\,ds$,
    the bound \eqref{eq:H_geometric_tail} gives
    \begin{equation}\label{eq:H_second_moment_bound_by_C}
        \sup_{y\in[x_{-},\,x_{+}]}\dE_{y}\bigl[H^{2}\bigr]
        \le \sum_{n\ge0}2^{-n}
        \int_{n\cC(a)}^{(n+1)\cC(a)}2s\,ds
        = \cC(a)^{2}\sum_{n\ge0}(2n+1)\,2^{-n}
        = 6\,\cC(a)^{2}\,.
    \end{equation}
    Combining \eqref{eq:H_second_moment_bound_by_C} with
    \eqref{eq:def_C_sup_H} and
    \eqref{eq:apx_double_integral_bound} yields
    \begin{equation*}
        \dE_{x_{+}-\delta(a)/2}\bigl[H^{2}\bigr]
        \le 6\,\cC(a)^{2}
        \le 96\,\cD(a)^{2}
        \le C_{1}\exp\bigl(4(\log a)^{2}\bigr)\,
        a^{1/2}(\log a)^{2}\,,
    \end{equation*}
    which proves \textup{(ii)}.
    We now prove \textup{(iii)}. Since $G\ge0$, restricting the integral
    in \eqref{eq:apx_Green_representation} to
    $y\ge x_{+}-\delta(a)/2$ and using
    \eqref{eq:apx_kernel_identity}, we obtain
    \begin{equation*}
        \cT\bigl(x_{+}-\delta(a)/2\bigr)
        \ge \frac{q\bigl(x_{+}-\delta(a)/2\bigr)}{q(x_{+})}\,
        2\int_{x_{+}-\delta(a)/2}^{x_{+}}\!\!\int_{y}^{x_{+}}
        \exp\Bigl(2\bigl(V(z)-V(y)\bigr)\Bigr)\,dz\,dy\,.
    \end{equation*}
    We bound the ratio and the double integral separately,
    taking $a$ sufficiently large throughout. We first show that
    the ratio is at least $1/2$. Since $V$ is increasing on
    $[x_{+}-\delta(a)/2,\,x_{+}]$ and $V(x_{+})<0$, we have
    $q(x_{+})-q(x_{+}-\delta(a)/2)\le\delta(a)/2$. On the other
    hand, since $V$ is increasing on $[x_{-},-\sqrt{a}\,]$
    and odd,
    \begin{equation*}
        q\bigl(x_{+}-\delta(a)/2\bigr)
        \ge \int_{x_{-}}^{-\sqrt{a}}e^{2V(z)}\,dz
        \ge \delta(a)\,e^{-2V(x_{+})}
        \ge \delta(a)\,e^{a^{3/2}}\,,
    \end{equation*}
    and the claim follows. For the double integral, $V$
    is increasing on $[\sqrt{a},x_{+}]$, so the integrand is at
    least $1$ and the double integral is at least
    $\delta(a)^{2}/8$. Combining the two bounds yields
    $\dE_{x_{+}-\delta(a)/2}[H]\ge\delta(a)^{2}/8$.
    Finally, for \textup{(iv)}, define
    \begin{equation*}
        h(x) := \dP_{x}\bigl(T_{x_{-}}<T_{x_{+}}\bigr)\,,
        \qquad
        \cT_{-}(x) := \dE_{x}\bigl[H\,\1_{\{T_{x_{-}}<T_{x_{+}}\}}\bigr]\,.
    \end{equation*}
    Then $\dE_{x}\bigl[H\mid T_{x_{-}}<T_{x_{+}}\bigr]=\cT_{-}(x)/h(x)$ for
    $x<x_{+}$. The function $h$ is given by the standard formula
    \begin{equation*}
        h(x) = \frac{q(x_{+})-q(x)}{q(x_{+})}\,.
    \end{equation*}
    By the Markov property, on the event $\{t<H\}$ we have
    $\dE_{x}\bigl[\1_{\{T_{x_{-}}<T_{x_{+}}\}}\mid\cF_{t}\bigr]
    =h\bigl(X_{a}(t)\bigr)$, so that
    \begin{equation*}
        \cT_{-}(x) = \dE_{x}\biggl[\int_{0}^{H}
        h\bigl(X_{a}(t)\bigr)\,dt\biggr]\,.
    \end{equation*}
    By Dynkin's formula applied up to $H$, the
    function $\cT_{-}$ solves
    \begin{equation*}
        \frac{1}{2}\cT_{-}''(x)+(a-x^{2})\,\cT_{-}'(x) = -h(x)\,,
        \qquad x\in(x_{-},x_{+})\,,
    \end{equation*}
    with boundary conditions $\cT_{-}(x_{-})=\cT_{-}(x_{+})=0$, and hence
    \begin{equation*}
        \cT_{-}(x) = \int_{x_{-}}^{x_{+}}G(x,y)\,h(y)\,w(y)\,dy\,.
    \end{equation*}
    We split the integral at $y=x$. For $y\le x$, the definition
    of $G$ gives $G(x,y)=h(x)\,q(y)$, and since
    $q(y)\,h(y)\le q(x_{+})-q(y)$,
    \begin{equation*}
        \frac{1}{h(x)}\int_{x_{-}}^{x}G(x,y)\,h(y)\,w(y)\,dy
        = \int_{x_{-}}^{x}q(y)\,h(y)\,w(y)\,dy
        \le \int_{x_{-}}^{x_{+}}
        \bigl(q(x_{+})-q(y)\bigr)\,w(y)\,dy
        = 2\,\cD(a)\,.
    \end{equation*}
    For $y\ge x$, since $h$ is nonincreasing we have
    $h(y)\le h(x)$, so
    \begin{equation*}
        \frac{1}{h(x)}\int_{x}^{x_{+}}G(x,y)\,h(y)\,w(y)\,dy
        \le \int_{x}^{x_{+}}G(x,y)\,w(y)\,dy
        \le \int_{x_{-}}^{x_{+}}
        \bigl(q(x_{+})-q(y)\bigr)\,w(y)\,dy
        = 2\,\cD(a)\,.
    \end{equation*}
    Combining the two bounds with \eqref{eq:apx_double_integral_bound} yields
    \begin{equation*}
        \dE_{x_{+}-\delta(a)/2}
        \bigl[H\mid T_{x_{-}}<T_{x_{+}}\bigr]
        \le 4\,\cD(a)
        \le C_{2}\exp\bigl(2(\log a)^{2}\bigr)\,a^{1/4}\log a\,,
    \end{equation*}
    which proves \textup{(iv)} and completes the proof of the lemma.
\end{proof}

The last lemma of this appendix shows that the mean
explosion time of $X_{a}$ started from $x_{-}$ is
negligible compared with $m(a)$. This is used in the
proof of Lemma~\ref{lem:renewal_mean_asymptotics}.

\begin{lemma}\label{lem:mean_explosion_from_x_minus}
    As $a\to\infty$,
    \begin{equation*}
        \dE_{x_{-}}\bigl[\zeta_{a}^{(1)}\bigr] \ll m(a)\,.
    \end{equation*}
\end{lemma}

\begin{proof}
    By \cite[Eq.~(3.7)]{AD14}, as in the proof of
    Lemma~\ref{lem:renewal_mean_asymptotics},
    \begin{equation*}
        \dE_{x_{-}}\bigl[\zeta_{a}^{(1)}\bigr]
        = 2\int_{-\infty}^{x_{-}}\!dx\int_{x}^{\infty}\!du\,
        \exp\Bigl(2a(u-x)
        +\frac{2}{3}\bigl(x^{3}-u^{3}\bigr)\Bigr)\,.
    \end{equation*}
    Substituting $v:=u-x\ge0$ and using the identity
    \begin{equation*}
        2a(u-x)+\frac{2}{3}\bigl(x^{3}-u^{3}\bigr)
        = \psi_{a}(v)-2v\Bigl(x+\frac{v}{2}\Bigr)^{2}\,,
        \qquad
        \psi_{a}(v) := 2av-\frac{v^{3}}{6}\,,
    \end{equation*}
    we perform the Gaussian integral in $x$ and obtain
    \begin{equation*}
        \dE_{x_{-}}\bigl[\zeta_{a}^{(1)}\bigr]
        = \sqrt{\frac{\pi}{2}}
        \int_{0}^{\infty}\frac{dv}{\sqrt{v}}\,e^{\psi_{a}(v)}
        \Bigl(1+\operatorname{erf}\Bigl(\sqrt{2v}\,
        \Bigl(x_{-}+\frac{v}{2}\Bigr)\Bigr)\Bigr)\,,
    \end{equation*}
    where
    $\operatorname{erf}(z):=\frac{2}{\sqrt{\pi}}\int_{0}^{z}e^{-s^{2}}\,ds$.
    Note that the same computation with $x_{-}$ replaced by
    $+\infty$ recovers \eqref{eq:first_explosion_mean_formula},
    that is,
    \begin{equation*}
        m(a) = \sqrt{2\pi}\int_{0}^{\infty}
        \frac{dv}{\sqrt{v}}\,e^{\psi_{a}(v)}\,.
    \end{equation*}
    Let $v_{0}:=2\sqrt{a}$ be the maximizer of $\psi_{a}$, and set
    $O_{a}:=\{v\,:\,|v-v_{0}|\le\delta(a)\}$. Note that
    $O_{a}\subset(\sqrt{a},4\sqrt{a}\,)$ for all sufficiently large $a$. 
    Near $v_{0}$ the factor $1+\operatorname{erf}$ is
    exponentially small, while away from $v_{0}$ the integral of
    $e^{\psi_{a}}$ is negligible compared with $m(a)$.
    We split the domain of integration into
    $(0,\sqrt{a}\,)$, $(4\sqrt{a},\infty)$,
    $(\sqrt{a},4\sqrt{a}\,)\cap O_{a}^{c}$, and $O_{a}$, bounding
    $1+\operatorname{erf}\le2$ on the first three regions.
    On $(0,\sqrt{a}\,)$, we have
    $\psi_{a}(v)\le\frac{11}{6}a^{3/2}$, so this part is at most
    \begin{equation*}
        \sqrt{2\pi}\int_{0}^{\sqrt{a}}
        \exp\Bigl(\frac{11}{6}a^{3/2}\Bigr)\frac{dv}{\sqrt{v}}
        = \sqrt{8\pi}\,a^{1/4}
        \exp\Bigl(\frac{11}{6}a^{3/2}\Bigr)\,.
    \end{equation*}
    On $(4\sqrt{a},\infty)$, we have
    $\psi_{a}(v)\le-\frac{2a}{3}v$ and $v^{-1/2}\le1$, so this
    part is at most
    \begin{equation*}
        \sqrt{2\pi}\int_{4\sqrt{a}}^{\infty}
        \exp\Bigl(-\frac{2a}{3}v\Bigr)dv
        \le 3\sqrt{\frac{\pi}{2}}\,
        \exp\Bigl(-\frac{8}{3}a^{3/2}\Bigr)\,.
    \end{equation*}
    On $(\sqrt{a},4\sqrt{a}\,)\cap O_{a}^{c}$, note that
    $\psi_{a}'(v_{0})=0$ and $\psi_{a}''(v)=-v\le-\sqrt{a}$ on
    $(\sqrt{a},4\sqrt{a}\,)$. By Taylor's theorem, for $v$ in
    this range with $|v-v_{0}|>\delta(a)$,
    \begin{equation*}
        \psi_{a}(v)
        \le \psi_{a}(v_{0})-\frac{\sqrt{a}\,\delta(a)^{2}}{2}
        = \frac{8}{3}a^{3/2}-\frac{(\log a)^{2}}{2}\,,
    \end{equation*}
    so this part is at most
    \begin{equation*}
        \sqrt{2\pi}\,
        \exp\Bigl(\frac{8}{3}a^{3/2}
        -\frac{(\log a)^{2}}{2}\Bigr)
        \int_{\sqrt{a}}^{4\sqrt{a}}\frac{dv}{\sqrt{v}}
        \le \sqrt{8\pi}\,a^{1/4}
        \exp\Bigl(\frac{8}{3}a^{3/2}
        -\frac{(\log a)^{2}}{2}\Bigr)\,.
    \end{equation*}
    Finally, on $O_{a}$ we keep the error function.
    For $|v-v_{0}|\le\delta(a)$ we have
    $x_{-}+v/2\le x_{-}+v_{0}/2+\delta(a)/2=-\delta(a)/2$ and
    $\sqrt{2v}\ge a^{1/4}$ for all sufficiently large
    $a$. Since $\operatorname{erf}$ is increasing,
    \begin{equation*}
        1+\operatorname{erf}\Bigl(\sqrt{2v}\,
        \Bigl(x_{-}+\frac{v}{2}\Bigr)\Bigr)
        \le 1+\operatorname{erf}\Bigl(-\frac{\log a}{2}\Bigr)
        \le \frac{2}{\sqrt{\pi}\,\log a}\,
        \exp\Bigl(-\frac{(\log a)^{2}}{4}\Bigr)\,,
    \end{equation*}
    where the last step is the standard Gaussian tail bound
    $1+\operatorname{erf}(-z)\le e^{-z^{2}}/(\sqrt{\pi}\,z)$.
    Hence this part is at most
    \begin{equation*}
        \begin{aligned}
            \sqrt{\frac{\pi}{2}}\cdot
            \frac{2}{\sqrt{\pi}\,\log a}\,
            \exp\Bigl(-\frac{(\log a)^{2}}{4}\Bigr)
            \int_{O_{a}}\frac{dv}{\sqrt{v}}\,e^{\psi_{a}(v)}
            &= \frac{1}{\sqrt{\pi}\,\log a}\,
            \exp\Bigl(-\frac{(\log a)^{2}}{4}\Bigr)\,
            \sqrt{2\pi}\int_{O_{a}}
            \frac{dv}{\sqrt{v}}\,e^{\psi_{a}(v)}\\
            &\le \frac{1}{\sqrt{\pi}\,\log a}\,
            \exp\Bigl(-\frac{(\log a)^{2}}{4}\Bigr)\,m(a)\,.
        \end{aligned}
    \end{equation*}
    Since $m(a)\sim(\pi/\sqrt{a}\,)
    \exp\bigl(\frac{8}{3}a^{3/2}\bigr)$ by
    \eqref{eq:first_explosion_mean_asymptotics}, each of the four
    bounds is $o\bigl(m(a)\bigr)$, which proves the lemma.
\end{proof}

%============================================================================================================
%============================================================================================================
% \printbibliography
\bibliographystyle{alpha}   
\bibliography{refs}

\end{document}